\documentclass[a4paper,11pt]{amsart}
\usepackage[foot]{amsaddr}
\usepackage{etex}
\usepackage[a4paper,left=25mm,right=25mm,top=30mm,bottom=30mm,marginpar=25mm]{geometry}
\usepackage[utf8x]{inputenc}
\usepackage[english]{babel}
\usepackage{amsmath}
\usepackage{amssymb}
\usepackage{amsthm}
\usepackage{enumerate}
\usepackage[shortlabels]{enumitem}
\usepackage{enumitem}
\usepackage{empheq}
\usepackage{fancybox}
\usepackage{xcolor}
\usepackage{mathrsfs}
\usepackage{slashed}
\usepackage{tikz}
\usepackage{harpoon}
\usepackage{setspace}
\usepackage{color,soul}
\usepackage{fancybox}
\usepackage{esint}
\usepackage{bm}
\usepackage{subfigure}
\usepackage[all]{xy}
\usepackage[colorlinks, citecolor=hanblue,linkcolor=red]{hyperref}
\definecolor{hanblue}{rgb}{0.27, 0.42, 0.81}
\usepackage{todonotes}

\usepackage{graphicx}

\DeclareMathOperator{\D}{\mathrm{d}}

\DeclareMathOperator{\ext}{Ext}

\DeclareMathOperator{\supp}{supp}

\newcommand{\AC}{{\rm AC}}

\renewcommand{\div}{{\rm div}\,}

\newcommand{\R}{\mathbb{R}}

\newcommand{\Q}{\mathbb{Q}}
\newcommand{\N}{\mathbb{N}}
\newcommand{\M}{\mathcal{M}}
\newcommand{\Distr}{\mathcal{D}}

\newcommand{\Ha}{\mathcal{H}}

\newcommand{\car}{\mathscr{S}}
\newcommand{\cone}{\mathscr{C}}
\newcommand{\Bb}{\mathscr{B}}

\newcommand{\Om}{\Omega}

\newcommand{\de}{\partial}
\newcommand{\f}{\varphi}
\newcommand{\e}{\varepsilon}
\DeclareRobustCommand{\rchi}{{\mathpalette\irchi\relax}}
\newcommand{\irchi}[2]{\raisebox{\depth}{$#1\chi$}}
\newcommand{\nor}[1]{\left\| #1 \right\|} %

\newcommand{\zak}{%
  \mathbin{\vrule height 1.6ex depth 0pt width
0.13ex\vrule height 0.13ex depth 0pt width 1.3ex}
}    %

\newtheorem{theorem}{Theorem}[section]

\newtheorem{definition}[theorem]{Definition}

\newtheorem{proposition}[theorem]{Proposition}

\newtheorem{lemma}[theorem]{Lemma}
\newtheorem{remark}[theorem]{Remark}

\DeclareMathOperator*{\argmin}{arg\,min}

  \let\div\relax
  \DeclareMathOperator*{\div}{div}

	\newcommand{\norm}[1]{\left\lVert#1\right\rVert}

	\DeclareMathOperator{\dist}{dist}
\newcommand{\curves}{C_{\rm w} ([0,1];\M(\Om))}
\newcommand{\pcurves}{C_{\rm w} ([0,1];\M^+(\Om)) }
\newcommand{\pcurvesV}{C_{\rm w} ([0,1];\M^+(V))}
\newcommand{\weakstar}{\stackrel{*}{\rightharpoonup}}  %
         
\DeclareMathOperator*{\Lip}{Lip}

\title[A superposition principle for the inhomogeneous continuity equation]{A superposition principle for the inhomogeneous continuity equation with Hellinger-Kantorovich-regular coefficients}

 \author[K. Bredies]{Kristian Bredies}
\author[M. Carioni]{Marcello Carioni} 
\author[S. Fanzon]{Silvio Fanzon}

 \address[Kristian Bredies, Silvio Fanzon]{University of Graz, Institute of Mathematics and Scientific \mbox{Computing}, Heinrichstra\ss e 36, 8010 Graz, Austria}

 \address[Marcello Carioni]{University of Cambridge, Department of Applied Mathematics and Theoretical Physics, Wilberforce Road, Cambridge
CB3 0WA, UK}

\email[Kristian Bredies]{Kristian.Bredies@uni-graz.at}
\email[Marcello Carioni]{mc2250@maths.cam.ac.uk}
\email[Silvio Fanzon]{Silvio.Fanzon@uni-graz.at}

\begin{document}

\begin{abstract}
\small{ %
 We study measure-valued solutions of the inhomogeneous continuity equation
  $\partial_t \rho_t + \div(v\rho_t) = g \rho_t$ where the coefficients $v$ and $g$ are of low regularity. %
  A new superposition principle is proven for positive measure solutions and coefficients for which %
  the recently-introduced dynamic Hellinger-Kantorovich energy is finite. %
  This principle gives a decomposition of
  the solution into curves $t \mapsto h(t)\delta_{\gamma(t)}$ that satisfy the characteristic system $\dot \gamma(t) = v(t, \gamma(t))$, $\dot h(t) = g(t, \gamma(t)) h(t)$ in an appropriate sense.
  In particular, it provides a generalization of existing superposition
  principles to the low-regularity case of $g$ %
  where characteristics are not unique with respect to $h$. Two applications of this principle are presented. First, uniqueness of minimal total-variation solutions for the inhomogeneous continuity equation is obtained if characteristics are unique up to their possible vanishing time. Second, the extremal points of dynamic Hellinger-Kantorovich-type regularizers are characterized. Such regularizers arise, e.g., in the context of dynamic inverse problems and dynamic optimal transport.


}
\end{abstract} 
\maketitle
%\tableofcontents

\section{Introduction}

The main objective of this paper is to present a new superposition principle for positive measure solutions to the linear inhomogeneous continuity equation, assuming natural regularity on the velocity field and on the source term. Such assumptions are substantially weaker than what is currently available in the literature, as we will discuss below.
To be more precise, given $\Om \subset \R^d$ the closure of an open bounded domain, we consider narrowly continuous curves of positive measures $t \mapsto \rho_t$ in $\M^+(\Om)$ solving
\begin{equation}\label{eq:intro_cont}
	\de_t \rho_t + \div(v\rho_t ) =g \rho_t   \,\,\,\, \text{ in } \,\,\,\, (0,1) \times \Om 
\end{equation}
in the sense of distributions, where $v \colon (0,1) \times \Om \to \R^d$ is a velocity field satisfying no flux boundary conditions on $\de\Om$ and $g \colon (0,1) \times \Om \to \R$ is a source term encoding the inhomogeneity of the equation. 
We assume that the coefficients $v$ and $g$ are Hellinger-Kantorovich-regular, namely, they are Borel measurable and satisfy the bound  
\begin{equation}\label{eq:quadraticbound}
\int_0^1 \int_{\Om} |v(t,x)|^2 + |g(t,x)|^2 \, d\rho_t (x) \, dt< \infty \,.
\end{equation}
In the following we will clarify the role of \eqref{eq:quadraticbound} in connection to recent advancements in the theory of Unbalanced Optimal Transport. Our task is to provide a superposition principle for \eqref{eq:intro_cont} that allows to represent any positive solution $t \mapsto \rho_t$ as a superposition of elementary solutions, that is, %
curves of measures of the form $t \mapsto h(t)\delta_{\gamma(t)}$, where the trajectories $\gamma \colon [0,1] \to \Om$ and the weights $h \colon [0,1] \to [0,\infty)$ solve, in an appropriate sense, the system of characteristics for \eqref{eq:intro_cont}:
\begin{equation}\label{eq:int_ODEnonhom}
(i) \ \ \dot{\gamma}(t)= v(t,\gamma(t)) \qquad (ii)\ \  \dot{h}(t)= g(t,\gamma(t))h(t)\quad \text{in} \quad (0,1)\,.
\end{equation}
Notice that $(i)$ describes all possible elementary trajectories which follow the flow given by $v$, while $(ii)$ encodes the lack of mass preservation for solutions to \eqref{eq:intro_cont}, due to the inhomogeneity. The precise statement of such superposition principle is given in Theorem \ref{thm:intro_main} below. Subsequently we provide two applications of the superposition principle for \eqref{eq:intro_cont}. First we prove uniqueness for minimal norm solutions to \eqref{eq:intro_cont} under the assumption of uniqueness for solutions to \eqref{eq:int_ODEnonhom} up to their possible vanishing time (see Theorem \ref{thm:uniquenessintrro}); Second, we characterize extremal points of regularizers closely related to the energy at \eqref{eq:quadraticbound}, and apply such result to sparsity for dynamic inverse problems regularized via unbalanced optimal transport (see Theorem \ref{thm:intro_ext}).  

Concerning relevant literature, we mention that the superposition principle for narrowly continuous curves of probability measures $t \mapsto \rho_t$ solving the homogeneous continuity equation
\begin{equation}\label{eq:intro_cont_hom}
	\de_t \rho_t + \div(v\rho_t ) =0   \,\,\,\, \text{ in } \,\,\,\, (0,1) \times \Om 
\end{equation}
is by now classical. %
It was first introduced in the Euclidean setting by Ambrosio in \cite{ambrosioinventiones}, where it was employed to investigate uniqueness and stability of Lagrangian flows in the context of DiPerna-Lions Theory \cite{dipernalions}. Since then it has been applied to different tasks \cite{acfarma,bernard1, bernard2, bianchinibonicattoinventiones,bianchiniSIAM} and extended to various settings \cite{bonicattogusev2018,lisini,maniglia,3superposition}. %
In \cite{ags} the velocity field $v$ is assumed to satisfy
\begin{equation} \label{intro:hom_bound}
\int_0^1\int_{\Om} |v(t,x)|^2 \, d\rho_t(x) \, dt<\infty\,.
\end{equation}
 An elementary solution to \eqref{eq:intro_cont_hom} is of  the form $t \mapsto \delta_{\gamma(t)}$ where $\gamma : [0,1] \rightarrow \Om$ is an absolutely continuous curve solving the characteristic equation $(i)$ in \eqref{eq:int_ODEnonhom}. 
Due to the lack of regularity of $v$, solutions to the initial value problem associated to $(i)$ are not unique. Such non-uniqueness is reflected in the superposition formula, which in this case is achieved by constructing a probability measure $\sigma$ on the set $\Gamma:=C([0,1];\Om)$. To be more precise, it can be shown that if $\rho_t \in \M^+(\Om)$ is a narrowly continuous solution to \eqref{eq:intro_cont_hom}  and $v$ satisfies \eqref{intro:hom_bound}, then there exists a measure $\sigma \in \M^+(\Gamma)$ concentrated on absolutely continuous curves satisfying $(i)$, with the property that $\rho_t$ can be represented by the pushforward of $\sigma$ via the evaluation map $e_t(\gamma) := \gamma(t)$, that is, 
\begin{equation}\label{eq:superpositionhomintro}
\int_\Om \f(x)  \,d\rho_t(x) = \int_\Gamma \varphi(\gamma(t))\, d\sigma(\gamma) \quad \text{ for all } \quad \f \in C(\Om)\,, \,\, t \in [0,1]\,.
\end{equation}   
We refer the reader to \cite[Theorem 8.2.1]{ags} for a proof of \eqref{eq:superpositionhomintro} with $\Om=\R^d$ and to \cite[Theorem 7]{bcfr} for the case of $\Om$ being the closure of a bounded domain.

A generalization of \eqref{eq:superpositionhomintro} for positive measure solutions to the inhomogeneous continuity equation \eqref{eq:intro_cont} in $\Om=\R^d$ is presented in \cite{maniglia}. Specifically, the following is proven in \cite[Theorem 4.1]{maniglia}: suppose that $\rho_t \in \M^+(\Om)$ is a narrowly continuous solution to \eqref{eq:intro_cont}, that $v$ satisfies \eqref{intro:hom_bound} and $g$ is bounded; then there exists a representing measure $\sigma \in \M^+(\Gamma \times \Om)$, concentrated on pairs $(\gamma,x)$ with $\gamma$ absolutely continuous curve solving $(i)$ in \eqref{eq:int_ODEnonhom} with the initial condition $\gamma(0)=x$, and such that $\rho_t$ is represented via the implicit formula 
\begin{equation}\label{eq:intro_implicit}
	\int_\Om \f(x)  \,d\rho_t(x) = \int_{\Gamma \times \Om} \varphi(\gamma(t))\, d\sigma(\gamma,x) + \int_0^t \left( \int_\Om \int_\Gamma \f(\gamma(t)) \, d\sigma_s^x(\gamma) g(s,x) \, d\rho_s(x)   \right) \, ds  \,,
	\end{equation} 
	for all $\f \in C(\Om)$, where for fixed $t$, the family $\{\sigma_t^x\}_{x \in \Om}$ is the disintegration of $\sigma$ with respect to $(\tilde{e}_t)_\# \sigma \in \M^+(\Om)$, with $\tilde{e}_t(\gamma,x):=\gamma(t)$. There are two main drawbacks with the superposition principle from  \cite{maniglia}: First, the representation formula \eqref{eq:intro_implicit} is implicit; Second, the source term $g$ is required to be bounded. Such assumption on $g$ is substantial, as it implies uniqueness of solutions to $(ii)$ in \eqref{eq:int_ODEnonhom} along any trajectory. This fact essentially allows the author of \cite{maniglia} to construct the measure $\sigma$ in \eqref{eq:intro_implicit} in the same way as the one in  \eqref{eq:superpositionhomintro}. Another limitation of \cite{maniglia} is that it is not possible to provide a representation via \eqref{eq:intro_implicit} for solutions with mass that is vanishing or generating from zero during the evolution (for an example, see Remark~\ref{rem:zero_solutions}).

The main focus of this paper is to obtain a superposition principle for \eqref{eq:intro_cont} which overcomes the above mentioned limitations of \cite{maniglia}. Indeed we obtain an explicit representation formula for \eqref{eq:intro_cont} that resembles \eqref{eq:superpositionhomintro}. In addition, we remove the boundedness assumption on $g$, and we replace it by the growth condition \eqref{eq:quadraticbound}. Removing such assumption on $g$ is far from straightforward, as it requires a new functional analytic framework for constructing a representation measure $\sigma$. In fact, the low regularity of $g$ implies non-uniqueness for the initial value problem associated with $(ii)$ in \eqref{eq:int_ODEnonhom}. This suggests that a measure $\sigma$ representing a solution $t\mapsto \rho_t$ to \eqref{eq:intro_cont} has to account for non-uniqueness both for the trajectories $\gamma$ and the weights $h$. Therefore, $\sigma$ cannot just be a measure on $\Gamma$, but rather on a space of pairs $(\gamma,h)$, as discussed in Theorem \ref{thm:intro_main} below.

We now discuss the coupling of the continuity equation at \eqref{eq:intro_cont} with the energy at \eqref{eq:quadraticbound}, which is at the center of recent important developments in the theory of Unbalanced Optimal Transport. The classical theory of Optimal Transport, in its Monge-Kantorovich  formulation \cite{kantorovich,sant,villani}, concerns the problem of transporting mass from a probability measure into a target one, while minimizing a given cost. 
 Benamou and Brenier \cite{bb} made the crucial observation that the classical formulation of optimal transport has a dynamic counterpart, which links the continuity equation \eqref{eq:intro_cont_hom} with the energy at \eqref{intro:hom_bound}.
More precisely they observed that it is possible to compute the optimal transport between two probability measures $\rho_0$ and $\rho_1$ by minimizing the dissipation at \eqref{intro:hom_bound} 
among all the curves of probability measures $t \mapsto \rho_t$ and velocity fields $v$ solving the continuity equation %
\eqref{eq:intro_cont_hom} 
with initial and final conditions given by $\rho_0$ and $\rho_1$ respectively. %
Such dynamic formulation makes possible to endow the space of probability measures with a differentiable structure \cite{ags}, bringing to light deep connections between optimal transport and functional analytic issues, such as the characterization of differential equations as gradient flows in spaces of measures \cite{ags,agsduke,agsinventiones,ottokinderleher,ottokinderleher2,otto1,otto2} or the derivation of sharp inequalities \cite{agueh,cordero,maggi2,maggi,nazaret,otto_villani}. Particularly in connection to applications, the assumption of mass preservation during the evolution is quite restrictive. Overcoming this limitation is at the core of the so-called unbalanced optimal transport theory. Among the various formulations, we highlight the one introduced in \cite{chizat,kmv,liero}. There, transporting a positive measure $\rho_0$ into a target one $\rho_1$ corresponds to minimize a weighted version of \eqref{eq:quadraticbound} 
among all curves of positive measures $t \mapsto \rho_t$ and fields $v, g$ satisfying the inhomogeneous continuity equation \eqref{eq:intro_cont} with initial and final conditions given by $\rho_0$ and $\rho_1$ respectively. The quantity at \eqref{eq:quadraticbound} takes the name of Wasserstein-Fisher-Rao or Hellinger-Kantorovich energy in the literature. Such an approach has been successfully employed in applications where mass preservation is violated \cite{bf, chizat3,liero2,schmitzerwirth2019}. In particular in \cite{liero} it is shown that the above minimization procedure induces a distance which is compatible with a differentiable structure on the space $\M^+(\Om)$. This distance can also be derived from the dynamic formulation of the Logarithmic-Entropy Optimal Transport problem \cite{liero} or can be regarded as dissipation energy for a certain class of scalar reaction-diffusion equations \cite{liero2}. 

We conclude this introduction by discussing in more details the superposition principle we propose for \eqref{eq:intro_cont}, as well as the applications provided in this paper.  
The rest of the manuscript is organized as follows. In Section \ref{sec:preliminaries} we introduce basic notations, as well as presenting some results on continuity equations and optimal transport energies. In Section \ref{sec:functional} we set the functional analytic framework needed in order to prove our superposition principle. In particular we investigate properties of the Hellinger-Kantorovich energy \eqref{eq:quadraticbound} when restricted to elementary solutions to \eqref{eq:intro_cont}. In Section \ref{sec:main_thm} we provide a proof for the main result of this paper, that is, the superposition principle in Theorem \ref{thm:intro_main} below. Finally, in Sections \ref{sec:uniqueness}, \ref{sec:extremalpoints} we detail applications of the superposition principle to uniqueness for solutions to \eqref{eq:intro_cont} and to sparsity for dynamic inverse problems with Hellinger-Kantorovich-type regularizers.

\subsection{Main result}

To obtain a superposition principle for \eqref{eq:intro_cont} under the energy bound \eqref{eq:quadraticbound}  we construct a positive measure $\sigma$ on  
the set $\car_\Om$ of narrowly continuous curves $t \mapsto \rho_t$ with values in 
$
\cone_\Om:=\{ h \delta_{\gamma} \in \M(\Om) \, \colon \, h \geq 0 \,, \gamma \in \Om \}\,.
$ 
We endow $\cone_\Om$ with the flat distance of measures and $\car_\Om$ with the respective supremum distance. In this way $\car_\Om$ becomes a separable metric space. Notice that $\car_\Om$ plays the role of the set of continuous curves $\Gamma$ in \eqref{eq:superpositionhomintro}.  As we will see, c.f.~Remark~\ref{rem:compcone}, the construction of $\cone_\Om$ closely resembles the cone space introduced in \cite{liero2,liero} to study absolutely continuous curves with respect to the Hellinger-Kantorovich distance. It is immediate to check that elements of $\car_\Om$ can be represented by $\rho_t = h(t) \delta_{\gamma(t)}$, for some non-negative weight $h \in C[0,1]$ and curve $\gamma \in C(\{h>0\}; \Om)$, where we set $\{h>0\} := \{t \in [0,1] \, \colon \, h(t)>0\}$.
Thus, the mass of the elements of $\car_\Om$  is varying continuously in time and is allowed to vanish, reflecting the behavior of solutions to \eqref{eq:intro_cont}. The measure $\sigma$ we construct is concentrated on elements $\rho_t = h(t) \delta_{\gamma(t)} \in \car_\Om$, with $h$ and $\gamma$ solving the system of ODEs:
\begin{equation}\label{eq:int_ODEnonhom2}
(i) \ \ \dot{\gamma}(t)= v(t,\gamma(t)) \quad a.e. \ \text{ in }\  \{h>0\} \qquad (ii)\ \  \dot{h}(t)= g(t,\gamma(t))h(t)\qquad a.e. \ \text{ in }\  (0,1) \,.
\end{equation}
Notice that, in comparison to the system of characteristics at \eqref{eq:int_ODEnonhom}, we are restricting the first ODE to the set $\{h>0\}$. Indeed, if $h(t) = 0$, then $\rho_t = 0$ and thus we lose any information on the trajectories for that time instant. The above observations are formalized in the following theorem, which is the main result of our paper (c.f. Theorem \ref{thm:lifting}).

\begin{theorem} \label{thm:intro_main}
Let $\Om \subset \R^d$ be the closure of an open bounded domain.
Let $\rho_t \colon [0,1] \to \M^+(\Om)$ be a narrowly continuous solution to \eqref{eq:intro_cont} for some Borel measurable $v \colon (0,1) \times \Om \to \R^d$ , $g \colon (0,1) \times \Om \to \R$ satisfying \eqref{eq:quadraticbound} and such that $v$ has no flux on $\de \Om$. 
Then there exists a measure $\sigma \in \M^+(\car_\Om)$ concentrated on curves of measures $\rho_t = h(t) \delta_{\gamma(t)}$ with $h,
\gamma$ solving  \eqref{eq:int_ODEnonhom2} and such that 
\begin{equation} \label{eq:intro_representation}
\int_\Om \f(x) \, d\rho_t(x) =\int_{\car_\Om} h(t)\f(\gamma(t))\, d\sigma(\gamma,h)  \,\,\, \text{ for all } \,\, \, \f \in C(\Om)\,,\,\, t \in [0,1] \,.
\end{equation}
Conversely, assume that $\sigma \in \M^+(\car_\Om)$ is concentrated on solutions to \eqref{eq:int_ODEnonhom2} and satisfies 
\begin{equation} \label{eq:intro_representation_bound}
\int_0^1 \int_{\car_\Om} h(t) \left(1+ | v(t,\gamma(t))| + |g(t,\gamma(t))|  \right) \, d\sigma(\gamma,h)\, dt < \infty\,.
\end{equation}
Then \eqref{eq:intro_representation} defines a narrowly continuous curve of positive measures solving \eqref{eq:intro_cont}.
\end{theorem}

Notice that the growth condition \eqref{eq:intro_representation_bound} is natural, in the sense that if a measure $\sigma$ represents $\rho_t$ and \eqref{eq:quadraticbound} holds, then automatically $\sigma$ satisfies \eqref{eq:intro_representation_bound}. We refer the reader to Remark \ref{rem:conv_bound} below for more details.
We also remark that the set $\Om$ in Theorem \ref{thm:intro_main} is required to be bounded. Indeed it would be interesting to extend our result to unbounded domains, in the spirit of \cite{ags,maniglia} where $\Om=\R^d$ is considered. However, it seems that a different proof strategy or stronger assumptions are required, see Remark \ref{rem:proofRd} below for details. Moreover, similarly to \cite{ags,maniglia}, it should be possible to prove a version of Theorem \ref{thm:intro_main} in which  \eqref{eq:quadraticbound} is replaced by an $L^p$ bound for $1 \leq p \leq \infty$. Such analysis falls outside the scope of our paper.

The proof of Theorem \ref{thm:intro_main} is presented in Section \ref{sec:main_thm}. It is based on a similar smoothing strategy as the one employed in \cite{ambrosioinventiones} to prove \eqref{eq:superpositionhomintro}. However in this case there are two main differences: first one needs to establish compactness properties for a coercive version of the Hellinger-Kantorovich energy when restricted to elements of $\cone_\Om$, see Proposition \ref{lem:compact_sublevels}; second the smoothing needs to take into account the possibility of the measure $\rho_t$ vanishing at some time instance, as detailed in Remark~\ref{rem:zero_solutions} below.

\subsection{Uniqueness of solutions to the continuity equation}
In Section \ref{sec:uniqueness} we present the first application of the superposition principle of Theorem \ref{thm:intro_main}. 
Our aim is to show that uniqueness of solutions for the system of ODEs at \eqref{eq:int_ODEnonhom}, up to their possible vanishing time, implies uniqueness for measure solutions to the inhomogeneous continuity equation \eqref{eq:intro_cont} satisfying the bound \eqref{eq:quadraticbound} and with minimal total variation. 
The key ingredient of the proof is formula \eqref{eq:intro_representation}, which allows to decompose any solution of \eqref{eq:intro_cont} satisfying the bound \eqref{eq:quadraticbound} into a superposition of elementary curves $t \mapsto h(t)\delta_{\gamma(t)}$ such that $(\gamma,h)$ are solutions to \eqref{eq:int_ODEnonhom2}. Such representation allows to link uniqueness for \eqref{eq:int_ODEnonhom} with the one for \eqref{eq:intro_cont}.
The main difference between our result and the classical one for the homogeneous continuity equation \cite[Theorem 9]{ambrosiocrippalecturenotes} lies in the fact that elementary solutions $\rho_t = h(t)\delta_{\gamma(t)}$ are allowed to vanish in time. 
In this case uniqueness for  \eqref{eq:int_ODEnonhom} is not enough to ensure uniqueness of solutions to the inhomogeneous continuity equation. Indeed, when the mass of a solution vanishes at a given time instant $\bar t \in (0,1)$, the uniqueness assumption for \eqref{eq:int_ODEnonhom} is not providing any information on the behavior of the solutions for $t > \bar t$: this is because the measure $\sigma$ is concentrated on solutions to \eqref{eq:int_ODEnonhom2} where $i)$ is only valid in the set $\{h>0\}$.
Therefore, in order to recover uniqueness for \eqref{eq:intro_cont}, we impose an extra constraint on the total variation of its solutions. More precisely, we show that solutions to \eqref{eq:intro_cont} with minimal mass can be represented, invoking Theorem \ref{thm:intro_main}, by a measure $\sigma$ concentrated on curves $t \mapsto h(t)\delta_{\gamma(t)}$ such that $(\gamma,h)$ solves \eqref{eq:int_ODEnonhom2} and $h$ is strictly positive in an interval $[0,\tau) \cap [0,1]$ for some $\tau \in \R$. Such observation allows to employ uniqueness for the system of characteristics at \eqref{eq:int_ODEnonhom}, up to their possible vanishing time, to infer uniqueness for measure solutions to \eqref{eq:intro_cont} with minimal total variation.  We obtain the following theorem, c.f. Theorem \ref{thm:uniqueness}.

\begin{theorem}\label{thm:uniquenessintrro}
Let $v \colon (0,1) \times \Om \to \R^d$, $g \colon (0,1) \times \Om \to \R$ be Borel measurable functions and $A\subset \Om$ be a Borel measurable set. Suppose that: 
\begin{itemize}
\item[\textsc{(Hyp)}] For each $x\in A$ the solution of the system of ODEs \eqref{eq:int_ODEnonhom2} with initial value $(x,1)$ is unique in $[0,\tau)$ for every $\tau \in (0,1)$ such that $[0,\tau) \subset \{h>0\}$.
\end{itemize}
Then, for any initial datum $\rho_0 \in \mathcal{M}^+(\Om)$ concentrated on $A$, the inhomogeneous continuity equation \eqref{eq:intro_cont} admits at most one positive narrowly continuous solution $t \mapsto \rho_t$ satisfying \eqref{eq:quadraticbound}, with initial datum $\rho_0$, and such that
$\| \rho\|_{\mathcal{M}} \leq  \|\tilde \rho\|_{\mathcal{M}}$
for every $t \mapsto \tilde \rho_t$ positive narrowly continuous solution to \eqref{eq:intro_cont}  satisfying \eqref{eq:quadraticbound}, and such that $\tilde \rho_0 = \rho_0$. 
\end{theorem} 

\medskip

\subsection{Extremal points of the Hellinger-Kantorovich energy}
In the context of inverse problems, the knowledge of the structure of extremal points of the regularizer allows to numerically reconstruct sparse solutions, i.e., solutions given by the superposition of finitely many extremal points \cite{bcfr2,inprep2}. It has been recently proposed \cite{bf} to regularize dynamic inverse problems via an energy related to the one at \eqref{eq:quadraticbound}.  
To be more specific, the energy at \eqref{eq:quadraticbound} can be recast into a convex functional $B_\delta$ over the space $\M:=\M((0,1)\times \Om)^{d+2}$  defined by
\begin{equation} \label{eq:intro_formula B}
B_\delta(\rho,m,\mu) :=  \frac12\int_0^1\int_{\Om} \left| \frac{dm}{d\rho}\right|^2 + \delta^2\left| \frac{d\mu}{d\rho}\right|^2 \, d\rho 
\end{equation}
if $\rho \geq 0$, $m,\mu \ll \rho$, and set to $\infty$ otherwise, where $\delta > 0$ is a parameter (c.f. Section \ref{sec:benamou}). The regularizer studied in \cite{bf} consists in the energy at \eqref{eq:intro_formula B} to which the total variation of $\rho$ is added, while enforcing the continuity equation constraint $\de_t \rho + \div m = \mu$.  
An analysis of the extremal points of such energy is currently missing in the literature: 
Therefore, in this paper, we employ the superposition principle of Theorem \ref{thm:intro_main} to characterize the extremal points of the set 
\begin{equation}\label{eq:introball}
\mathscr{B} = \{(\rho,m,\mu) :\partial_t \rho + \div m = \mu,\,\,  \beta B_{\delta}(\rho,m,\mu) + \alpha \|\rho\|_{\mathcal{M}} \leq 1\}\,,
\end{equation}
where $\alpha,\beta > 0$ are parameters.
Notice that we do not impose boundary conditions in the continuity equation at \eqref{eq:introball}.  Moreover the total variation of $\rho$ is added to the functional $B_\delta$, in order to enforce coercivity, and thus compactness of $\mathscr{B}$. We prove the following result (c.f. Theorem \ref{thm:fr}).

\begin{theorem} \label{thm:intro_ext}
The extremal points of the set defined in \eqref{eq:introball}
are exactly given by the zero measure $(0,0,0)$ and the triples of measures $(\rho, m,\mu)$ such that $\rho= h(t)\, dt \otimes \delta_{\gamma(t)}$, $m=\dot{\gamma}(t) \rho$, $\mu=\dot{h}(t)\,  dt \otimes \delta_{\gamma(t)}$ with the following properties:
\begin{itemize}
\item[a)]  $h , \, \sqrt{h} \in \AC^2[0,1]$, $\gamma \in C(\{h>0\};\Om)$ and
 $\sqrt{h}\gamma \in\AC^2 ([0,1];\R^d)$,
\item[b)] the set $\{h>0\}$ is connected,
\item[c)] the energy satisfies $ \beta B_{\delta}(\rho,m,\mu) + \alpha \|\rho\|_{\mathcal{M}} =1$.
\end{itemize}
\end{theorem}
In the above we denote by $\AC^2$ the set of absolutely continuous functions with a.e.~derivative in $L^2$ (see \cite[Section 1.1]{ags} for a precise definition).

Theorem \ref{thm:intro_ext} is a generalization of the results obtained in \cite{bcfr}, where the Benamou-Brenier energy with homogeneous continuity equation constraint is considered.
In Section \ref{sec:ext_sparse} we apply Theorem~\ref{thm:intro_ext} to understand the structure of sparse solutions for dynamic inverse problems with unbalanced optimal transport regularization. In particular, we consider the inverse problem proposed in \cite{bf}, where the minimization of the %
energy at \eqref{eq:introball} is coupled with a fidelity term penalizing  the distance between $\rho$ and some fixed observation. %
Applying recent results on sparsity \cite{chambolle,bc} we show that the minimization problem in \cite{bf} admits a solution which is a finite linear combination of extremal points of $\mathscr{B}$, that is, of curves as described in Theorem~\ref{thm:intro_ext}. % 

\section{Preliminaries} \label{sec:preliminaries}
For measure theory notations and definitions we follow \cite{afp}. Given a metric space $Y$ we denote by $\M(Y)$, $\M(Y ; \R^d)$, $\M^+(Y)$ the spaces of bounded Borel measures, bounded vector Borel measures, bounded positive Borel measures on $Y$, respectively. Throughout the paper, whenever we say that a set or a function is measurable, we always intend Borel measurable, i.e., measurability with respect to the Borel $\sigma$-algebra. For a measure $\mu$ we denote its total variation measure by $|\mu|$. We say that a sequence of measures $\{\mu_n\}_n$ on $Y$ converges narrowly to $\mu$ if 
$
\int_Y \f(y) \, d\mu_n(y) \to \int_Y \f(y) \, d \mu(y)$ for all $\f \in C_b(Y)$, where $C_b(Y)$ denotes the set of real valued continuous and bounded functions on $Y$. 
 
 Let $\Om \subset \R^d$ be the closure of a bounded domain, with $d \in \N$, $d \geq 1$, and define the time-space domain $X_\Om:=(0,1) \times \Om$.
We say that $\rho \in \M(X_\Om)$ disintegrates with respect to time if there exists a Borel family of measures $\{\rho_t\}_{t \in [0,1]} \subset
\M(\Om)$ such that
	$
	\int_{X_\Om} \f \, d\rho = \int_0^1 \int_{\Om} \f (t,x) \, d\rho_t(x) \, dt$ for all $\f \in L^1_{\rho}(X_\Om)$. 
The disintegration is denoted by $\rho =dt \otimes  \rho_t$. Further, a curve of measures $t \in [0,1] \mapsto \rho_t \in \M(\Om)$ is narrowly continuous if the map
$t \mapsto \int_{\Om} \f(x) \, d\rho_t(x)$
is continuous for each fixed $\f \in C(\Om)$. The family of narrowly continuous curves is denoted by $\curves$. Notice that if $t \mapsto \rho_t$ is narrowly continuous, by the principle of uniform boundedness, it follows that $\rho:=dt \otimes \rho_t$ belongs to $\M(X_\Om)$. 
We also introduce $\pcurves$ as the family of narrowly continuous curves with values into $\M^+(\Om)$. The above definitions extend verbatim to the case $\Om=\R^d$.

\subsection{Continuity equation} 
Set $\M_\Om := \M (X_\Om) \times \M (X_\Om;\R^d) \times \M(X_\Om)$. We say that the triple $(\rho,m,\mu) \in \M_\Om$ solves the continuity equation 
\begin{equation} \label{cont}
\de_t \rho + \div m =\mu \quad \text{ in } \quad  X_\Om \,,
\end{equation}
whenever \eqref{cont} holds in the sense of distributions, i.e.,
\begin{equation} \label{cont weak}
\int_{X_\Om}  \de_t \f \, d\rho + 
\int_{X_\Om}  \nabla \f \cdot  dm  + \int_{X_\Om}   \f \, d\mu=0   \quad \text{for all} \quad \f \in C^{\infty}_c (X_\Om) \,.
\end{equation} 
Here, $\rho$ represents a density, $m$ a momentum field advecting $\rho$, while $\mu$ is a source term accounting for mass change. 
The above definition also holds for unbounded spatial domains, e.g., $\Om =\R^d$. 
Moreover the time interval $(0,1)$ can be replaced by $(0,T)$ with $T>0$. 
We remark that \eqref{cont weak} includes no flux boundary conditions for $m$ on $\de \Om$, and no initial conditions for $\rho$ are prescribed. Moreover \eqref{cont weak} can be equivalently tested with maps in $C^1_c (X_\Om)$  \cite[Remark 8.1.1]{ags}. The following lemma provides some properties of solutions to \eqref{cont weak} which will be needed in the coming analysis. The statement holds both in bounded domains as well as in $\R^d$. For a proof in bounded domains see, e.g., Propositions 2.2, 2.4 in \cite{bf}, which can be easily generalized to $\R^d$.

\begin{lemma} \label{lem:prop cont}
	Assume that $(\rho,m,\mu) \in \M_\Om$ satisfies \eqref{cont weak} with $\rho \in \M^+(X_\Om)$. Then $\rho =dt \otimes  \rho_t$, 
	where $\rho_t \in \M^+(\Om)$ for a.e.~$t$ in $(0,1)$. Moreover the map $t \mapsto \rho_t(\Om)$ belongs to $BV(0,1)$, with distributional derivative given by $\pi_\#\mu$, where $\pi \colon X_\Om \to (0,1)$ is the projection on the time coordinate. If in addition $m=v\rho$, $\mu = g \rho$ for some measurable $v \colon X_\Om \to \R^d$, $g \colon X_\Om \to \R$ with
	\[
	\int_0^1\int_{\Om} |v(t,x)| + |g(t,x)| \, d\rho_t(x) \, dt < \infty\,,
	\]
then there exists a unique curve $t \mapsto \tilde{\rho}_t$ in $\pcurves$ such that $\rho_t = \tilde{\rho}_t$ a.e.~in $(0,1)$. 
\end{lemma}
In the rest of the paper we will identify $\rho_t$ with its narrowly continuous representative $\tilde{\rho}_t$, whenever the assumptions of Lemma \ref{lem:prop cont} hold.

\subsection{Optimal transport energy} \label{sec:benamou}

We now introduce the Wasserstein-Fisher-Rao energy, also known as the  Hellinger-Kantorovich energy, as originally done in \cite{chizat, kmv,liero}. %
To this end, let $\delta>0$ be a fixed parameter. Define the convex, one-homogeneous and lower semi-continuous map $\Psi_\delta \colon \R \times \R^d \times \R \to [0,\infty]$ by setting
\begin{equation} \label{conv conj}
\Psi_\delta (t,x,y):= \begin{cases}
 \frac{|x|^2 +\delta^2 y^2}{2t}    & \text{if } t >0 \,,\\
 0                   & \text{if } t=|x|=y=0 \,, \\
 \infty	         & \text{otherwise}\,,
 \end{cases}
\end{equation} 
where $\infty y^2 = \infty$ for $y \neq 0$ and $\infty y^2 = 0$ for $y=0$.
The Wasserstein-Fisher-Rao energy is given by the map $B_\delta \colon \M_\Om \to [0,\infty]$ defined by
\begin{equation} \label{bb}
B_\delta (\rho,m,\mu):= \int_{X_\Om} \Psi_\delta \left( \frac{d\rho}{d\lambda}, \frac{dm}{d\lambda},\frac{d\mu}{d\lambda}\right) \, d \lambda \,,
\end{equation}
where $\lambda \in \M^+(X_\Om)$ is an arbitrary measure such that $\rho,m,\mu \ll \lambda$.
Definition \eqref{bb} does not depend on the choice of $\lambda$, as $\Psi_\delta$ is one-homogeneous. Properties of the energy $B_\delta$ which are relevant in the following analysis are summarized in Lemma~\ref{lem:prop B} (for a proof see \cite[Proposition~2.6]{bf}).
We now introduce a coercive version of $B_\delta$: Set
\[
\Distr_\Om:= \left\{(\rho,m,\mu) \in \M_\Om \, \colon \, \de_t \rho +  \div m = \mu \, \, \text{in the sense of} \,\, \eqref{cont weak}\right\} \,,
\] 
and define the functional $J_{\alpha, \beta,\delta}\colon \M_\Om \to [0,\infty]$ as
 \begin{equation} \label{bb reg}
  J_{\alpha, \beta,\delta}(\rho,m,\mu) := \begin{cases} 
      \beta B_\delta(\rho, m,\mu) +
      \alpha \norm{\rho}_{\M(X_\Om)} & \,\, \text{ if }(\rho,m,\mu) \in \Distr_\Om, \\ 
       \infty\qquad  & \,\,   \text{ otherwise},
    \end{cases} 
\end{equation}
where $\alpha>0$ and $\beta > 0$ are fixed constants.
We remark that adding the total variation of $\rho$ to $B_\delta$ enforces the balls of $J_{\alpha, \beta,\delta}$ to be  compact in the weak* topology of $\M_\Om$. Such property, together with others, is the object of Lemma \ref{lem:prop J}. The content of Lemma \ref{lem:prop J} is based on results proven in  \cite[Lemmas 4.5, 4.6]{bf}.

\subsection{Characteristics theory for the continuity equation}

We start by recalling a classical result on the theory of ordinary differential equations in $\R^d$ \cite[Lemma 8.1.4]{ags}.

\begin{proposition}\label{prop:globallydefined}
Let $v \colon [0,1] \times \R^d \to \R^d$ be measurable and such that
	\begin{equation}
	\int_0^1 \sup_{x \in \R^d}|v(t,x)| +  \Lip (v(t,\cdot),\R^d)  \, dt < \infty \,. \label{eq:comp1}
	\end{equation}
	Then for each $x \in \R^d$ the ODE 
	\begin{equation} \label{prel:ODE}
	\dot{X}_x(t) = v(t,X_x(t))\,\,  \text{ for a.e. } t \in (0,1)\,, \qquad
	 X_x(0)=x\,,  
	\end{equation}
	admits a unique absolutely continuous  
	solution $t \mapsto X_x(t)$ defined for all $t \in [0,1]$. 
\end{proposition}

Next we provide a representation formula for measure solutions of the continuity equation \eqref{cont}.
This is the analogue of \cite[Lemma 8.1.6]{ags} for the inhomogeneous continuity equation, and a generalization of \cite[Proposition 3.6]{maniglia} to the case of $g$ unbounded.

\begin{proposition} \label{prop:push_forward}
	Let $v \colon [0,1]\times \R^d \to \R^d$, $g \colon [0,1]\times \R^d \to \R$ be measurable. Assume  that 
\begin{equation}	\int_0^1 \sup_{x \in \R^d}|g(t,x)| +  \Lip (g(t,\cdot),\R^d)  \, dt < \infty\label{eq:comp2}
\end{equation}	
and  \eqref{eq:comp1} hold. 
Let $\rho_0 \in \M^+(\R^d)$  and denote by $t \mapsto X_x(t)$ the unique solution to \eqref{prel:ODE} defined for all $t \in [0,1]$ and $x \in \R^d$. Then, the map
	\begin{equation}  \label{prop:push_forward:formula}
	t \mapsto \rho_t :=(X_{(\cdot)}(t))_\#  \left( \rho_0 \, e^{\int_0^t g(s,X_{(\cdot)}(s)) \, ds}  \right)
	\end{equation}
	is a narrowly continuous solution to the continuity equation $\de_t \rho_t + \div (v \rho_t) = g \rho_t$ in $(0,1) \times \R^d$ in the sense of \eqref{cont weak}, where the push-forward in \eqref{prop:push_forward:formula} is with respect to the space variable. 
\end{proposition}

\begin{proof}
Narrow continuity of $t \mapsto\rho_t$ follows immediately from \eqref{eq:comp2}, dominated convergence and the continuity of $t \mapsto X_x(t)$ for each $x$.
Let now $\f \in C^1_c((0,1) \times \R^d)$. Then for $\rho_0$-a.e.~$x$ in $\R^d$, the map $t \mapsto \f(t,X_x(t))$ is absolutely continuous in $(0,1)$, with a.e.~derivative given by
\begin{equation} \label{prop:push_forward:1}
\frac{d}{dt} \, \f(t,X_x(t)) = \de_t \f(t,X_x(t)) + \nabla \f(t,X_x(t)) \cdot v(t,X_x(t))\,,
\end{equation}
thanks to Proposition \ref{prop:globallydefined}. By \eqref{eq:comp2} we also have that 
$t \mapsto \f(t,X_x(t)) e^{\int_0^t g(s,X_x(s)) \, ds}$ is absolutely continuous in $(0,1)$, and for a.e.~$t \in (0,1)$ it holds 
\begin{equation} \label{prop:push_forward:2}
\frac{d}{dt}  \bigg( \f(t,X_x(t)) e^{\int_0^t g(s,X_x(s)) \, ds}   
\bigg) = \bigg( \frac{d}{dt}  \, \f(t,X_x(t))  + 
\f(t,X_x(t)) \, g(t,X_x(t)) \bigg) e^{\int_0^t g(s,X_x(s)) \, ds}\,.
\end{equation}
In particular, it is immediate to check that
\[
\int_0^{1} \int_{\R^d} \left| \frac{d}{dt}  \left( \f(t,X_x(t)) e^{\int_0^t g(s,X_x(s)) \, ds}   
\right)    \right| \, d\rho_0(x) \, dt \leq \norm{\f}_{C^1} \, \rho_0(\R^d) \, e^{M_g} \left( 1 + M_v + M_g \right)\,,
\]
where $M_v := \int_0^1 \sup_{x \in \R^d} |v(t,x)|\,dt$, $M_g := \int_0^1 \sup_{x \in \R^d} |g(t,x)| \, dt$, which are finite by \eqref{eq:comp1}, \eqref{eq:comp2}. Therefore, we can apply Fubini's theorem and \eqref{prop:push_forward:formula}, \eqref{prop:push_forward:1}, \eqref{prop:push_forward:2}, to compute
\[
\begin{aligned}
 \int_{X_{\R^d}} \left(\de_t \f + \nabla \f \cdot v + \f \, g \right) \, d\rho 
 = \int_{\R^d} \int_0^{1}\frac{d}{dt}  \left( \f(t,X_x(t)) e^{\int_0^t g(s,X_x(s)) \, ds}  
\right) \, dt  \, d\rho_0(x)\,,
\end{aligned}
\]
where $\rho=dt \otimes \rho_t$. 
Now notice that the above right-hand side vanishes since $\f$ is compactly supported, concluding the proof. 
\end{proof}

The next proposition states that, under some regularity assumptions, every solution of \eqref{cont weak} can be represented as in \eqref{prop:push_forward:formula}.

\begin{proposition}\label{prop:solutionbypush}
Assume that $\rho_t \colon [0,1]  \to \M^+(\R^d)$ is a narrowly continuous solution to the continuity equation $\de_t \rho_t + \div(v\rho_t ) = g\rho_t  $ in $(0,1) \times \R^d$ in the sense of \eqref{cont weak}, for some Borel maps $v \colon [0,1] \times \R^d \to \R^d$, $g \colon [0,1] \times \R^d \to \R$ satisfying \eqref{eq:comp1},  \eqref{eq:comp2} and
\begin{equation}
\int_0^1 \int_{\R^d}  |v(t,x)|+|g(t,x)|  \, d\rho_t(x) \, dt < \infty\,. \label{eq:comp0}	
\end{equation}
 Then for $\rho_0$-a.e. $x \in \R^d$ the ODE \eqref{prel:ODE} admits a solution $X_x(t)$ for $t \in [0,1]$, and 
\[
\rho_t  =(X_{(\cdot)}(t))_\#  \left( \rho_0 \, e^{\int_0^t g(s,X_{(\cdot)}(s)) \, ds}  \right) \,\,\, \text{ for each } \,\,\, t \in [0,1]\,,
\]
where the push-forward is with respect to the space variable.
\end{proposition}

\begin{proof}
Define the map $t \mapsto \mu_t := (X_{(\cdot)}(t))_\# \left( \rho_0 \, e^{\int_0^t g(s,X_{(\cdot)}(s)) \, ds}  \right)$. Proposition \ref{prop:push_forward} implies that $\mu_t$ is a narrowly continuous solution to the continuity equation in $(0,1) \times \R^d$. Moreover $\mu_0=\rho_0$ by construction. It is immediate to check that $\mu_t - \rho_t$ and $\rho_t -\mu_t$ satisfy \eqref{eq:comp222}. As $\mu_t - \rho_t$ and $\rho_t - \mu_t$ both satisfy the continuity equation,  we can apply (twice) the comparison principle in Proposition \ref{prop:comparison} to deduce that $\mu_t =\rho_t$ for every $t \in [0,1]$.
\end{proof}

\section{Functional analytic setting} \label{sec:functional}

In this  section we discuss the functional analytic setting that is instrumental in proving the superposition principle in Theorem \ref{thm:intro_main}. Throughout the section, $V$ will be the closure of a bounded domain of $\R^d$, with $d \in \N$, $d \geq 1$. 
We recall the notations  $X_V:=(0,1) \times V$ and $\M_V:=\M(X_V) \times \M(X_V;\R^d) \times \M(X_V)$.

\subsection{Curves in cones of measures} \label{sec:cone}
We start by introducing the set
\begin{equation} \label{eq:cone}
\cone_V:=\left\{ h \delta_\gamma \in \M(V) \, \colon \, h \geq 0,\, \gamma \in V \right\} 
\end{equation}
and the space of narrowly continuous curves with values in $\cone_V$, i.e.,
\begin{equation} \label{eq:narr}
\car_V := \left\{(t \mapsto \rho_t) \in \pcurvesV \, \colon \, \rho_t \in \cone_V \, \text{ for all } \, t \in [0,1] \right\}\, .
\end{equation}
Notice that if $t \mapsto \rho_t$ belongs to $\car_V$, then $\rho:=dt \otimes \rho_t$ belongs to $\M(X_V)$. With a little abuse of notation, in what follows, we will denote by $\rho$ both the curve $t \mapsto \rho_t$ and the measure $dt \otimes \rho_t$.

\begin{remark}\label{rem:char}
If $\rho \in \car_V$, then
$\rho_t = h(t) \delta_{\gamma(t)}$
for $h \colon [0,1] \to [0,\infty)$ and $\gamma \colon [0,1] \to V$, where $\gamma$ is uniquely determined in the set $\{h>0\}$.%
\end{remark}

We endow the set $\cone_V$ with the flat distance on $\M (V)$, that is, for $\rho^i \in \cone_V$ we set
\begin{equation} \label{eq:flat_dual}
\D_F(\rho^1,\rho^2) := \sup\left\{\int_{V} \f  \, d(\rho^1 - \rho^2) \, \colon \, \f \in C(V),\, \norm{\f}_{\infty}\leq 1, \, \, {\rm Lip}(\f,V)  \leq 1\right\}\,. 
\end{equation}
We then define a distance over $\car_V$, by setting 
\begin{equation} \label{distance_sup}
\D(\rho^1,\rho^2) := \sup_{t\in [0,1]} \D_F(\rho^1_t,\rho^2_t)\,.
\end{equation}

\begin{remark}\label{rem:compcone}
In \cite{liero2,liero} the authors introduced the cone space over $V$ given by $C_V:=(V \times [0,\infty))/\sim$, where $\sim$ is the equivalence relationship such that the pairs $(\gamma_1,h_1)$ and $(\gamma_2,h_2)$ are identified if and only if $\gamma_1=\gamma_2$ and $h_1=h_2$, or if $h_1=h_2=0$.	 Notice that $C_V$ is in one-to-one correspondence with $\cone_V$. However in \cite{liero2,liero} the cone space is equipped with the cone distance 
\[
H^2(\rho^1,\rho^2) :=
\begin{cases}
h_1 + h_2 - 2\sqrt{h_1h_2} \cos(|\gamma_1 - \gamma_2|) & \,\text{ if } \, |\gamma_1 - \gamma_2|\leq \pi\,, \\
	h_1 + h_2 + 2\sqrt{h_1h_2} & \,\text{ otherwise,}
\end{cases}
\]
for all $\rho^1,\rho^2 \in \cone_V$. By elementary calculations, and employing \eqref{eq:char_flat} below, it is possible to show that $H^2$ and $\D_F$ induce equivalent topologies on $\cone_V$, e.g., there exists a constant $C>0$ such that
\[
\frac{1}{C} H^2(\rho^1 , \rho^2) \leq \D_F(\rho^1 ,\rho^2) \leq C \sqrt{ (h_1 + h_2) H^2(\rho^1 , \rho^2)}\,.
\]
\end{remark}

The following characterization for $\D_F$ holds. 

\begin{lemma}\label{lemma:characterizationflat}
For $\rho^1,\rho^2 \in \cone_V$ we have
\begin{equation} \label{eq:char_flat}
\D_F(\rho^1,\rho^2) = 
\begin{cases}
	|h_1 - h_2| +  \min(h_1,h_2) |\gamma_1 - \gamma_2| & \,\, \text{ if } \,\, |\gamma_1 - \gamma_2| \leq 2 \,,\\
	h_1 + h_2 & \,\, \text{ otherwise.}
\end{cases}
\end{equation}
\end{lemma}

\begin{proof}
By definition it follows that
\[
\D_F(\rho^1,\rho^2) = \sup_{c_1,c_2 \in \R} \left\{h_1 c_1 - h_2 c_2 \, \colon \, |c_1|, |c_2|\leq 1 ,\, |c_1 - c_2| \leq  |\gamma_1 - \gamma_2|\right\}\,. 
\]
By symmetry we can assume $h_1 \geq h_2$. For all $c_1,c_2 \in \R$ such that $|c_1|,|c_2| \leq 1$ and $|c_1-c_2|\leq  |\gamma_1-\gamma_2|$, we estimate
\[
\begin{aligned}
h_1 c_1 - h_2 c_2 & \leq |h_1 c_1 - h_2 c_1|+|h_2 c_1 - h_2 c_2|
 \leq |h_1 - h_2| +  \, \min(h_1,h_2) |\gamma_1 - \gamma_2|\,.
\end{aligned}
\]
The thesis follows since the supremum is achieved by
$(1 ,  1 -  |\gamma_1 - \gamma_2|)$ if $|\gamma_1- \gamma_2| \leq 2$ and by $(1,-1)$ otherwise.\end{proof}

We will now show that the metric space $(\car_V,\D)$ can be identified with $C([0,1];\cone_V)$, where $\cone_V$ is equipped with $\D_F$ and $C([0,1];\cone_V)$ inherits the relative topology as a subset of $C([0,1];\M_{\rm flat}(V))$, $\M_{\rm flat}(V)$ being the space $\M(V)$ equipped with the flat norm. In order to achieve that, we need a preliminary lemma.

\begin{lemma}\label{lem:narrow_cone}
Let $\rho_t \colon [0,1] \to \cone_V$. Then the following statements are equivalent:
\begin{enumerate}[(i)]
\item $\rho_t$ is narrowly continuous,
\item $\rho_t = h(t)\delta_{\gamma(t)}$ with $h \in C[0,1]$ and $\gamma \in C(\{h> 0\};\R^d)$.	
\end{enumerate}
\end{lemma}

\begin{proof}
	Assume (i), so that the map 
	$t \mapsto  h(t)\f (\gamma(t))$
	is continuous for each $\f \in C(V)$. By choosing $\f \equiv 1$  we conclude that $h$ is continuous. If we pick $\f(x):=x_i$ coordinate function, for all $i=1,\ldots,d$, we also infer continuity for $h\gamma$, so that $\gamma$ is continuous in $\{h>0\}$.
	Conversely, assume (ii). Let $\f \in C(V)$ and $\hat t \in [0,1]$. If $h(\hat{t})=0$, we conclude continuity of $t \mapsto  h(t)\f (\gamma(t))$ at $\hat t$ by boundedness of $\f$ and continuity of $h$, while if $h(\hat t)>0$, we conclude by (ii).%
\end{proof}

\begin{proposition} \label{prop:narrow}
Assume that $\rho_t \colon [0,1] \to \cone_V$. Then the following statements are equivalent:
\begin{enumerate}[(i)]
\item $\rho_t$ is narrowly continuous,
\item $\rho_t$ is continuous with respect to $\D_F$.
\end{enumerate}
In particular, we have that $(\car_V, \D)$ is a metric space that can be identified with $C([0,1];\cone_V)$.%
\end{proposition}

\begin{proof}
Assume (i), so that $h \in C[0,1]$ and $\gamma \in C(\{h>0\};\R^d)$ by Lemma \ref{lem:narrow_cone}. Fix $t \in [0,1]$ and $t_n \to t$. %
If $h(t)=0$, by continuity of $h$ and \eqref{eq:char_flat} we infer $\D_F(\rho_{t_n},\rho_t)=h(t_n) \to 0$. If instead $h(t)>0$, by continuity of $\gamma$ in $t$, it holds that 
 $|\gamma(t_n)-\gamma(t)|\leq 2$ for $n$ sufficiently large. 
By continuity of $h$ we conclude (ii). Conversely, assume (ii). In order to show (i), 
 we prove that $h \in C[0,1]$ and $\gamma \in C(\{h>0\};\R^d)$ (Lemma \ref{lem:narrow_cone}). From \eqref{eq:char_flat} we have
 $|h(t_1) - h(t_2)| \leq  \D_F(\rho_{t_1},\rho_{t_2})$ for all $t_1,t_2 \in [0,1]$, so that $h$ is continuous by (ii).
Let us now fix $t \in \{h >0\}$ and $t_n \to t$. Since $h(t)>0$, it is immediate to check by contradiction that $|\gamma(t_n)- \gamma(t)| \leq 2$ eventually, and hence
\begin{equation}\label{prop:narrow:3}
  |h(t_n)-h(t)| + \min (h(t_n),h(t)) 
 |\gamma(t_n) - \gamma (t)| = \D_F(\rho_{t_n},\rho_t)\,,
\end{equation}
for sufficiently large $n$. By continuity of $h$, (ii), and the assumption $h(t)>0$, we conclude continuity for $\gamma$, and hence (i). The final part of the statement follows from the first part and from the definition of $\D$.
\end{proof}

For the space $(\mathscr{S}_V,\D)$ the following holds.

\begin{proposition} \label{prop:complete}
We have that $(\mathscr{S}_V,\D)$ is a complete separable metric space.    
\end{proposition}

The above statement is somewhat classical. However, due to the lack of a reference, we provide a proof in Section \ref{app:complete}. 
We conclude this section with a useful lemma that provides sufficient conditions for continuity and measurability for scalar maps on $(\car_V,\D)$.

\begin{lemma}\label{lem:evaluationsecond}
Let $\f : V \times  [0,\infty)  \rightarrow \R$ be such that $\f(x,0) = 0$ for all $x \in V$. For $t\in [0,1]$ define the map $\Psi_t : \car_V \rightarrow \R$ by
$\Psi_t(\rho) := \f(\gamma(t), h(t))$, 
where $\rho_t=h(t)\delta_{\gamma(t)}$. If $\f$ is measurable (resp. continuous), then $\Psi_t$ is measurable (resp. continuous) with respect to $\D$.
\end{lemma}
\begin{proof}
Notice that the condition $\f(x,0)=0$ for all $x \in V$ implies that $\Psi_t$ is well defined. 
Suppose first that $\varphi$ is continuous and  
assume that $\D(\rho^n,\rho)\to 0$ as $n \to \infty$. 
By \eqref{eq:char_flat} we have $|h_n(t)-h(t)| \leq \D_F(\rho^n_t,\rho_t)$, so that $h_n(t) \to h(t)$. If $h(t)=0$, then $\rho_t=0$ and $\Psi_t(\rho)=0$. By continuity of $\f$ and compactness of $V$ we infer that $\Psi_t(\rho^n) \to 0$. If $h(t)>0$, the usual argument by contradiction implies that $|\gamma_n(t)-\gamma(t)|\leq 2$ for $n$ sufficiently large. 
Thus by \eqref{eq:char_flat} and the convergences $\min (h_n(t), h(t)) \to h(t)>0$ and $\D_F(\rho_t^n,\rho_t) \to 0$, we have that $\gamma_n(t) \to \gamma(t)$. By continuity of $\f$ we conclude $\Psi_t(\rho^n) \to \Psi_t(\rho)$.
Suppose now that $\varphi$ is measurable. 
Define the evaluation map $e_t \colon \car_V \to \cone_V$ by $e_t(\rho):=\rho_t$ and the projection $\pi \colon \cone_V \to V \times [0,\infty)$ by 
\[
\pi(h\delta_\gamma):= 
(\gamma,h) \, \rchi_{\cone_V \smallsetminus \{(0,0)\}} (\gamma,h) + (p,0) \, \rchi_{\{(0,0)\}} (\gamma,h)\,,
\]
where $p \in V$ is arbitrary but fixed. Notice that by construction $e_t$ is continuous from $(\car_V,\D)$ into $(\cone_V,\D_F)$. Additionally the map $h\delta_\gamma \mapsto (\gamma,h)$ is continuous in $\cone_V \smallsetminus \{(0,0)\}$ by repeating the above arguments. Hence $\pi$ is measurable, being sum of measurable functions. Noting that $\Psi_t= \varphi\circ \pi \circ e_t$, we see that $\Psi_t$ is measurable.  
\end{proof}

\subsection{Properties of the Hellinger-Kantorovich energy over $\cone_V$}

In this section we investigate some properties of the coercive version of the Hellinger-Kantorovich energy at \eqref{bb reg} when restricted to measures belonging to $\car_V$. To be more precise, we consider the functional $\mathscr{F} : \car_V \rightarrow [0,\infty]$ defined by
\begin{equation} \label{inf_F}
\mathscr{F}(\rho) := \inf\{J_{\alpha,\beta,\delta}(\rho,m,\mu) : (m,\mu) \in \mathcal{M}(X_V;\R^d) \times \mathcal{M}(X_V)\}\,,
\end{equation}
where $J_{\alpha,\beta,\delta}$ is defined at \eqref{bb reg} and $\alpha,\beta,\delta>0$. 
We start by introducing the subset of $\car_V$
\begin{equation} \label{eq:char_smooth}
\Ha_V :=  \left\{ \rho_t = h(t)\delta_{\gamma(t)} \in \car_V \, \colon \,   h , \, \sqrt{h} \in \AC^2[0,1], \\ 
 \sqrt{h}\gamma \in\AC^2 ([0,1];\R^d) \right\}\,.
\end{equation}
As already mentioned in the introduction, we denote by $\AC^2$ the set of absolutely continuous functions with a.e.~derivative in $L^2$ (see \cite[Section 1.1]{ags} for a precise definition).

\begin{lemma}\label{lem:hgammasolutioncontinuity}
Let $\rho_t = h(t)\delta_{\gamma(t)} \in \Ha_V$, $b \in C^1(V)$. Then $h (b \circ\gamma) \in \AC^2[0,1]$ with
\begin{equation}\label{eq:productruleeq}
(h(t)b(\gamma(t)))' =  \dot{h}(t)b(\gamma(t)) + h(t) \nabla b(\gamma(t)) \cdot \dot{\gamma}(t) \quad a.e. \ \text{in } \ (0,1)\,.
\end{equation}
\end{lemma}

\begin{proof}
By definition of $\Ha_V$, it follows that $h\gamma \in \AC^2([0,1];\R^d)$. For every $0\leq t\leq s\leq 1$
\begin{align*}
|h(t) b(\gamma(t))  & - h(s) b(\gamma(s))|  \leq {\rm Lip}(b) \, h(t) |\gamma(s) - \gamma(t)| + \|b \|_\infty |h(s) - h(t)|\\
& \leq  {\rm Lip}(b) \,  |h(s)\gamma(s) - h(t)\gamma(t)| +  \|\gamma\|_\infty  {\rm Lip}(b) \, |h(s) - h(t)|+ \|b\|_\infty |h(s) - h(t)|\,.
\end{align*}
Hence $h (b \circ\gamma) \in \AC^2[0,1]$. From the regularity assumed, we immediately infer the product rule at \eqref{eq:productruleeq} for a.e.~$t \in \{h>0\}$. Moreover, using that $h (b\circ\gamma) \in \AC^2([0,1];\R^d)$, we have $(h(t)b(\gamma(t)))' = 0$ almost everywhere in $\{h=0\}$ (\cite[Theorem 4.4]{evansgariepy}), so that \eqref{eq:productruleeq} follows. 
\end{proof}

\begin{proposition}\label{prop:sublevelJ}
Let $\rho_t = h(t)\delta_{\gamma(t)}  \in \car_V$ and $(m,\mu)  \in \mathcal{M}(X_V;\R^d) \times \mathcal{M}(X_V)$ be such that $J_{\alpha,\beta,\delta}(\rho,m,\mu) < \infty$. Then the following properties hold:
\begin{itemize}
\item[i)] There exist  $v\colon X_V \to \R^d$, $g \colon X_V \to \R$ measurable maps such that $m=v\rho$, $\mu = g\rho$,
\item[ii)] $\dot{\gamma}(t) = v(t,\gamma(t))$ for  a.e. $t \in \{h>0\}$ and $\dot h (t)= g(t,\gamma(t)) h(t)$  for a.e. $t \in (0,1)$,
\item[iii)] The curve $t \mapsto \rho_t$ belongs to $\Ha_V$. 
\end{itemize}
Moreover the energy $J_{\alpha,\beta,\delta}$ can be computed by
\begin{equation} \label{eq:J_char}
J_{\alpha,\beta,\delta}(\rho,m,\mu)=  \int_{\{h>0\}} \frac{\beta}{2} \, |\dot{\gamma}(t)|^2 h(t) + \frac{\beta \delta^2}{2} \, \frac{|\dot{h}(t)|^2}{h(t)} + \alpha h(t) \, dt\,.
\end{equation}
Conversely, let $\rho_t = h(t)\delta_{\gamma(t)}  \in \Ha_V$ and set $m:= h(t) \dot{\gamma}(t)\, dt \otimes \delta_{\gamma(t)}$, $\mu:=\dot h(t) \, dt \otimes \delta_{\gamma(t)}$. Then $(\rho,m,\mu)$ belongs to $\M_V$ and solves the continuity equation \eqref{cont weak} in $X_V$. Moreover $J_{\alpha,\beta,\delta}(\rho,m,\mu)<\infty$ and \eqref{eq:J_char} holds.  
\end{proposition}

\begin{proof}
Assume $\rho_t = h(t)\delta_{\gamma(t)}  \in \car_V$, $(m,\mu)  \in \mathcal{M}(X_V;\R^d) \times \mathcal{M}(X_V)$ and $J_{\alpha,\beta,\delta}(\rho,m,\mu) < \infty$. In particular, by definition of $J_{\alpha,\beta,\delta}$, we have that $(\rho,m,\mu)$ solves \eqref{cont weak}.  
By Lemma \ref{lem:prop B} we deduce $(i)$. We now show that the second ODE in $(ii)$ holds. By Lemma~\ref{lem:prop cont} we have $h \in BV(0,1)$, with distributional derivative given by $\pi_\#(g\rho)$, where $\pi \colon  X_V \to (0,1)$ is the projection on the time coordinate. Thus, for all $\f \in C_c^\infty(0,1)$,  
\[
\pi_\#(g\rho)(\f) = \int_0^1 \int_{V} \f(t)g(t,x) \, d\rho_t(x) \, dt  =  \int_0^1 \f(t)g(t,\gamma(t)) h(t) \, dt \,.
\]
Since $J_{\alpha,\beta,\delta}(\rho,m,\mu)<\infty$, by \eqref{formula B} and continuity of $h$, we conclude that $\dot{h}(t) = g(t,\gamma(t))h(t)$ almost everywhere and $h \in \AC^2[0,1]$. 
We will now show that the first ODE in $(ii)$ holds.
By testing \eqref{cont weak} against $\f(t,x):=a(t)b(x)$ with $a \in C_c^1(0,1)$, $b \in C^1 (V)$, we obtain  
\[
\frac{d}{dt} \int_{V} b(x) \, d\rho_t(x) = \int_{V} \left( \nabla b(x) \cdot v(t,x)  + b(x) g(t,x) \right) \, d\rho_t(x) \,, \,\, \text{ for a.e. } \,\, t \in (0,1)\,,
\]
since the right-hand side belongs to $L^2(0,1)$, thanks to Jensen's inequality, \eqref{formula B} and the assumption $J_{\alpha,\beta,\delta}(\rho,m,\mu)<\infty$. In particular, choosing $b$ as the coordinate functions, we deduce that $h\gamma  \in \AC^2([0,1];\R^d)$ with $(h\gamma)'(t)=h(t)[ v(t,\gamma(t)) + \gamma(t) g(t,\gamma(t))]$.  In particular $\gamma \in \AC^2(\{h \geq c\};\R^d)$ for every $c>0$, given that $V$ is bounded. Consider now the test function $\f \in C^1_c ((\{h>0\} \cap(0,1) ) \times V)$. Using that  $\dot{h}(t) = g(t,\gamma(t))h(t)$ almost everywhere, it is easy to check that the equation $\de_t \rho_t + \div (v \rho_t) =g \rho_t$ can be rewritten as%
\begin{equation} \label{eq:prel2}
 \int_0^1 \frac{d}{dt} \left( h(t) \f (t, \gamma(t)) \right) \, dt + \int_0^1  \, \nabla \f (t,\gamma(t)) \cdot ( v(t, \gamma(t)) - \dot{\gamma}(t) ) \, h(t) \, dt =0\,.
\end{equation}
Notice that the first integral in \eqref{eq:prel2} vanishes, as $\f$ is compactly supported. Set $\f(t,x):=a(t) x_i$ with $a \in C^1_c(\{h>0\} \cap (0,1))$ and $x_i$ coordinate function. Testing \eqref{eq:prel2} against $\f$ yields $(ii)$. 
By $(ii)$, Lemma \ref{lem:prop B},  and the energy bound, we also see that \eqref{eq:J_char} holds. 
We are left to show $(iii)$. 
First we claim that $\sqrt{h}\gamma \in \AC^2([0,1];\R^d)$. Indeed, for $\e>0$ and $\f \in C_c^\infty(0,1)$, an integration by parts yields
\begin{equation} \label{eq:prel1}
\int_0^1 h(t)\gamma(t) \frac{1}{\sqrt{h(t) + \e}} \dot\f(t)\,dt 
 = - \int_{\{h >0\}} \left[\frac{(h(t)\gamma(t))'}{\sqrt{h(t) + \e}}  - \frac{ h(t)\gamma(t) \dot h(t)}{2(h(t) + \e)^{3/2}}\right] \, \f(t) \,dt\,,
\end{equation}
where we used that $h\gamma \in \AC^2([0,1];\R^d)$, $(h\gamma)'=0$~a.e. in $\{h=0\}$ (see, e.g., \cite[Theorem~4.4]{evansgariepy}) and $(h\gamma)'=\dot h \gamma + h \dot \gamma$~a.e. in $\{h>0\}$. 
By \eqref{eq:J_char}, continuity of $h$, boundedness of $V$, we can invoke dominated convergence and pass to the limit as $\e \to 0$ in \eqref{eq:prel1}, thus concluding that $\sqrt{h}\gamma \in \AC^2([0,1];\R^d)$ with derivative given by 
$2^{-1}\rchi_{\{h>0\}} \dot{h}\gamma /\sqrt{h} + \sqrt{h} \dot \gamma$.
A similar argument shows that $\sqrt{h} \in \AC^2[0,1]$ with derivative given by 
$2^{-1}\rchi_{\{h>0\}}\dot{h}/\sqrt{h}$, concluding the proof of $(iii)$ and of the direct implication.

Conversely, assume that $\rho_t = h(t)\delta_{\gamma(t)}  \in \Ha_V$ and set $m:= h(t) \dot{\gamma}(t)\, dt \otimes \delta_{\gamma(t)}$, $\mu:=\dot h(t) \, dt \otimes \delta_{\gamma(t)}$. It is clear that $(\rho,m,\mu) \in \M_V$, as a consequence of the regularity on $h$ and $\gamma$. We claim that $(\rho,m,\mu)$ solves \eqref{cont weak} in $X_V$. Fix $b \in C^1(V)$. By Lemma \ref{lem:hgammasolutioncontinuity} we have that $h (b \circ\gamma) \in \AC^2[0,1]$ and \eqref{eq:productruleeq} holds. 
Thus, for all $a \in C^1_c(0,1)$, 
\begin{align*}
\int_{X_V} a' (t)  b(x)\, d\rho_t(x) \, dt  
=  -\int_{X_V} a(t)b(x) \, d\mu - \int_{X_V} a(t) \nabla b(x) \cdot  \,dm \,.
\end{align*}
Therefore, by employing a standard density argument, $(\rho,m,\mu)$ solves \eqref{cont weak} in $X_V$. Finally, by the regularity of $h$, $\gamma$ and \eqref{formula B}, we conclude that $J_{\alpha,\beta,\delta}(\rho,m,\mu)<\infty$ and \eqref{eq:J_char} holds.  
\end{proof}

\begin{proposition} \label{lem:compact_sublevels}
Let  $\mathscr{F} \colon (\car_V,\D) \to [0,\infty]$ be the functional defined at \eqref{inf_F}.  The domain of $\mathscr{F}$ is given by $\Ha_V$, where we have%
\begin{equation} \label{eq:repr_F}
\mathscr{F}(\rho) = \mathscr{F}(\gamma,h) = \int_{\{h>0\}}\frac{\beta}{2} \, |\dot{\gamma}(t)|^2 h(t) + \frac{\beta \delta^2}{2} \, \frac{|\dot{h}(t)|^2}{h(t)} +\alpha h(t) \, dt \,.
\end{equation}
Moreover $\mathscr{F}$ is lower semi-continuous and its sublevel sets are compact.
\end{proposition}

\begin{proof}
We start by showing that the domain of $\mathscr{F}$ is given by $\Ha_V$ and that \eqref{eq:repr_F} holds. Assume first that $\rho^* \in \car_V$ and $\mathscr{F}(\rho^*)<\infty$. We claim that exists a pair $(m^*,\mu^*) \in \M(X_V;\R^d) \times \M(X_V)$ such that
\begin{equation} \label{eq:direct_meth}
\mathscr{F}(\rho^*)=J_{\alpha,\beta,\delta}(\rho^*,m^*,\mu^*)\,.
\end{equation}
Indeed the functional $(m,\mu) \mapsto J_{\alpha,\beta,\delta}(\rho,m,\mu)$  is weak* lower semi-continuous by Lemma \ref{lem:prop J}. Invoking \eqref{lem:prop J est} and the direct method, we conclude that the infimum at \eqref{inf_F} is achieved, showing \eqref{eq:direct_meth}. %
 Hence we can apply the direct implication of Proposition \ref{prop:sublevelJ} to $(\rho^*,m^*,\mu^*)$ to obtain that $\rho^* \in \Ha_V$ and that \eqref{eq:repr_F} holds.  Conversely, assume that $\rho_t^*=h(t)\delta_{\gamma(t)} \in \Ha_V$ and set $m:=\dot{\gamma} \rho^*$, $\mu:=(\dot{h}/h) \rho^*$. By the converse implication of Proposition \ref{prop:sublevelJ} we know that $(\rho^*,m,\mu) \in \M_V$ and $J_{\alpha,\beta,\delta}(\rho^*,m,\mu)<\infty$, from which we infer $\mathscr{F}(\rho^*)<\infty$. Thus there exists a pair $(m^*,\mu^*) \in \M(X_V;\R^d) \times \M(X_V)$ such that \eqref{eq:direct_meth} holds. An application of the direct implication of Proposition \ref{prop:sublevelJ} to $(\rho^*,m^*,\mu^*)$ yields \eqref{eq:repr_F}.

We now prove that $\mathscr{F}$ is lower semi-continuous with respect to $\D$. To this end, assume that $\D(\rho^n,\rho) \to 0$ as $n \to \infty$. We claim that $dt \otimes \rho_t^n \weakstar dt \otimes \rho_t$ weakly* in $\M(X_V)$. By density, it is sufficient to prove convergence for test functions $\f(t,x) = a(t)b(x)$ with $a \in C_c(0,1), b \in C(V)$. Moreover, it is not restrictive to assume that $\nor{b}_\infty \leq 1$. For a fixed $\e>0$ there exists $c \in C^1(V)$ such that $\nor{c}_\infty \leq 1$ and $\nor{b-c}_{\infty} < \e$. For $t \in [0,1]$ we have
\[
\begin{aligned}
\left| \int_V b(x) \,d(\rho_t^n-\rho_t) \right| & \leq \nor{b-c}_\infty (\nor{\rho^n_t}_{\M(V)} + \nor{\rho_t}_{\M(V)}) + 
  \left| \int_V c(x) \,d(\rho_t^n-\rho_t) \right|\\
 & \leq \e (\D(\rho^n,0) + \D(\rho,0)) + \Lip(c) \, \D(\rho^n,\rho) \leq \e C +  \Lip(c) \, \D(\rho^n,\rho) \,
\end{aligned} 
\]
where the first term in the first line was estimated by \eqref{eq:char_flat}, and the second one by \eqref{eq:flat_dual}. Since the estimate does not depend on $t$, and $\e$ is arbitrary, we conclude that $dt \otimes \rho_t^n \weakstar dt \otimes \rho_t$. We now claim that $\mathscr{F}$ is weak* lower semi-continuous in $\car_V$ considered as a subset of $\M(X_V)$: Indeed assume that $\rho^n \weakstar \rho$ in $\M(X_V)$. Without loss of generality we can assume that $\sup_n \mathscr{F}(\rho_n)<\infty$ along a subsequence, so that there exist $(m^n,\mu^n) \in \mathcal{M}(X_V;\R^d) \times \mathcal{M}(X_V)$ such that, up to subsequences, $\mathscr{F}(\rho^n)=J_{\alpha,\beta,\delta}(\rho^n,m^n,\mu^n)$. 
 By \eqref{lem:prop J est} we infer the existence of a pair $(m,\mu)$ such that, up to subsequences, $m^n \weakstar m$, $\mu^n \weakstar \mu$.
 We can now invoke weak* lower semi-continuity of $J_{\alpha,\beta,\delta}$ (Lemma \ref{lem:prop J}) to conclude weak* lower semi-continuity of $\mathscr{F}$.  
Since $dt \otimes \rho_t^n \weakstar dt \otimes \rho_t$ in $\M(X_V)$ whenever $\D(\rho^n,\rho)\to 0$, we infer lower semi-continuity of $\mathscr{F}$ with respect to $\D$.

Finally, we show that the sublevel sets of $\mathscr{F}$ are compact with respect to $\D$. As $\mathscr{F} \geq 0$ and is positively one-homogeneous, it is enough to show that $S_\mathscr{F} :=\{\rho \in \car_V:  \mathscr{F}(\rho) \leq 1\}$ is compact. 
Let $\rho_t = h(t) \delta_{\gamma(t)} \in S_{\mathscr{F}}$, so that, in particular, $\rho \in \Ha_V$. In order to show compactness of $S_{\mathscr{F}}$ we first provide some preliminary estimates for the maps $h$ and $h\gamma$. By \eqref{eq:repr_F} we immediately infer that $\|h\|_{1} \leq 1/\alpha$.  Let $0\leq t_1 \leq t_2 \leq 1$. There holds
\begin{equation}
\begin{aligned}
h(t_2) - h(t_1) & \leq \int_{t_1}^{t_2} |\dot{h}(s)| ds 
 = \int_{(t_1,t_2) \cap \{h>0\}} |\dot{h}(s)| ds 
 = \int_{(t_1,t_2) \cap\{h>0\}} \frac{|\dot{h}(s)|}{\sqrt{h(s)}}\sqrt{h(s)} ds \\
&  \leq \left(\int_{\{h>0\}} \frac{|\dot{h}(s)|^2}{h(s)} ds\right)^{1/2} \left(\int_{t_1}^{t_2}  h(s) ds \right)^{1/2} 
 \leq \frac{2}{\beta \delta^2}\, \left(\int_{t_1}^{t_2}  h(s) ds \right)^{1/2}  \,,
 \label{eq:holderinsublevel}
\end{aligned}
\end{equation}
where we used that $\dot{h} = 0$ almost everywhere in $\{h=0\}$ (\cite[Theorem 4.4]{evansgariepy}), H\"older's inequality, and the fact that $\mathscr{F}(\rho) \leq 1$ in conjunction with \eqref{eq:repr_F}. Since $h \geq 0$, choosing $t_1 \in \argmin h$ in the above estimate yields 
\begin{equation} \label{eq:infinity_bound}
\nor{h}_{\infty} \leq \frac{2}{\beta \delta^2} \, \nor{h}_{1}^{1/2} + \nor{h}_{1} \leq C \,, \qquad \norm{h \gamma}_{\infty} \leq C R\,, 
\end{equation}
where $R:=\max\{ |p| \, \colon \, p \in V\}$, $C:=2/(\beta\delta^2 \sqrt{\alpha})+ 1/\alpha$. Recall that $R <\infty$ as $V$ is bounded. 
Thus, by \eqref{eq:holderinsublevel} and \eqref{eq:infinity_bound}, 
\begin{equation} \label{eq:holderinsublevel2}
|h(t_2) - h(t_1)| \leq \left(\int_{t_1}^{t_2}  h(s) ds \right)^{1/2} \leq C \, |t_1-t_2|^{1/2} \quad \text{ for all } \quad 0\leq t_1 \leq t_2 \leq 1\,.
\end{equation}
Moreover, by \eqref{eq:holderinsublevel}-\eqref{eq:infinity_bound} we can estimate
\[
\int_{t_1}^{t_2} |\dot{h}(s)\gamma(s)|ds   \leq R\int_{t_1}^{t_2} |\dot{h}(s)|ds 
\leq R \left(\int_{t_1}^{t_2}  h(s) ds \right)^{1/2} 
\leq CR |t_1 - t_2|^{1/2} \,.
\]
and also
\begin{align*}
\int_{t_1}^{t_2} |h(s)\dot{\gamma}(s)|ds  
& \leq  \left(\int_{t_1}^{t_2} h(s)ds \right)^{1/2}  \left(\int_{0}^{1} |\dot{\gamma}(s)|^2h(s)ds\right)^{1/2} 
 \leq \frac{2C}{\beta} \, |t_1 - t_2|^{1/2}\,,
\end{align*}
where we used H\"older's inequality, \eqref{eq:repr_F},  \eqref{eq:holderinsublevel2}, and $\mathscr{F}(\rho) \leq 1$. By Lemma \ref{lem:hgammasolutioncontinuity} and the above estimates we thus infer
\begin{equation} \label{eq:holderinsublevel3}
|h(t_1)\gamma(t_1) - h(t_2)\gamma(t_2)|  %
 \leq \int_{t_1}^{t_2} |\dot{h}(s)\gamma(s)|ds +  \int_{t_1}^{t_2} |h(s)\dot{\gamma}(s)|ds \leq C(R+2\beta^{-1}) |t_1-t_2|^{1/2}
\end{equation}
for every $0\leq t_1 \leq t_2 \leq 1$. Hence, considering a sequence $\{\rho^n\}_n$ in $S_{\mathscr{F}}$ with $\rho_t^n= h_n(t) \delta_{\gamma_n(t)}$, by \eqref{eq:infinity_bound}- \eqref{eq:holderinsublevel3} we have that $h_n$ and $h_n \gamma_n$ are equibounded and equicontinuous. Therefore Ascoli-Arzel\`a's theorem implies that, up to subsequences, $h_n \rightarrow h$ and $\gamma_n h_n \rightarrow f$ uniformly, where $h \in C[0,1]$, $h \geq 0$ and $f \in C([0,1];\R^d)$.  Define $\gamma(t) := f(t)/h(t)$ if $h(t) > 0$. By the uniform convergence $h_n\rightarrow h$ we have that $\gamma(t)  \in V$ for  $t \in \{h>0\}$. %
Therefore, by setting
$\rho_t :=  h(t) \delta_{\gamma(t)}$, Lemma \ref{lem:narrow_cone} implies that $\rho \in \car_V$. Since $h_n \to h$ pointwise and $\gamma_n \to \gamma$ pointwise in $\{h>0\}$, and since $\nor{h_n}_{\infty} \leq C$, by dominated convergence one immediately concludes that $dt \otimes \rho^n_t \weakstar dt \otimes \rho_t$ in $\M(X_V)$. We can then invoke the weak* lower semi-continuity of $\mathscr{F}$ to conclude that $\rho \in S_{\mathscr{F}}$. We are left to prove that $\rho^n  \to \rho$ with respect to $\D$. 
Fix $\e>0$. By the uniform convergences $h_n \to h$ and $h_n\gamma_n \to h \gamma$, there exists $N(\e) \in \N$ such that 
\begin{equation} \label{eq:sublevels_eps}
|h_n(t) - h(t)| < \frac{\e}{R} \,, \quad 
|h_n(t)\gamma_n(t) - h(t) \gamma(t)|< \e  \,, \quad
\text{ for all } \,\, n \geq N(\e)\,, \,\, t \in [0,1] \,.
\end{equation}
Let $t \in \{h \geq \e\}$ and $n \geq N(\e)$. Using the above condition we infer
\[
|\gamma_n(t)-\gamma(t)| \leq 
\left|  \frac{h_n(t)}{h(t)} \gamma_n(t) - \gamma(t)   \right| + 
|\gamma_n(t)| \left|  \frac{h_n(t)}{h(t)} - 1   \right|
< \frac{\e}{h(t)} + R \, \frac{\e}{R h(t)} \leq 2 \,.
\]
Set $m_n(t):=\min(h_n(t),h(t))$. Then, by \eqref{eq:char_flat},
\[
\begin{aligned}
\D_F(\rho_t^n,\rho_t) & < \frac{\e}{R} + m_n(t) \, |\gamma_n(t)-\gamma(t)| \\
& \leq \frac{\e}{R} + |\gamma_n(t)| \, | m_n(t) - h_n(t)| + 
 |h_n(t)\gamma_n(t) - h(t) \gamma(t)| + 
 |\gamma(t)| \, |m_n(t) - h(t)| \\
& \leq \frac{\e}{R} + 2 R |h_n(t)-h(t)| +  |h_n(t)\gamma_n(t) - h(t) \gamma(t)| < (R^{-1} + 3) \e \,.
\end{aligned}
\]
Let now $t \in \{h \leq \e\}$. By triangle inequality and \eqref{eq:char_flat}, \eqref{eq:sublevels_eps}
\[
\D_F (\rho_t^{n},\rho_t)
  \leq  h_n(t) + h(t) \leq 
  |h_{n}(t)-h(t)| + 2h(t) 
   \leq  \e (R^{-1}+2) \,.
\]
In total we infer $\D(\rho^n,\rho) < C \e$ for $n \geq N(\e)$, concluding the proof.
\end{proof}

\section{The main decomposition theorem} \label{sec:main_thm}

In this section we will prove the decomposition result in Theorem \ref{thm:intro_main} anticipated in the introduction. Specifically, the proof is presented in Sections \ref{sec:proof_converse}, \ref{sec:proof_direct}, while Section \ref{sec:regularized} contains auxiliary results which are instrumental to the proof.

For reader's convenience we will recall a few notations and the statement of Theorem \ref{thm:intro_main}. Let $d \in \N$, $d\geq 1$ and $V \subset \R^d$ be the closure of a bounded domain of $\R^d$. We denote the time-space cylinder by $X_V :=(0,1) \times V$. We also recall the definitions of $\cone_V$ and $\car_V$ at \eqref{eq:cone}-\eqref{eq:narr}.
The set $\cone_V$ is equipped with the flat metric $\D_F$ defined at \eqref{eq:flat_dual}, while $\car_V$ is equipped with the supremum distance $\D$ defined at \eqref{distance_sup}. We remind the reader that $(\car_V,\D)$ is a complete metric space (Proposition \ref{prop:complete}). Moreover we will also consider the set $\Ha_V$ introduced at \eqref{eq:char_smooth}. 
Let $v\colon X_V \to \R^d$, $g \colon X_V \to \R$ be given measurable maps and consider the system of ODEs
\begin{align}
&\dot{\gamma}(t) = v(t,\gamma(t)) \,\, \text{ a.e. in } \, \{h>0\}\,,\label{ODE1} \tag{O1} \\
& \dot h(t) = g(t,\gamma(t)) h(t)\,\, \text{ a.e. in }\,  (0,1) \,. \label{ODE2}	\tag{O2}
\end{align}

For $v$ and $g$ as above, we define the following subset of $\Ha_V$:
\begin{equation}
	 \label{def:havg}
\Ha_V^{v,g} := \left\{ \rho_t = h(t)\delta_{\gamma(t)}  \in \Ha_V \, \colon \,   (h,\gamma) \, \text{ satisfy } \, \eqref{ODE1}-\eqref{ODE2} \right\} \,. 
\end{equation}
Also define the subset of $\car_V$ 
\begin{equation}
\mathcal{H}^1_V := \{(h,\gamma)\in \car_V : \|h\|_{1} = 1\}\,.  \label{eq:defh1}
\end{equation}
Finally, define the subset of $\M^+(\car_V)$:
\[
\M^+_1 (\car_V):= \left\{ \sigma \in \M^+(\car_V) \, \colon \,
\int_{\car_V} \|h\|_\infty \, d\sigma (\gamma,h) <\infty \right\}\,,
\]
where the notation $d\sigma(\gamma,h)$ is a shorthand for expressing that the integral is computed on all curves $\rho_t =h(t) \delta_{\gamma(t)} \in \car_V$.

\begin{definition} \label{def:building}
For a measure $\sigma \in \M^+(\car_V)$ we define the set function $\rho_t^\sigma$ as
\begin{equation} \label{superposition}
\rho^\sigma_t(E) := \int_{\car_V}  h(t)\,\rchi_E(\gamma(t))\,  d\sigma(\gamma,h) 
\end{equation}
for all Borel sets $E \subset V$  and $t \in [0,1]$. 
\end{definition}

\begin{remark} \label{rem:representationL1}
The map $(\gamma,h) \mapsto h(t)\rchi_E(\gamma(t))$ at \eqref{superposition} is measurable in $(\car_V,\D)$ by Lemma \ref{lem:evaluationsecond}; therefore the integral is well defined, possibly unbounded. Assume in addition that $\sigma \in \mathcal{M}_1^+(\car_V)$. It is easy to check that $\rho_t^\sigma$ at \eqref{superposition} belongs to $\M^+(V)$ for all $t \in [0,1]$. Moreover, if  $\f \in L^1_{\rho^\sigma_t}(V)$ for some fixed $t \in [0,1]$, then the map $(\gamma,h) \mapsto h(t)\f(\gamma(t))$ belongs to $L^1_\sigma(\car_V)$ and
\begin{equation}\label{eq:representationL1}
\int_V \f(x) \, d \rho_t^\sigma(x) = \int_{\car_V}  h(t)\,\f(\gamma(t))\,  d\sigma(\gamma,h) \,.
\end{equation}
This fact can be shown by mimicking the proof of \cite[Theorem 3.6.1]{bogachev}, in conjunction with Lemma~\ref{lem:evaluationsecond}. 
Similarly, if $\f \colon V \to \R \cup \{\pm \infty\}$ is measurable and the map $(\gamma,h) \mapsto h(t)\f(\gamma(t))$ belongs to $L^1_{\sigma}(\car_V)$, then $\f \in L^1_{\rho_t^\sigma}(V)$ and \eqref{eq:representationL1} holds.
\end{remark}

We are now ready to state the main decomposition result of the paper.

\begin{theorem} \label{thm:lifting}
Assume that $\Om \subset \R^d$ is the closure of a bounded domain, with $d \in \N$, $d \geq 1$. 
Let $\rho_t \in \pcurves$ be a measure solution of the continuity equation $\de_t \rho_t + \div (v \rho_t) = g \rho_t$ in $X_\Om$
 in the sense of \eqref{cont weak}, for some measurable maps $v \colon X_\Om \to \R^d$, $g \colon X_\Om \to \R$ satisfying
\begin{equation}\label{thm:lifting:1}
\int_0^1 \int_{\Om} |v(t,x)|^2 + |g(t,x)|^2 \, d\rho_t (x) \, dt< \infty \,.
\end{equation}
Then there exists a measure $\sigma \in \M_1^+(\car_\Om)$ concentrated on $\Ha_\Om^{v,g} \cap \Ha_\Om^1$ and such that $\rho_t = \rho_t^\sigma$ for all $t \in [0,1]$, where $\rho_t^\sigma$ is defined at \eqref{superposition}, that is,
\begin{equation} \label{eq:representation}
\int_\Om \f(x) \, d \rho_t(x) =\int_{\car_\Om} h(t)\f(\gamma(t))\, d\sigma(\gamma,h)  \,\,\, \text{ for all } \,\, \, \f \in C(\Om)\,.
\end{equation}
Conversely, assume that $\sigma \in \M^+(\car_\Om)$ is concentrated on $\Ha_\Om^{v,g}$ and satisfies the bound
\begin{equation} \label{eq:conv_bound}
\int_0^1 \int_{\car_\Om} h(t) \left(1+ | v(t,\gamma(t))| + |g(t,\gamma(t))|  \right) \, d\sigma(\gamma,h)\, dt < \infty\,.
\end{equation}
Then $\sigma$ belongs to $\M^+_1(\car_\Om)$ and $\rho_t^\sigma$ defined by \eqref{superposition} belongs to $\pcurves$ and satisfies $\de_t \rho_t^\sigma + \div (v \rho^\sigma_t) = g \rho^\sigma_t$ in $X_\Om$.
\end{theorem}

\begin{remark} \label{rem:conv_bound}
Condition \eqref{eq:conv_bound} is natural in the following sense. If $\rho_t$ satisfies the assumptions of Theorem \ref{thm:lifting}, then in particular the map $\f(t,x):=1 + |v(t,x)|+|g(t,x)|$ belongs to $L^1_{\rho_t^\sigma}(\Om)$ for a.e.~$t \in (0,1)$, thanks to \eqref{thm:lifting:1},  \eqref{eq:representation} and narrow continuity of $\rho_t$. Therefore, by applying Remark~\ref{rem:representationL1}, we see that the measure $\sigma$ representing $\rho_t$ satisfies \eqref{eq:conv_bound}.
\end{remark}

\subsection{Proof of the converse implication of Theorem \ref{thm:lifting}}\label{sec:proof_converse}

We now prove the converse statement in Theorem \ref{thm:lifting}. To this end, assume that $\sigma \in \M^+(\car_\Om)$ is concentrated on $\Ha_\Om^{v,g}$ and \eqref{eq:conv_bound} holds.
Let us first show that $\sigma \in \M^+_1(\car_\Om)$. Let $\rho_t =h(t) \delta_{\gamma(t)} \in \car_\Om$ and $t^* \in \argmin h$, which exists by continuity of $h$ (see Lemma \ref{lem:narrow_cone}). 
Using the definition of $\Ha_\Om^{v,g}$ we can estimate  
\[
	h(t)  =  h(t^*) + \int_{t^*}^{t} \dot h (\tau) \, d\tau 
	\leq \int_0^1 h(\tau) \, d\tau + \int_{t^*}^{t} g(\tau,\gamma(\tau))h(\tau) \, d\tau \, \,\,\,\, \sigma\text{-a.e. in } \car_\Om\,,
\]
	for all $t \in [0,1]$. In particular,
	\begin{equation} \label{thm:lifting:3}
	\norm{h}_{\infty} \leq \int_0^1 h(t) (1 + |g(t,\gamma(t))|) \, dt \, \,\,\,\, \sigma\text{-a.e. in } \car_\Om\,,
	\end{equation}
	concluding that $\sigma \in \M^+_1(\car_\Om)$, thanks to \eqref{eq:conv_bound}. 
		We now show that the curve $t \mapsto \rho_t^\sigma$ defined by \eqref{superposition} belongs to $\pcurves$.
	First, Remark~\ref{rem:representationL1} implies that $\rho_t^\sigma \in \M^+(\Om)$ for all $t \in [0,1]$.  
For the narrow continuity, fix $\f \in C(\Om)$ and notice that by definition the map $t \mapsto h(t)\f(\gamma(t))$ is continuous for all $\rho_t =h(t) \delta_{\gamma(t)} \in \car_\Om$. Since $\sigma \in \M^+_1(\car_\Om)$ we can apply dominated convergence and conclude that also 
$t \mapsto \int_\Om \f(x)\, d\rho_t^\sigma(x)$ 
is continuous. 	We are left to show that  $\rho^\sigma$ solves the continuity equation $\de_t \rho^\sigma_t + \div (v \rho^\sigma_t) = g \rho^\sigma_t$ in $X_\Om$. To this end, fix $b \in C^1(\Om)$.
	By Lemma~\ref{lem:hgammasolutioncontinuity} the map $t\mapsto h(t)b(\gamma(t))$ is differentiable almost everywhere and \eqref{eq:productruleeq} holds. Therefore, for all $0 \leq s \leq t \leq 1$ the following holds
	\[
	\begin{aligned}
	 \int_{\Om} b \, d\rho^\sigma_t -  \int_{\Om} b  \, d\rho^\sigma_s & = \int_{\car_\Om} \int_s^t \frac{d}{d\tau} [h(\tau)b(\gamma(\tau))] \, d\tau \, d\sigma(\gamma,h)  \\         
	&  =\int_{\car_\Om} \int_s^t   \dot h(\tau) b(\gamma(\tau)) + h(\tau )\nabla b(\gamma(\tau)) \cdot \dot \gamma(\tau)   \, d\tau \, d\sigma(\gamma,h)\\
	& =  \int_s^t \int_{\car_\Om} h(\tau) \left[ b(\gamma(\tau)) g(\tau,\gamma(\tau))  + \nabla b(\gamma(\tau)) \cdot v(\tau,\gamma(\tau))    \right] \,  d\sigma(\gamma,h)\,  d\tau \,,
	\end{aligned}
	\]
where in the last equality we used that $\sigma$ is concentrated on $\Ha_\Om^{v,g}$ and  applied Fubini's Theorem, which we are allowed to do as the integrand is absolutely integrable by \eqref{eq:conv_bound}, triangle inequality, and the fact that $b \in C^1(\Om)$.  
In particular, the map $t \mapsto \int_{\Om} b(x) \, d\rho^\sigma_t(x)$ is absolutely continuous with almost everywhere derivative given by
\begin{equation} \label{thm:lifting:100}
\begin{aligned}
\frac{d}{dt}\int_{\Om } b(x)\, d\rho_t^\sigma(x) & = \int_{\car_\Om} h(t) \left[b(\gamma(t)) g(t,\gamma(t))  + \nabla b (\gamma(t)) \cdot v(t,\gamma(t)) \right] \, d\sigma(\gamma,h)  \\
& =  \int_{\Om} b(x) g(t,x)  + \nabla b(x) \cdot v(t,x) \, d\rho_t^\sigma(x)\,.
\end{aligned}
\end{equation}
The second equality in \eqref{thm:lifting:100} follows because $v$ and $g$ are measurable and hence $\Psi (t,x):= b(x)g(t,x) + \nabla b(x)\cdot v(t,x)$ is measurable in $\Om$ for a.e.~$t$ fixed. From \eqref{eq:conv_bound} we have that $(\gamma,h) \mapsto h(t)\Psi(t,\gamma(t))$ belongs to $L^1_\sigma(\car_\Om)$ for a.e.~$t$, and hence by Remark~\ref{rem:representationL1} we can apply \eqref{eq:representationL1} to $\Psi(t,\cdot)$ and obtain the second equality in \eqref{thm:lifting:100}.   Identity \eqref{thm:lifting:100} implies that 
$\rho^\sigma_t$ solves the continuity equation in $X_\Om$ in the sense of \eqref{cont weak}, for all $\f \in C^1_c(X_\Om)$ of the form $\f (t,x) = a(t)b(x)$ for $a \in C^1_c(0,1), b \in C^1(\Om)$, and hence, by density, for all the elements of $C^1_c(X_\Om)$.

\subsection{Regularized solutions of the continuity equation} \label{sec:regularized} 
Before starting the proof of the direct statement in Theorem \ref{thm:lifting}, we provide some smoothing arguments which will be employed to construct the measure $\sigma$. To this end, let $\Om \subset \R^d$, $d \in \N$, $d \geq 1$ be the closure of a bounded domain. Let $v: X_\Om \rightarrow \R^d$, $g:X_\Om \rightarrow \R$ be given measurable maps, and $\rho_t \in \pcurves$ be such that $\de_t \rho_t + \div (v \rho_t) = g \rho_t$ in $X_\Om$ in the sense of \eqref{cont weak}. %
We extend $v,g$ to zero to the space $(0,1) \times \R^d$. Similarly extend $\rho_t$ to zero so that $\rho_t \in \M^+(\R^d)$. Notice that the extensions $(\rho, v,  g)$ satisfy the continuity equation in $(0,1) \times \R^d$, due to the no-flux boundary conditions. 
For $x \in \R^d$, $r>0$ let $B_r(x):=\{x \in \R^d \, \colon \, |x|<r\}$ and let $\xi \in C^\infty(\R^d)$ be such that $\xi \geq 0$, $\supp \xi \subset B_1(0)$ and $\int_{\R^d} \xi \, dx=1$.
For every $0<\e < 1$ and $x \in \R^d$ set $\xi_\e (x):= \e^{-d} \xi (x \e^{-1})$. Note that $\supp \xi_\e \subset B_\e(0)$. 
Let $R>0$ be such that 
\begin{equation} \label{def_of_R}
\{x \in \R^d : \dist(x,\Omega) \leq  2\} \subset V\,\, ,
\,\,\,\, \,\, V:=\overline{B_R(0)} \,,
\end{equation}
and define 
\begin{equation}\label{eq:regularizations}
\rho_t^\e := (\rho_t \ast \xi_\e ) + \eta_\e \,  ,  \quad \eta_\e := \e \rchi_{V}\,, \quad v_t^\e := \frac{(v_t\rho_t)\ast \xi_\e}{\rho^\e_t}, \quad g_t^\e := \frac{(g_t\rho_t)\ast \xi_\e}{\rho_t^\e }\, , 
\end{equation}
where $v_t^\e$ and $g_t^\e$ are set to be zero in the region where $\rho_t^\e (x)=0$, i.e., in $(0,1) \times (\R^d \smallsetminus V)$. Here, with a little abuse of notation, we denote $v_t=v(t,\cdot)$, $v_t^\e=v^\e(t,\cdot)$, $g_t=g(t,\cdot)$, $g^\e_t=g^\e(t,\cdot)$. %

\begin{lemma}\label{lem:convolution}
Let $\rho_t \in \pcurves$ and $v: X_\Om \rightarrow \R^d$, $g:X_\Om \rightarrow \R^d$ be measurable. Suppose that $\de_t \rho_t + \div (v \rho_t) = g \rho_t$ in $X_\Om$ in the sense of \eqref{cont weak} and that \eqref{thm:lifting:1} holds.
Let $(\rho^\e_t, v^\e_t,g_t^\e)$ be defined as in \eqref{eq:regularizations}.
Then $(\rho^\e_t \,dx, v^\e_t,g_t^\e)$ is a solution to $\de_t \rho^\e_t \,dx + \div(v_t^\e \rho_t^\e \,dx) = g_t^\e \rho_t^\e \, dx$ in $(0,1) \times \R^d$ and $\rho_t^\e \, dx \to \rho_t$ narrowly in $\M(V)$ as $\e \to 0$, for all $t \in [0,1]$. Moreover $v^\e$ and $g^\e$ satisfy \eqref{eq:comp1} and \eqref{eq:comp2}, respectively.
Finally, for every $t \in [0,1]$ there holds  
\begin{equation} \label{eq:est_energy_reg_g}
\begin{gathered}
 \int_{\R^d}  |v^\e(t,x)|^2  \, \rho^\e_t(x) \, dx \leq   
\int_{\Om}  |v(t,x)|^2 \, d\rho_t(x)    \,, \\
\int_{\R^d}  |g^\e(t,x)|^2 \, \rho^\e_t(x)\, dx \leq   
\int_{\Om}  |g(t,x)|^2 \, d\rho_t(x)  \,. 
\end{gathered}
\end{equation}
\end{lemma}

\begin{proof}
By the interplay between weak differentiation and mollification, it is immediate to check that $(\rho^\e_t\,dx , v_t^\e,g_t^\e)$ solves the continuity equation in $(0,1) \times \R^d$ for all $0<\e <1$. The fact that $\rho_t^\e \,dx \to \rho_t$ narrowly is an immediate consequence of the properties of convolutions and of the convergence $\eta_\e \to 0$ as $\e \to 0$. %
We now prove that $v^\e$ satisfies \eqref{eq:comp1}. 
Notice that by definition $v_t^\e(x) = 0$ in $\R^d \smallsetminus (\overline{\Omega + B_1(0)})$ for every $t \in [0,1]$. Moreover $\rho_t^\e \geq \e$ in $V$ for all $t$. Therefore
\[
\int_0^1 \sup_{x\in \R^d} |v^\e(t,x)|\, dt %
\leq \frac{1}{\e} \int_0^1 \sup_{x\in\overline{\Omega + B_1(0)}} |(v_t \rho_t) \ast \xi_\e (x)|\, dt \\
 \leq \frac{1}{\e^{d+1}} \int_0^1 \int_{\R^d} |v_t(y)| d\rho_t(y)\, dt < \infty \\
\]
by \eqref{thm:lifting:1}. By direct calculation 
$
\nabla v_t^\e =  [((v_t\rho_t)\ast \nabla \xi_\e)\rho_t^\e -  ((v_t\rho_t)\ast  \xi_\e) (\rho_t \ast \nabla \xi_\e)] /(\rho^\e_t)^2\,,
$
so that
\begin{align*}
|\nabla v_t^\e| & \leq {\e^{-1}} |(v_t\rho_t)\ast \nabla \xi_\e| + \e^{-2} |(v_t\rho_t)\ast  \xi_\e)| |\rho_t \ast \nabla \xi_\e|  \\
& \leq \lbrack \e^{-1} \nor{\nabla \xi_\e}_\infty + \e^{-2} \nor{\xi_\e}_\infty \nor{\nabla \xi_\e }_\infty \rho_t(\Om) \rbrack \int_\Om |v_t(x)| \, d\rho_t(x) \,.
\end{align*}
As $t \mapsto \rho_t(\Om)$ is continuous, the quantity $C(\rho):=\max_t |\rho_t(\Om)|$ is well defined. Therefore 
\begin{align*}
\int_0^1 \text{Lip}(v^\e(t,\cdot), \R^d)\, dt = \int_0^1  \sup_{x\in \R^d}|\nabla v_t^\e(x)| dt \leq C(\e)C(\rho)%
\int_0^1\int_{\Om} v_t(x) d\rho_t(x) dt < \infty\,,
\end{align*}
where the last term is finite by \eqref{thm:lifting:1}. 
By similar computations and by \eqref{thm:lifting:1}, one can easily show that $g^\e$ satisfies \eqref{eq:comp2}. %
We now prove the first estimate in \eqref{eq:est_energy_reg_g}. Fix $t \in [0,1]$. If $\rho_t = 0$, there is nothing to prove. Otherwise we have 
\[
\begin{aligned} 
\int_{\R^d}  |v^\e(t,x)|^2   \, \rho^\e_t(x) \, dx  
& =  \int_{\Omega + B_\e(0)} \frac{|(v_t \rho_t) \ast \xi_\e|^2}{\rho_t \ast \xi_\e + \e}  \, dx 
\leq  \int_{\Omega + B_\e(0)} \frac{|(v_t \rho_t) \ast \xi_\e|^2}{\rho_t \ast \xi_\e }  \, dx \\
& =  \int_{\R^d} \left|\frac{(v_t \rho_t) \ast \xi_\e}{\rho_t \ast \xi_\e}\right|^2  (\rho_t \ast \xi_\e )\, dx  
 \leq   \int_{\R^d} \left|v(t,x)\right|^2  \, d\rho_t(x) \,,
\end{aligned}
\]
where in the last inequality we used Proposition \ref{prop:smoothmeasure}. Since $v(t,\cdot)$ vanishes in $\R^d \smallsetminus \Om$, we conclude the first estimate in  \eqref{eq:est_energy_reg_g}. A similar argument yields the second estimate in  \eqref{eq:est_energy_reg_g}.
\end{proof}

\begin{remark} \label{rem:zero_solutions}
Notice that there exist nontrivial $\rho_t \in \pcurves$ which solve the continuity equation \eqref{cont weak} for $v$ and $g$ that satisfy the bound \eqref{thm:lifting:1}, but such that $\rho_t=0$ on an open interval in $[0,1]$. For example, consider $\Om :=[0,1]^2$, $v(t,x):=(0,0)$, $g(t,x):=-(t-1/2)^{-2} \rchi_{(0,1/2)}(t)$, $\gamma(t):=(1/2,1/2)$, $h(t):=\exp\left( 2- 2(1-2t)^{-1} \right) \rchi_{(0,1/2)} (t)$. It is easy to check that $\rho_t :=h(t)\delta_{\gamma(t)}$ belongs to $\pcurves$, 
 solves \eqref{cont weak} and \eqref{thm:lifting:1} holds.

The above is the reason why we add $\eta_\e$ to the definition of $\rho_t^\e$ in \eqref{eq:regularizations}, since otherwise, we could have $\rho_t \ast \xi=0$ for some $t$, independently on the chosen mollifier. We remark that the addition of $\eta_\e$ is the main difference to the smoothing results \cite[Lemma 8.1.9]{ags} and \cite[Lemma 3.10]{maniglia}, where narrowly continuous measure solutions $\rho_t$ to \eqref{cont weak} are smoothed via $\rho_t \ast \xi$ with $\xi$ being a mollifier.
\end{remark}

\subsection{Proof of the direct implication of Theorem \ref{thm:lifting}} \label{sec:proof_direct} 
We divide the proof of the direct implication of Theorem \ref{thm:lifting} into two steps: First we construct a measure $\sigma \in \M^+(\car_\Om)$ satisfying \eqref{eq:representation}; Then we prove that $\sigma$ is concentrated on $\Ha_\Om^{v,g}$.

\medskip

\emph{Step 1 - Construction of the measure $\sigma$}.

\medskip

Let $V:=B_R(0)$, with $R>0$ as in \eqref{def_of_R}. 
For each $0<\e<1$ define $\rho^\e_t,  v^\e_t, g_t^\e$ according to \eqref{eq:regularizations}. By Lemma \ref{lem:convolution} the triple $(\rho^\e_t \, dx, v^\e_t,g_t^\e)$ solves $\de_t \rho_t^\e \, dx + \div(v_t^\e \rho_t^\e \, dx )= g_t^\e \rho_t^\e \, dx$ in $(0,1) \times \R^d$ and satisfies the bounds \eqref{eq:comp1}, \eqref{eq:comp2}, \eqref{eq:est_energy_reg_g}. As \eqref{thm:lifting:1} holds, we can then apply Proposition~\ref{prop:solutionbypush} and obtain the representation
\begin{equation} \label{eq:repr_reg_bis}
\rho_t^\e \,dx = (X_{(\cdot)}^\e(t))_{\#} [ R_{(\cdot)}^\e(t)  \,dx] \,, \quad R_x^\e (t)=\rho_0^\e(x) \, e^{\int_0^t g^\e(s,X_x^\e(s))\, ds} \,,  
\end{equation}
where $X^\e_x$ and $R^\e_x$ are the unique solutions to the ODEs system
\begin{equation}\label{systemode}
\left\{\begin{array}{l}
\dot{X}_x^\e(t) = v^\e(t,X^\e_x(t))\,,
\medskip\\
X_x^\e(0)= x\,,
\end{array}
\right.
\qquad 
\left\{\begin{array}{l}
\dot{R}_x^\e(t) = g^\e(t,X_x^\e(t)) \,  R^\e_x(t) \,,
\medskip\\
R^\e_x(0) = \rho_0^\e(x)\,,
\end{array}
\right.
\end{equation}
for all $t \in [0,1]$. 
We define $\sigma^\e$ by duality as
\begin{equation}\label{eq:defofmeasure}
\int_{\car_V} \f(\gamma,h) \, d\sigma^\e(\gamma,h) := \int_{V} \f \left(X_x^\e, \frac{R_x^\e}{\int_0^1 R_x^\e(t) \, dt} \right) \, \left( \int_0^1 R_x^\e(t) \, dt \right) \,  dx \,, %
\end{equation}
for all $\f \in C_b(\car_V)$. Here we adopted the notation $\f(\gamma,h)$ to denote that $\f$ is evaluated on the curve $t \mapsto h(t)\delta_{\gamma(t)}$. 
We claim that $\sigma^\e \in \M^+(\car_V)$. First, we show that $\sigma_\e$ is well-defined. Indeed, notice that $\rho_t^\e \geq \e$ in $V$ by construction. Hence by \eqref{eq:repr_reg_bis} and \eqref{eq:comp2} we estimate
\begin{equation} \label{eq:int_h}
\int_0^1 R_x^\e(t) \, dt =  \rho^\e_0(x) \int_0^1 e^{\int_0^t g^\e(s,X_x^\e(s))\, ds}\, dt \geq C(\e) \rho_0^\e(x) \geq  C(\e) \e  > 0\,,
\end{equation}
for all $x \in V$, where $C(\e)>0$ is a constant depending only on $\e$.
Also, by construction, $v^\e(t,x) = 0$ for $x \in \R^d \smallsetminus \overline{\Omega + B_1(0)}$ and $t \in [0,1]$. Therefore from \eqref{systemode} we deduce that $X_x^\e(t) \in V$ for each initial datum $x\in V$ and $0<\e < 1$. Thanks to Lemma \ref{lem:narrow_cone}, we then obtain that the curve $t \mapsto (\int_0^1 R_x^\e(s) \, ds)^{-1} R_x^\e(t) \delta_{X_x^\e(t)}$ belongs to $\car_V$ for all $x \in V$. Moreover the map 
\[
x \mapsto \left( t \mapsto \left(\int_0^1 R_x^\e(s) \, ds \right)^{-1} R_x^\e(t) \delta_{X_x^\e(t)} \right)
\] 
is continuous from $\R^d$ to $(\car_V,\D)$, which is a consequence of the  stability of solutions to \eqref{systemode} with respect to the initial datum $x \in \R^d$, and of the fact that uniform convergence of weights and curves implies $\D$-convergence of measures in $\car_V$, thanks to \eqref{eq:char_flat}. This proves that the definition at \eqref{eq:defofmeasure} is well posed. We now estimate the total variation of $\sigma^\e$. By  \eqref{eq:regularizations} and standard properties of convolutions we have that $\nor{\rho_t^\e \,dx}_{\M(V)} \leq \nor{\rho_t}_{\M(\Om)} + \e |V|$ for all $t \in [0,1]$, $\e \in (0,1)$. 
Hence, by testing $\sigma^\e $ against $\f \equiv 1$ and using \eqref{eq:repr_reg_bis} we infer
\begin{equation} \label{eq:estimate_totvar}
\norm{\sigma^\e}_{\M (\car_V)} \leq \int_{V} \int_0^1 R_x^\e(t) \, dt \,  dx =  \int_0^1  \nor{\rho_t^\e \,dx}_{\M(V)} \, dt  
\leq \nor{\rho}_{\M(X_\Om)} + \e |V|\,.
\end{equation}
Moreover $\sigma^\e \geq 0$ by \eqref{eq:int_h}, showing that $\sigma_\e \in \M^+(\car_V)$. We also remark that $\sigma^\e$ is concentrated on $\Ha_V$, given that the curve $t \mapsto (\int_0^1 R_x^\e(s) \, ds)^{-1} R_x^\e(t) \delta_{X_x^\e(t)}$ belongs to $\Ha_V$ for each $x \in V$, thanks to the regularity of solutions to \eqref{systemode}.

We now show that the family $\sigma^\e$ is tight as $0<\e<1$, by proving that
\begin{equation} \label{eq:bound_F}
\sup_{0<\e<1} \, \int_{\car_V} \mathscr{F}(\gamma,h) \, d \sigma^\e(\gamma,h) <  \infty\,,
\end{equation}
where $\mathscr{F} \colon (\car_V,\D) \to [0,\infty]$ is the functional defined at \eqref{inf_F}: Indeed assume that \eqref{eq:bound_F} holds; by Proposition \ref{lem:compact_sublevels} we know that $\mathscr{F}$ is $\D$-measurable and its sublevels are compact. Moreover $(\car_V,\D)$ is a complete separable metric space (see Proposition \ref{prop:complete}). Thus we can apply Proposition \ref{prop:tight} to conclude tightness for $\sigma^\e$. Let us proceed with the proof of \eqref{eq:bound_F}. First notice that \eqref{eq:defofmeasure} can be tested against $\mathscr{F}$, as $\sigma^\e\geq 0$ and $\mathscr{F}$ is lower semi-continuous with respect to the metric $\D$ (Proposition \ref{lem:compact_sublevels}). Since $\sigma^{\e}$ is concentrated on $\Ha_V$, by formula \eqref{eq:repr_F} and one-homogeneity of $\mathscr{F}$ with respect to $h$ we have
\begin{equation} \label{eq:bound_repr}
\int_{\car_V} \mathscr{F}(\gamma,h) \, d\sigma^{\e}(\gamma,h) = \int_{\car_V} \int_{\{h>0\}} 
\frac{\beta}{2} \, |\dot{\gamma}(t)|^2 h(t) + \frac{\beta\delta^2}{2} \,  \frac{|\dot{h}(t)|^2}{h(t)} + \alpha h(t) \, dt \, d\sigma^{\e}(\gamma,h)\,.
\end{equation}
By \eqref{systemode}, \eqref{eq:repr_reg_bis} and \eqref{eq:est_energy_reg_g}  we estimate
\begin{align*}
\int_{\car_V} \int_0^1 |\dot{\gamma}(t)|^2 h(t)\, dt \,  d\sigma^\e &  = 
\int_{V}\int_0^1  |\dot{X}_x^\e(t)|^2 \, R_x^\e(t) \, dt \,dx 
= \int_0^1  \int_V |v^\e(t,X^\e_x(t))|^2 \, R_x^\e(t) \,dx \, dt \\
& =  \int_0^1 \int_{\R^d} |v^\e(t,x)|^2 \, \rho_t^\e(x) \, dx \, dt 
\leq \int_0^1 \int_{\Om} |v(t,x)|^2 \, d\rho_t(x)\,  dt \,, 
\end{align*}
and, in a similar fashion, 
\begin{align*}
\int_{\car_V} \int_{\{h>0\}} \frac{|\dot{h}(t)|^2}{h(t)}\,dt \, d\sigma^\e  &  = 
\int_{V} \int_{\{h>0\}}  \frac{|\dot{R}_x^\e(t)|^2}{R_x^\e(t)} \,dt \, dx 
=\int_{\{h>0\}} \int_V |g^\e(t,X_x^\e(t))|^2 \, R_x^\e(t) \, dx \,dt \\
&  = \int_{\{h>0\}} \int_{\R^d} |g^\e(t,x)|^2  \,\rho_t^\e(x) \, dx \,dt  
 \leq \int_{0}^1 \int_{\Om} |g(t,x)|^2  \, d\rho_t(x) \, dt  \,.
\end{align*}
Finally, by \eqref{eq:estimate_totvar}, 
\begin{align}\label{eq:boundintotalvariation}
	\int_{\car_V}\int_0^1 h(t) \,dt \, d\sigma^\e = 
\int_{V} \int_0^1 R_x^\e(t) \, dt \, dx \leq \nor{\rho}_{\M(X_\Om)} + \e |V|  \,.
\end{align}
From the above estimates, and \eqref{eq:bound_repr},  \eqref{thm:lifting:1}, we conclude \eqref{eq:bound_F}, proving that $\{\sigma^\e\}_\e$ is tight. 
Since $\{\sigma_\e\}_\e$ is uniformly bounded by  \eqref{eq:estimate_totvar}, we can apply the compactness result \cite[Theorem 8.6.2]{bogachev} to infer the existence of $\sigma \in \M^+(\car_V)$ such that $\sigma_\e \to \sigma$ narrowly as $\e \to 0$. In particular, as $\mathscr{F}$ is \mbox{$\D$-lower semi-continuous}, $\mathscr{F}\geq 0$ and \eqref{eq:bound_F} holds, we can apply \eqref{prop:narrow_unbounded:1} in Proposition \ref{prop:narrow_unbounded} to infer $\int_{\car_V} \mathscr{F}(\gamma,h) \, d\sigma(\gamma,h) < \infty$. From the latter, we see that $\sigma$ is concentrated on the domain of $\mathscr{F}$, that is, on the set $\Ha_V$ (Proposition \ref{lem:compact_sublevels}).

We now prove that $\sigma$ satisfies the representation formula  \eqref{eq:representation}. To this end, let $\f \in C_c(X_V)$ and define the map 
$\Psi(\gamma,h):= \int_0^1 h(t) \f(t,\gamma(t))\, dt$
for $\rho=h\delta_\gamma \in \car_V$. We claim that $\sigma^\e$ according to \eqref{eq:defofmeasure} can be tested against $\Psi$: indeed, first notice that 
$\Psi$ is $\D$-continuous. This is because the map $(\gamma,h) \mapsto h(t) \f(t,\gamma(t))$ is continuous for $t$ fixed, by Lemma~\ref{lem:evaluationsecond}; if $\D(\rho^n, \rho) \rightarrow 0$, then $\nor{h_n -h}_\infty \leq \D(\rho^n, \rho)$ by \eqref{eq:char_flat}, so that $\{h_n\}_n$ is uniformly bounded; thus by dominated convergence we conclude continuity for $\Psi$, since $\f$ is bounded, and since $\f (t,\gamma_n(t)) \to \f(t,\gamma(t))$ when $h(t)>0$. Moreover, thanks to \eqref{eq:boundintotalvariation},  we can estimate 
\begin{equation} \label{review:test}
\begin{aligned}
\int_{V} \left|\Psi\left(X_x^\e, \frac{R_x^\e}{\int_0^1 R_x^\e(t) \, dt} \right) \right| \, \left( \int_0^1 R_x^\e(t) \, dt \right) \,  dx  
& \leq \nor{\f}_\infty (\nor{\rho}_{\M(X_\Om)} + \e |V| ) \,,
\end{aligned}
\end{equation}
showing that the right-hand side of \eqref{eq:defofmeasure} tested against $|\Psi |$ is finite. The fact that $\sigma^\e$ can be tested against $\Psi$ follows immediately. By \eqref{eq:defofmeasure}, the latter yields 
\begin{equation} \label{eq:limit1000}
 \int_{\car_V}\Psi(\gamma,h)    \,  d\sigma^\e(\gamma,h)  = \int_V \int_0^1  \f(t,X_x^\e(t)) R^\e_x(t) \, dt\, dx
= \int_0^1\int_{V}\f(t,x)\, \rho_t^\e(x)\, dx\,dt  \,,
\end{equation}
where in the last equality we used \eqref{eq:repr_reg_bis}. We want to pass to the limit as $\e \to 0$ in \eqref{eq:limit1000}. Notice that the right-hand side passes to the limit since $dt \otimes \rho_t^\e \,dx \weakstar dt \otimes \rho_t$ in $\M(X_V)$: Indeed $\rho_t^\e \, dx \to \rho_t$ narrowly in $\M(V)$ for all $t$ (Lemma~\ref{lem:convolution}) and $\rho_t^\e \, dx$ is uniformly bounded in $\M(V)$, as previously shown. Concerning the left-hand side of \eqref{eq:limit1000}, we first claim that the map $|\Psi|$  is uniformly integrable with respect to $\sigma^\e$ according to definition \eqref{def:UI}. To this end, for $k>0$ define 
$A_k := \{(\gamma,h) \in \car_V : |\Psi(\gamma,h)|    \geq k\}$.
By the definition of $\sigma^\e$ and by  \eqref{eq:boundintotalvariation} we get
\[
\begin{aligned} 
\int_{A_k} \left| \Psi(\gamma,h) \right| \, d\sigma^\e(\gamma,h) & 
 \leq \frac1k \int_{\car_V} |\Psi(\gamma,h)|^2 \, d\sigma^\e(\gamma,h) 
& \leq \frac{\nor{\f}_\infty^2}{k} \int_V \int_0^1 R^\e_x(t) \, dt \, dx \\ 
& \leq \frac{\nor{\f}_\infty^2}{k} ( \norm{\rho}_{\M(X_\Om)} + |V|  ) \,,
\end{aligned}
\]
concluding uniform integrability for $|\Psi|$. 
Therefore we can invoke \eqref{prop:narrow_unbounded:2} and pass to the limit as $\e \rightarrow 0$ in the left-hand side of \eqref{eq:limit1000}. After one application of Fubini's Theorem we obtain
\begin{equation} \label{eq:representation:integral}
\int_0^1  \int_{\car_V}h(t) \f(t,\gamma(t))   \,  d\sigma(\gamma,h) \, dt  =  
\int_0^1\int_{V}\f(t,x)\, d\rho_t (x)\,dt \,, \,\,\, \,\,\, \text{ for all } \, \f \in C_c(X_V) \,.
\end{equation}
We claim that  \eqref{eq:representation} descends from \eqref{eq:representation:integral}. In order to show it, we first derive a pointwise in time version of \eqref{eq:representation:integral}. %
We start by showing that $\Theta(t):=\int_{\car_V}h(t) \f(t,\gamma(t))\,  d\sigma(\gamma,h)$
is continuous for all $\f \in C_c(X_V)$ fixed. Indeed, the map $t \mapsto h(t) \f(t,\gamma(t))$ is continuous for each fixed $(\gamma,h) \in \car_V$, by Lemma~\ref{lem:narrow_cone}.  
Moreover, by recalling that $\sigma^\e$ is concentrated on solutions of \eqref{systemode}, and by arguing as in the proof of \eqref{thm:lifting:3}, we can show that for all $\e$ it holds that
\[
\int_{\car_V} \nor{h}_\infty \, d \sigma^\e (\gamma,h) \leq   
\int_{\car_V} \int_0^1 h(t) (1 + |g^\e (t,\gamma(t))|) \, dt \, d \sigma^\e (\gamma,h) \,.
\]
Therefore, by employing \eqref{eq:boundintotalvariation}, \eqref{eq:defofmeasure}, \eqref{eq:repr_reg_bis}, \eqref{eq:est_energy_reg_g}, and setting $C:=\norm{\rho}_{\M(X_\Om)} + |V|$,  we obtain
\[
\begin{aligned}
\int_{\car_V} \nor{h}_\infty \, d \sigma^\e (\gamma,h) & \leq C +  \int_{V} \int_0^1  R_x^\e(t) |g^\e (t,X_x^\e(t))| \, dt\, dx 
 =  C +  \int_{0}^1 \int_V  |g^\e (t,x)| \,  \rho_t^\e(x)  \, dx \, dt \\
& \leq C + \left( \int_0^1  \norm{\rho_t^\e \, dx}_{\M(V)}  \, dt \right)^{1/2} 
\left(\int_{0}^1 \int_V  |g^\e (t,x)|^2 \,  \rho_t^\e(x) \,dx  \, dt
\right)^{1/2} \\
& \leq C + C^{1/2} \left(\int_{0}^1 \int_V  |g (t,x)|^2 \, d \rho_t(x)  \, dt \right)^{1/2}   \,,
\end{aligned}
\]
and the last term is bounded by assumption \eqref{thm:lifting:1}. 
Finally, the map $(\gamma,h) \in \car_V \mapsto \nor{h}_{\infty}$ is $\D$-continuous and non-negative, therefore by the narrow convergence $\sigma^\e \to \sigma$ and \eqref{prop:narrow_unbounded:1} we infer
$
\int_{\car_V} \nor{h}_\infty \, d \sigma (\gamma,h) < \infty.
$
By dominated convergence we then conclude continuity of $\Theta$. As a byproduct of this argument, we have additionally shown that $\sigma \in \M^+_1(\car_V)$.  Notice that also the map $t \mapsto  \int_V \f(t,x) \, d \rho_t(x)$ is continuous, as a consequence of the narrow continuity of $t \mapsto \rho_t$. Testing \eqref{eq:representation:integral} against $\f(t,x):=a(t)b(x)$ for $a \in C_c(0,1)$, $b \in C(V)$, yields 
\begin{equation} \label{eq:representation:proof}
\int_{\car_V} h(t) b (\gamma(t)) \, d \sigma(\gamma,h) = \int_V b(x) \, d \rho_t(x) \,, \,\,\,\,\,\, \text{ for all } \,\,\, b \in C(V) ,\,\, t \in [0,1]\,.
\end{equation}
Fix $t \in [0,1]$ and $b \in C(V)$ such that $b=0$ in $\Om$ and $b>0$ in $V \smallsetminus \Om$. Recalling that $\rho_t$ is concentrated on $\Om$, from \eqref{eq:representation:proof} we obtain a set $E_t \subset \car_V$ such that $\sigma(\car_V \smallsetminus E_t)=0$ and $h(t)b(\gamma(t))=0$ for all $(\gamma,h) \in E_t$. In particular, by definition of $b$,
\begin{equation} \label{eq:representation:100}
\gamma(t) \in \Om  \,\, \text{ if } \,\, h(t)>0 \,, 
\end{equation}
for all $(\gamma,h) \in E_t$.
Let $Q \subset [0,1]$ be a dense countable subset and define $E:=\cap_{t \in Q} E_t$, so that $\sigma(\car_V \smallsetminus E)=0$ and \eqref{eq:representation:100} holds for all $(\gamma,h) \in E$, $t \in Q$, that is, $\gamma( \{h>0\} \cap Q ) \subset \Om $ for  $\sigma$-a.e.~$(\gamma,h) \in \car_V$.
By density of $Q$ and continuity of $h, \gamma$ we deduce
$\gamma(\{h>0\}) \subset \Om$ for $\sigma$-a.e.~$(\gamma,h) \in \car_{V}$, 
from which we conclude concentration of $\sigma$ on $\car_\Om$. Since we already showed that $\sigma$ is concentrated on $\Ha_V$ we also infer that $\sigma$ is concentrated on $\Ha_\Om$. It is immediate to check that $\car_\Om$ is $\D$-closed in $\car_V$, and hence $\D$-measurable. Therefore we can restrict $\sigma$ to $\car_\Om$ to obtain a measure in $\M^+_1(\car_\Om)$ satisfying \eqref{eq:representation}, as claimed.

\medskip

\emph{Step 2 - $\sigma$ is concentrated on $\Ha_\Om^{v,g}$.}

\medskip

So far we have constructed a measure $\sigma \in \M^+_1(\car_\Om)$ concentrated on $\Ha_\Om$ and satisfying \eqref{eq:representation}. We now prove that $\sigma$ is concentrated on $\Ha_\Om^{v,g}$, i.e., that the ODEs \eqref{ODE1}-\eqref{ODE2} hold for $\sigma$-a.e.~$(\gamma,h)$ in $\Ha_\Om$. 
This claim follows from two preliminary estimates, whose proof we postpone for a moment: For $\bar v \in C_c(X_\Om;\R^d), \bar g \in C_c(X_\Om)$ and any $\f \in C_c^1(0,1)$, there exists a constant $C>0$ depending only on $\f$ and on the radius of $V$, such that
\begin{gather} \label{eq:concentrationh}
\int_{\car_\Om} \left| \int_0^1 h(t) \f'(t) + h(t)\,  \bar g(t,\gamma(t)) \f(t)\, dt \right| \, d\sigma \leq C \int_{X_\Om} |\bar g -  g| \, d\rho_t \, dt\,, \\
\int_{\car_\Om}\left|\int_{0}^1 h(t)  \gamma(t) \cdot  \f'(t) + \Psi_{\bar v,\bar g}(t,\gamma,h) \cdot \f(t) \, dt\, \right| d\sigma \leq C  \int_{X_\Om} |\bar v - v|  + |\bar g - g| \, d\rho_t \, dt \label{eq:concentrationv}\,,
\end{gather}
where $\Psi_{\bar v,\bar g}(t,\gamma,h) := h(t) \gamma(t)\, \bar g(t,\gamma(t))  + h(t) \, \bar v(t,\gamma(t))$. We start by showing \eqref{ODE2}. 
By the energy bound \eqref{thm:lifting:1} and H\"older's inequality, we can find two sequences  
$\{v_n\}_n$ in $C_c(X_\Om;\R^d)$ and $\{g_n\}_n$ in $C_c(X_\Om)$ converging to $v$ and $g$ in $L^1_{\rho}(X_\Om)$, respectively. By \eqref{eq:representation} and Remark~\ref{rem:representationL1} we get
\begin{equation} \label{eq:concentrationh400}
\int_{\car_\Om}   \left| \int_0^1 h(t) ( g(t,\gamma(t)) - g_n(t,\gamma(t))) \f(t)\, dt \right| \, d\sigma
\leq \nor{\f}_\infty \int_{X_\Om} |g_n -  g| \, d\rho_t \, dt\,.
\end{equation}
Hence by \eqref{eq:concentrationh} with $\bar g :=g_n$, \eqref{eq:concentrationh400}, and triangle inequality we get
\[
\int_{\car_\Om}   \left| \int_0^1 h(t) \f'(t) + h(t) g(t,\gamma(t)) \f(t)\, dt \right| \, d\sigma
\leq C  \int_{X_\Om} |g_n -  g| \, d\rho_t \, dt \to 0\,,
\]
as $n \to \infty$, for every test function $\f \in C_c^1(X_\Om)$. Therefore
\begin{equation} \label{eq:density}
\int_0^1 h(t) \f'(t) +h(t)  g(t,\gamma(t)) \f(t)\, dt = 0 \quad \text{ for all } \,\,\, (\gamma,h) \in E_\f\,, 
\end{equation}
where $\sigma(\car_\Om \smallsetminus E_\f)=0$ and $E_\f$ depends on $\f$. Let $D \subset C_c^1 (0,1)$ be a dense countable set and $E:=\cap_{\f \in D} E_\f$, so that $\sigma(\car_\Om \smallsetminus E)=0$ and \eqref{eq:density} holds for all $\f \in D$ and $(\gamma,h) \in E$. 
Consider $\f(t,x):=1+|v(t,x)|+|g(t,x)|$ and notice that $\f(t,\cdot) \in L^1_{\rho_t}(\Om)$ for a.e.~$t \in (0,1)$, thanks to \eqref{thm:lifting:1} and narrow continuity of $\rho_t$. Hence, by Remark~\ref{rem:representationL1} applied to $\f$, we conclude that $\sigma$ satisfies \eqref{eq:conv_bound}. Therefore  %
there exists a set $F$ with $\sigma(\car_\Om \smallsetminus F)=0$ and such that 
\begin{equation} \label{eq:en_bound:1}
\int_0^1 h(t) (1 + |v(t,\gamma(t))| +|g(t,\gamma(t))|  ) \, dt < \infty
\end{equation}
for all $(\gamma,h) \in F$. Consider now $\f \in C_c^1(0,1)$ and $\f_n \in D$ such that $\f_n \to \f$ in $C_c^1(0,1)$. As a consequence of \eqref{eq:density}, for any $(\gamma,h)$ in $E \cap F$ we have
\[
\left| \int_0^1 h(t) \f'(t) +h(t)  g(t,\gamma(t)) \f(t)\, dt \right| \leq \nor{\f_n - \f}_{C^1} \, \int_0^1 h(t) (1 + |g(t,\gamma(t))| ) \, dt \to 0\,,
\]
as $n \to \infty$, so that \eqref{eq:density} holds for all $\f \in C^1_c(0,1)$ and $(\gamma,h) \in E \cap F$. Therefore
\begin{equation} \label{eq:density:2}
\dot h = g(t,\gamma(t)) h(t)
\end{equation}
in the sense of distributions for $\sigma$-a.e. $(\gamma,h) \in \car_\Om$. Since $\sigma$ is concentrated on $\Ha_\Om$, the distributional formulation of \eqref{eq:density:2} coincides with the a.e.~one, so that \eqref{ODE2} holds. We now prove \eqref{ODE1}, which follows by similar arguments. First, by \eqref{eq:representation} and Remark~\ref{rem:representationL1}, we estimate  
\[
\begin{aligned}
\int_{\car_\Om}  \bigg| \int_0^1    (\Psi_{v,g} & (t,\gamma,h)  - \Psi_{v_n,g_n} (t,\gamma,h)) \cdot \f(t)  \, dt \bigg| \, d \sigma \\
& \leq \nor{\f}_\infty  \max\{1,R\}  \int_{X_\Om}   
 |v_n - v| + |g_n - g| \, d\rho_t \, dt \,,
\end{aligned}
\]
where $R$ is as in \eqref{def_of_R}. 
By applying \eqref{eq:concentrationv} to $\bar v :=v_n$, $\bar g :=g_n$ and by triangle inequality we infer
\[
\int_{\car_\Om}\left|\int_{0}^1 h(t)  \gamma(t) \cdot  \f'(t) + \Psi_{v,g}(t,\gamma,h) \cdot \f(t) \, dt\, \right| d\sigma \leq C \int_{X_\Om} |v_n - v|  + |g_n - g| \, d\rho_t \, dt \to 0\,, 
\]
as $n \to \infty$, for all $\f \in C^1_c((0,1);\R^d)$.  
By reasoning as above, we can find a countable dense subset $\tilde{D}$ of $C^1_c((0,1);\R^d)$ and a set $\tilde{E}$ with $\sigma(\car_\Om \smallsetminus \tilde{E} )=0$ such that
\begin{equation} \label{eq:density:100}
\int_{0}^1 h(t)  \gamma(t) \cdot  \f'(t) + \Psi_{v,g}(t,\gamma,h) \cdot \f(t) \, dt = 0
\end{equation}
for all $\f \in \tilde{D}$, $(\gamma,h) \in \tilde{E}$. By \eqref{eq:density:100} and \eqref{eq:en_bound:1}, we conclude that 
\[
(h(t)\gamma(t))' =  h(t) \gamma(t) g(t,\gamma(t)) +
h(t) v(t,\gamma(t)) 
\]
in the sense of distributions for all $(\gamma,h) \in \tilde{E}\cap F$. Recall that \eqref{eq:density:2} holds in the sense of distributions in $E \cap F$. Moreover $\sigma$ is concentrated on $\Ha_\Om$, whose elements satisfy $h\gamma \in \AC^2([0,1];\R^d)$ and the product rule holds (see Lemma \ref{lem:hgammasolutioncontinuity}). Hence we can find a set $\tilde{F}$ such that $\sigma(\car_\Om \smallsetminus \tilde{F})=0 $, and that \eqref{ODE2} and
\[
h(t) \dot{\gamma}(t) = h(t) v(t,\gamma(t)) \quad \text{ for a.e. } \,\,\, t \in (0,1)\,,
\]  
hold for $\sigma$-a.e.~$(\gamma,h) \in \Ha_\Om$. This establishes \eqref{ODE1}.  

We are left to show \eqref{eq:concentrationh}-\eqref{eq:concentrationv}. 
We start by proving \eqref{eq:concentrationh}. 
First notice that the map
\begin{equation}\label{eq:testfunctionconcentration}
\phi(\gamma,h) := \left|  \int_0^1 h(t) \f'(t)+ h(t)  \bar g(t,\gamma(t))  \f(t)\, dt \right|
\end{equation}
in the left-hand side of \eqref{eq:concentrationh} is $\D$-continuous: indeed, if $\D(\rho^n, \rho) \to 0$, from \eqref{eq:char_flat} we have $\nor{h_n -h}_\infty \leq \D(\rho^n, \rho)$, so that $h_n$ is uniformly bounded; by Lemma~\ref{lem:evaluationsecond} it follows that  $(\gamma,h) \mapsto h(t) \f'(t)+ h(t)  \bar g(t,\gamma(t))  \f(t)$ is $\D$-continuous for every $t$; thus continuity of $\phi$ follows by dominated convergence. Extend $\bar{g}$ to zero outside of $\Om$ and set $\bar g^\e := [(\bar g\rho_t) \ast \xi_\e]/\rho_t^\e$, 
with $\xi_\e$ as in \eqref{eq:regularizations}. 
As $\phi$ is continuous, we can test \eqref{eq:defofmeasure} against $\phi$, integrate by parts, and use \eqref{eq:repr_reg_bis}-\eqref{systemode} to get
\begin{equation} \label{eq:concentrationh300}
\begin{aligned}
 \int_{\car_V}   \bigg| \int_0^1 & h(t) \f'(t)+    h(t) \bar g(t,\gamma(t)) \f(t)\, dt \bigg| \, d\sigma^\e(\gamma,h) \\
& \leq \nor{\f}_\infty \int_{X_V}   \left|\bar g - g^\e\right|\, \rho_t^\e \, dx\, dt  
 \leq \nor{\f}_\infty  \int_{X_V}   \left|\bar g^\e - g^\e\right| +  \left|\bar g^\e - \bar g\right|\, \rho_t^\e\, dx \, dt \,.
\end{aligned}
\end{equation}
We then estimate each term separately. First, recalling \eqref{eq:regularizations}, 
\begin{equation} \label{eq:concentrationh200}
\int_{X_V}   \left|\bar g^\e - g^\e\right|\, \rho_t^\e\,dx  \, dt  
 = \int_0^1  \int_{\Omega+B_\e(0)} \left|((g - \bar g) \rho_t ) \ast \xi_\e\right|\,  dx \,dt
 \leq  \int_{X_\Om} \left|g - \bar g\right|\, d\rho_t \,dt\,,
\end{equation}
by standard estimates on convolutions of measures. Second, 
\begin{align} 
\int_{X_V} |\bar  g^\e & -  \bar g|\, \rho_t^\e \,dx\,dt 
 = \int_{X_V} |(\bar g \rho_t)\ast \xi_\e -  \bar g (\rho_t \ast \xi_\e + \e)| \, dx \,dt \nonumber \\
 & \leq      \int_{X_V} |(\bar g \rho_t)\ast \xi_\e -  \bar g(\rho_t \ast \xi_\e)| \, dx \,dt + 
 \e  \int_{X_V} |\bar g| \,dx \,dt  \,. \label{q:concentrationh20000}
\end{align} 
The second term in \eqref{q:concentrationh20000} converges to zero as $\e\rightarrow 0$. Moreover, by the uniform continuity of $\bar g$, for each $\zeta>0$ there exists $\tilde{\zeta}>0$ such that $ |\bar g(t,y)- \bar g(t,x)| < \zeta$ whenever $|x-y|\leq \tilde{\zeta}$, $t \in [0,1]$. Therefore, for all $\e< \tilde{\zeta}$ and $x \in V$ we have
\begin{align*}
|(\bar g \rho_t)\ast \xi_\e -  \bar g (\rho_t \ast \xi_\e)|(x) %
 \leq   \int_{B_\e(x)} |\bar g(t,y)- \bar g(t,x)| \xi_\e(x-y)\, d\rho_t(y) \leq \zeta \, (\rho_t \ast \xi_\e)(x) \,,
\end{align*}
from which we infer
\begin{align*}
\int_{X_V} |(\bar g \rho_t)\ast \xi_\e -  \bar g(\rho_t \ast \xi_\e)|\, dx \,dt %
\leq \zeta  \int_{X_V} \rho_t \ast \xi_\e \, dx \,dt \leq \zeta \nor{\rho}_{\M(X_\Om)} \,.
\end{align*}
As $\zeta$ is arbitrary, by \eqref{q:concentrationh20000} we conclude that
$\int_{X_V} |\bar  g^\e  -  \bar g|\, \rho_t^\e \, dx \,dt  \to 0$. Thus, from \eqref{eq:concentrationh300}-\eqref{eq:concentrationh200}
\[
\limsup_{\e\rightarrow 0} \int_{\car_V} \left|  \int_0^1 h(t) \f'(t) +   h(t) \bar g(t,\gamma(t))  \f(t)\, dt \right| \, d\sigma^\e(\gamma,h) \leq  \nor{\f}_\infty \int_{X_\Om} |g - \bar g| \, d\rho_t \, dt \,.
\]
As $\sigma^\e \rightarrow \sigma$ narrowly,  $\sigma$ is concentrated on $\car_\Om$, $\phi$ at \eqref{eq:testfunctionconcentration} is continuous and non-negative, by \eqref{prop:narrow_unbounded:1} we conclude \eqref{eq:concentrationh}. 
We now show \eqref{eq:concentrationv}. 
First notice that the function in the left integral of \eqref{eq:concentrationv} is $\D$-continuous, a fact that can be proven similarly to \eqref{eq:testfunctionconcentration}.  Set $C := (R + 1)\nor{\f}_\infty$, with $R$ as in \eqref{def_of_R}.  
Similarly to the above proof of \eqref{eq:concentrationh}, we can integrate by parts and make use of  \eqref{eq:repr_reg_bis}-\eqref{eq:defofmeasure}, and estimate
\begin{equation} \label{eq:concentrationh5000}
\begin{aligned}
 \int_{\car_V} & \left|\int_{0}^1  h(t) \gamma(t) \cdot \f'(t) + \Psi_{\bar v, \bar g}(t,\gamma,h) \cdot \f(t) \, dt\, \right| d\sigma^\e(\gamma,h) \\
& \leq C \int_{X_V} R_x^\e(t) \Big[ \left| g^\e(t,X_x^\e(t))-\bar g(t,X_x^\e(t)) \right| + \left| v^\e(t,X_x^\e(t) ) -   \bar v(t,X_x^\e(t))\right|\Big] \, dt\,  dx\\
& = C\int_{X_V}\left| g^\e-\bar g \right|    +   \left| v^\e-   \bar v\right| \, \rho_t^\e\, dx  \, dt 
 \leq  C \int_{X_V}\left| g-\bar g \right|    +   \left| v-   \bar v\right| \, \rho_t^\e \, dx \,  dt + o(1)\,,
\end{aligned}
\end{equation}
where in the last inequality we employed \eqref{eq:concentrationh200} and the convergence $\int_{X_V} |\bar  g^\e  -  \bar g|\, \rho_t^\e \, dx\,dt  \to 0$ to estimate the first term, and similar estimates involving $v$ for the second, and where $o(1) \to 0$ as $\e \to 0$. By passing to the limes superior in \eqref{eq:concentrationh5000} and by recalling that $\sigma^\e \to \sigma$ narrowly, we can invoke \eqref{prop:narrow_unbounded:1} and obtain \eqref{eq:concentrationv}.

\medskip

 \emph{Step 3 - $\sigma$ is concentrated on $\Ha_\Om^1$.} 

\medskip

We are left to prove that $\sigma$ is concentrated on 
$\mathcal{H}^1_\Om = \{(h,\gamma)\in \car_\Om : \|h\|_{1} = 1\}$.
By definition, the measure $\sigma^\e \in \mathcal{M}^+(\car_V)$ introduced at \eqref{eq:defofmeasure} is concentrated on $\mathcal{H}^1_V =  \{(h,\gamma)\in \car_V : \|h\|_{1} = 1\}$, where $V$ is as in \eqref{def_of_R}. Also recall that we have proven $\sigma^\e \rightarrow \sigma$ narrowly. Moreover note that the set $\mathcal{H}^1_V$ is closed in $\car_V$, as an immediate consequence of \eqref{eq:char_flat}. 
Hence, by \eqref{prop:narrow_unbounded:1}, we get 
$\sigma(\car_V \smallsetminus \mathcal{H}^1_V) \leq \liminf_{\e \rightarrow 0} \sigma^\e(\car_V \smallsetminus \mathcal{H}^1_V) = 0$, 
showing that $\sigma$ is concentrated on $\mathcal{H}^1_V$. 
As $\sigma$ is concentrated on $\car_\Om$, we conclude. This ends the proof of Theorem \ref{thm:lifting}.

\begin{remark}\label{rem:proofRd}
	As mentioned in the introduction, it would be interesting to extend Theorem~\ref{thm:lifting} to the case of $\Om=\R^d$. Notice however that our construction of the measure $\sigma$ is heavily reliant on the boundedness of $\Om$: first, such assumption is needed in proving compactness of the sublevels of the functional $\mathscr{F}$ (see \eqref{eq:infinity_bound} and estimates after), which in turn allows to show tightness for the family $\sigma^\e$ (see \eqref{eq:bound_F} and argument immediately after); second, boundedness of $\Om$ is employed to provide the uniform bound \eqref{eq:estimate_totvar} on the norm of $\sigma^\e$. These arguments are crucial to obtain compactness for $\sigma^\e$ and, consequently, the representing measure $\sigma$ as their limit.    
\end{remark}

\section{Uniqueness of characteristics and uniqueness for the PDE} \label{sec:uniqueness}

The aim of this section is to apply Theorem \ref{thm:lifting} to relate uniqueness of the characteristics with uniqueness of solutions for the continuity equation with given initial data and minimal total variation.
Throughout the section $\Om \subset \R^d$ with $d \geq 1$ is the closure of a bounded 
domain. We denote $X_\Om:=(0,1)\times \Om$. Moreover $\car_\Om$ denotes the set defined at \eqref{eq:narr}, equipped with the distance $\D$ at \eqref{distance_sup}. We remind the reader that $(\car_V,\D)$ is a complete metric space (Proposition \ref{prop:complete}).   Let $v \colon X_\Om \rightarrow \R^d$ and $g \colon X_\Om \rightarrow \R$ be  measurable maps and recall the definition of $\Ha_\Om^{v,g}$ at \eqref{def:havg}, i.e., the set of regular characteristics of the ODEs system \eqref{ODE1}-\eqref{ODE2}. Also recall the definition of $\Ha_\Om^1$ at \eqref{eq:defh1} Finally, we define  the following set
\begin{equation*}
\begin{aligned}
\mathcal{D}_{v,g} := \left\{(t \mapsto \rho_t )  \in \pcurves \, \colon \, (\rho_t,v_t,g_t)  \text{ satisfy }  \partial_t \rho_t + \div (v\rho_t) = g\rho_t \text{ and }  \eqref{thm:lifting:1}\right\}\,.
\end{aligned}
\end{equation*}

We will prove the following result:

\begin{theorem}\label{thm:uniqueness}
Let $A\subset \Om$ be a measurable set. Suppose that: 

\begin{itemize}
\item[\textsc{(Hyp)}] For each $x\in A$ the solution $(\gamma,h) \in \Ha^{v,g}_\Om$ to \eqref{ODE1}-\eqref{ODE2} with initial value $(x,1)$ is unique in $[0,\tau)$ for every $\tau \in (0,1)$ such that $[0,\tau) \subset \{h>0\}$, i.e., if $(\gamma_1,h_1),(\gamma_2,h_2) \in \Ha^{v,g}_\Om$ solve \eqref{ODE1}-\eqref{ODE2} in $[0,\tau)$ with initial data $(x,1)$  and $h_1 >0$, $h_2>0$ in $[0,\tau)$, then $h_1=h_2$ and $\gamma_1=\gamma_2$ in $[0,\tau)$.
\end{itemize}
Then, for any initial data $\rho_0 \in \mathcal{M}^+(\Om)$ concentrated on $A$, the continuity equation $\partial_t \rho_t + \div (v\rho_t ) = g\rho_t $ admits at most one solution $\rho \in \mathcal{D}_{v,g}$ with initial data $\rho_0$ and such that
\begin{equation}\label{eq:minimalityofJ}
\| \rho\|_{\mathcal{M}(X_\Om)} \leq  \|\tilde \rho\|_{\mathcal{M}(X_\Om)} \quad \text{ for all }  \, \tilde \rho \in  \mathcal{D}_{v,g} \, \text{ such that } \, \tilde \rho_0 = \rho_
0\,.
\end{equation} 
\end{theorem}

In the next section we provide several auxiliary lemmas and definitions, which will be instrumental in proving Theorem \ref{thm:uniqueness}. The proof of Theorem \ref{thm:uniqueness} will be carried out in Section \ref{sec:proof_uniqueness}.

\subsection{Auxiliary results}

Define the following subset of $\car_\Om$:
\begin{equation}\label{eq:essestar}
\car_\Om^* := \{(\gamma,h) \in \car_\Om : \{h>0\} = [0,1] \cap (-\infty,\tau) \text{ for some } \tau \in \R\}\,.
\end{equation}

The first step is to prove that condition \eqref{eq:minimalityofJ} implies that the measure $\sigma$ obtained by  Theorem~\ref{thm:lifting} is concentrated on $\car_\Om^*$.
To this aim, we define a cut-off operator on the space $\car_\Om$.

\begin{definition}[Cut-off operator]\label{def:cutoffoperator}
Define the \textit{vanishing time} map $\tau : \car_\Om \rightarrow [0,\infty]$ as
\begin{equation*}
\tau(\gamma,h) := 
\begin{cases}
\argmin \, \{t\in [0,1] \, \colon \, h(t) = 0\} & \text{ if } \, \{t\in [0,1] \, \colon \, h(t) = 0\} \neq \emptyset\,,\\
\infty & \text{ otherwise},
\end{cases}
\end{equation*}
and the cut-off operator $G:\car_\Om \rightarrow \car_\Om$ as 
$G(\gamma,h) := (\gamma, h \rchi_{[0,\tau(h,\gamma))})$.
\end{definition}

\begin{lemma}\label{lem:cutoffoperator}
Let $\tau$ and $G$ be as in in Definition \ref{def:cutoffoperator}. Then $\tau$ is lower semi-continuous and $G$ is measurable. Moreover,  $G(\car_\Om) \subset \car_\Om^*$, $G(\Ha_\Om^{v,g}) \subset \Ha_\Om^{v,g}$, and the set $\car_\Om^*$ is measurable.
\end{lemma}

\begin{proof}
We start by proving that $\tau$ is lower semi-continuous. 
Assume that $\D(\rho^n,\rho) \to 0$ as $n \to \infty$.
Set $\tau_n:=\tau(\rho^n)$, $\tau:=\tau(\rho)$. Without loss of generality we can suppose that $
\tau^* := \lim_{n\rightarrow \infty}\tau_n= \liminf_{n\rightarrow \infty} \tau_n< \infty$.
Thus, by definition, we have that $h_n(\tau_n) = 0$ for $n$ sufficiently large. Therefore
\begin{equation}\label{eq:limitis2}
h(\tau^*) \leq |h(\tau^*) - h(\tau_n)| + |h(\tau_n) - h_n(\tau_n)|\,.
\end{equation}
Notice that $\|h_n - h\|_\infty \rightarrow 0$ by \eqref{eq:char_flat} and $\D(\rho^n,\rho) \to 0$. Hence the second term in \eqref{eq:limitis2} converges to zero as $n \to \infty$.   
Thanks to the continuity of $h$ and the convergence $\tau_n \to \tau^*$ also the first term in \eqref{eq:limitis2} is infinitesimal, concluding that $h(\tau^*) = 0$. Thus $\tau \leq \tau^*$ by minimality, from which lower semi-continuity follows.

We now show that $G$ is measurable by constructing measurable maps $G_n \colon \car_\Om \to \car_\Om$ such that $G_n(\rho) \to G(\rho)$ for all $\rho \in \car_{\Om}$. Indeed this immediately implies measurability of $G$ (see, e.g., \cite[Corollary 6.2.6]{bogachev}). 
To this end, define the continuous maps $\f_n \colon \R \to [0,1]$ by setting
\[
\f_n(t):= \rchi_{(-\infty,-1/n]}(t) - nt \, \rchi_{(-1/n,0)} (t) \,.
\]
Introduce $T_n \colon [0,\infty] \to C([0,1];[0,1])$ by $T_n(s)(t):=\f_n(t-s)$ for all $s \in [0,\infty], t \in [0,1]$. It is straightforward to check that $T_n$ is continuous. Thus the map
\begin{equation} \label{lem:cutoffoperator:3}
\rho  \mapsto ( (T_n \circ \tau)(\rho),\rho)
\end{equation}
from $\car_\Om$ into $C([0,1];[0,1]) \times \car_\Om$ is measurable, given that $\tau$ lower semi-continuous, and hence measurable.  Moreover 
\begin{equation} \label{lem:cutoffoperator:4}
(\f,\rho) \mapsto \f\rho
\end{equation}
from $C([0,1];[0,1]) \times \car_\Om$ into $\car_\Om$ is continuous, since by triangle inequality and \eqref{eq:char_flat} we can readily check that for all $\rho^i=h_i\delta_{\gamma_i} \in \car_\Om$, $\f_i \in C([0,1];[0,1])$, $i=1,2$, it holds
\[
\D(\f_1 \rho^1, \f_2 \rho^2) 
\leq  \D(\f_1 \rho^1, \f_2 \rho^1) + \D(\f_2 \rho^1, \f_2 \rho^2) = \nor{\f_1 - \f_2}_{\infty} \nor{h_1}_{\infty} + \nor{\f_2}_\infty \D(\rho^1,\rho^2) \,.
\]
We now define $G_n \colon \car_\Om \to \car_{\Om}$ by composing the maps at \eqref{lem:cutoffoperator:3}-\eqref{lem:cutoffoperator:4}, that is, 
\[
G_n(\rho)= (T_n \circ \tau)(\rho) \, \rho = (\gamma,h\, T_n(\tau(\rho))) \,, \quad \text{ for } \, \rho \in \car_\Om\,.
\]
In view of the above, $G_n$ is measurable for all $n \in \N$. We now claim that $G_n \to G$ pointwise in $\car_\Om$. Indeed, by \eqref{eq:char_flat}, we see that
\begin{equation} \label{lem:cutoffoperator:1}
\D(G_n(\rho),G(\rho))=\sup_{t \in [0,1]} h(t) \left| \rchi_{[0,\tau(\rho))}(t) - \f_n (t-\tau(\rho))\right|  \,, \,\,\, \text{ for all } \,\, \rho = h \delta_\gamma \in \car_\Om \,.
\end{equation}
Fix $\rho = h \delta_\gamma \in \car_\Om$.  If $\tau(\rho)=\infty$ it is immediate to check that $T_n (\tau(\rho)) \equiv 1$ in $[0,1]$, so that $\D(G_n (\rho), G(\rho))=0$ for all $n \in \N$ by \eqref{lem:cutoffoperator:1}. Similarly, if $\tau(\rho)=0$, then $T_n (\tau(\rho)) \equiv 0$ in $[0,1]$ and again $\D(G_n (\rho), G(\rho))=0$ for all $n \in \N$. Finally, assume that $0<\tau(\rho)<\infty$ and fix $\e>0$. By continuity of $h$ there exists $n_0 \in \N$ such that $\tau(\rho)-1/n_0 >0$ and
\begin{equation} \label{lem:cutoffoperator:2}
h(t)<\e  \,\, \text{ for all } \,\, t \in [\tau(\rho)-1/n_0,\tau(\rho) ] \,.
\end{equation}
For all $n \geq n_0$ we can compute
\[
\begin{gathered}
h(t) \left| \rchi_{[0,\tau(\rho))}(t) - \f_n (t-\tau(\rho))\right| = 0\,, \,\,\,\, \text{ if } \,\,\,\, t \in  [0,\tau(\rho)-1/n] \cup [\tau(\rho),1]\,, \\
h(t) \left| \rchi_{[0,\tau(\rho))}(t) - \f_n (t-\tau(\rho))\right| = h(t)|1+n(t-\tau(\rho))| \leq h(t) \,,  \,\,\,\, \text{ if } \,\,\,\, t \in [\tau(\rho)-1/n,\tau(\rho)) \,.
\end{gathered}
\]
Recalling \eqref{lem:cutoffoperator:1}-\eqref{lem:cutoffoperator:2} we then obtain $\D(G_n(\rho),G(\rho))<\e$ for $n \geq n_0$, and the proof of the measurability of $G$ is concluded.

The inclusion $G(\car_\Om) \subset \car^*_\Om$ is immediate from the definition of $\car^*_\Om$. For the inclusion $G(\Ha_{\Om}^{v,g}) \subset \Ha_{\Om}^{v,g}$, consider $(\gamma,h) \in \Ha_\Om^{v,g}$ and notice that if $\tau(\gamma,h) = \infty$, then $G(\gamma,h) = (\gamma,h)$ and the thesis is immediate. On the other hand, if $\tau(\gamma,h)  \in [0,1]$, then  for every $\f \in C_c(0,1)$ there holds
\begin{align*}
\int_0^1 h(t) & \rchi_{[0,\tau(\gamma,h))}(t) 
\dot{\varphi}(t)\, dt  = \int_0^{\tau(\gamma,h)} h(t)\dot{\varphi}(t)\, dt = -\int_0^{\tau(\gamma,h)} \dot{h}(t)\varphi(t)\, dt + h(\tau(\gamma,h)) \f(\tau(\gamma,h))\\
& = -\int_0^{\tau(\gamma,h)} h(t) g(\gamma(t),t)\varphi(t)\, dt = -\int_0^{1} h(t)  \rchi_{[0,\tau(\gamma,h))}(t) g(\gamma(t),t)\varphi(t)\, dt \,,
\end{align*}
implying that $h\rchi_{[0,\tau(\gamma,h))} \in \AC^2[0,1]$ and
$(h\rchi_{[0,\tau(\gamma,h))})'(t) = \rchi_{[0,\tau(\gamma,h))} (t) g(\gamma(t),t)h(t)$ for  a.e.  $t\in (0,1)$.
Noticing that $\sqrt{h  \rchi_{[0,\tau(\gamma,h))}} = \sqrt{h}  \rchi_{[0,\tau(\gamma,h))}$, by a similar argument we obtain that $\sqrt{h  \rchi_{[0,\tau(\gamma,h))}} \in \AC^2[0,1]$ and $\sqrt{h  \rchi_{[0,\tau(\gamma,h))}} \gamma  \in \AC^2[0,1]$. This shows that $(\gamma,  h  \rchi_{[0,\tau(\gamma,h))}) \in \Ha_\Om^{v,g}$, concluding the claimed inclusion.
We finally show that $\car_\Om^*$ is measurable. Consider the map $\mathcal{G} : \car_\Om  \rightarrow \car_\Om \times \car_\Om$ defined as 
$\mathcal{G}(\gamma,h) := ((\gamma,h), G(\gamma,h))$ , where $\car_\Om \times \car_\Om$ is equipped with the Borel $\sigma$-algebra of the product space. As $\car_\Om$ is a separable metric space (Proposition \ref{prop:complete}) and 
$G$ is measurable, we deduce that also $\mathcal{G}$ is measurable. Define the diagonal 
$D := \{((\gamma,h), (\gamma,h)) : (\gamma,h) \in \car_\Om\}$, which
 is clearly closed.
As the set  $\text{Fix}(G) := \{(\gamma,h) \in \car_\Om : G(\gamma,h) = (\gamma,h)\}$ coincides with $\car^*_\Om$, we obtain that $\car_\Om^* = \mathcal{G}^{-1}(D)$, implying that $\car_\Om^*$ is measurable.
\end{proof}

\begin{lemma}\label{lem:minimalityimpliesconcentration}
Let $(t \mapsto \rho_t) \in \pcurves$ be in $\mathcal{D}_{v,g}$ and $\sigma \in \mathcal{M}_1^+(\car_\Om)$ be concentrated on $\mathcal{H}^{v,g}_\Om$. 
Suppose that \eqref{eq:representation} and \eqref{eq:minimalityofJ} hold. 
Then $\sigma$ is concentrated on $\car_\Om^*$. 
\end{lemma}

\begin{proof}
Suppose by contradiction that $\sigma(\car_\Om \smallsetminus \car_\Om^*)>0$. Let $\tau$ and $G$ be as in Definition \ref{def:cutoffoperator}. As $G$ is measurable (Lemma \ref{lem:cutoffoperator}), we can consider the measure $\hat \sigma := G_\# \sigma \in \mathcal{M}^+(\car_\Om)$. 
By the inclusion $G(\car_{\Om}) \subset \car_\Om^*$ in Lemma \ref{lem:cutoffoperator}, we have that $\hat \sigma$ is concentrated on $\car^*_\Om$. Using that $\sigma$ is concentrated on $\Ha_\Om^{v,g}$ and the inclusion $G(\Ha_\Om^{v,g}) \subset \Ha_\Om^{v,g}$ (Lemma \ref{lem:cutoffoperator}), we also deduce that $\hat \sigma$ is concentrated on $\Ha_\Om^{v,g}$.
By Remark~\ref{rem:conv_bound} and by definition of $\hat\sigma$, we  get that $\hat\sigma$ satisfies \eqref{eq:conv_bound} with respect to $v$ and $g$. 
Therefore we can apply Theorem \ref{thm:lifting} to $\hat \sigma$ and obtain a curve $t \mapsto \hat \rho_t$ in $\pcurves$ such that \eqref{eq:representation} hold and 
$\partial_t \hat\rho_t + \div(v\hat\rho_t) = g\hat\rho_t$ in $X_\Om$. Additionally, using Remark~\ref{rem:representationL1}, \eqref{eq:representation}, and the definition of $\hat \sigma$, we obtain that
\begin{align*}
\int_0^1  \int_{\Om} |v(t,x)|^2 + |g(t,x)|^2\, d\hat \rho_t(x)\, dt  \leq  \int_0^1 \int_{\Om} |v(t,x)|^2 + |g(t,x)|^2\, d\rho_t(x)\, dt<\infty\,,
\end{align*}
implying that $(t \mapsto \hat \rho_t) \in \mathcal{D}_{v,g}$. Moreover $\hat\rho_0=\rho_0$, by \eqref{eq:representation} at time $t=0$ and definition of $\hat\sigma$.
Finally, using again \eqref{eq:representation}, we estimate
\begin{equation} \label{eq:lastmini}
\begin{aligned}
\|\hat \rho\|_{\mathcal{M}(X_\Om)} & = \int_0^1 \int_{\car_\Om} h(t)\, d\hat \sigma(\gamma,h)\, dt 
 =  \int_{\car_\Om} \int_0^{\tau(h,\gamma)} h(t)\, dt\,  d\sigma(\gamma,h) \\
& \leq  \|\rho\|_{\mathcal{M}(X_\Om)} - \int_{\car_\Om\smallsetminus \car^*_\Om} \int_{\tau(h,\gamma)}^1 h(t)\, dt\,  d\sigma(\gamma,h)\,.
\end{aligned}
\end{equation}
Thanks to the continuity of $h$ for every $(\gamma,h) \in \car_\Om$ and the definition of $\car^*_\Om$, we know that 
$\int_{\tau(h,\gamma)}^1 h(t)\, dt > 0$ for all $(\gamma,h) \in \car_\Om\smallsetminus \car^*_\Om$. 
Hence, as $\sigma(\car_\Om \smallsetminus \car_\Om^*)>0$,  from \eqref{eq:lastmini} we conclude that $\|\hat \rho\|_{\mathcal{M}(X_\Om)} <  \|\rho\|_{\mathcal{M}(X_\Om)}$ contradicting \eqref{eq:minimalityofJ}.
\end{proof}

Next we show that we can disintegrate any measure obtained by the application of Theorem \ref{thm:lifting} into a family of Borel measures parametrized by $x\in \Om$ and concentrated on the set
\begin{equation}\label{eq:eerre}
E_{x} :=  \{(\gamma, h) \in \mathcal{H}^{v,g}_\Om \cap  \car^{*}_\Om \cap \Ha^1_\Om \,  : \, \gamma(0) = x \}\,.
\end{equation}
Notice that $E_x$ is measurable for every $x \in \Om$. Indeed, by employing similar arguments to the ones in Lemma \ref{lem:evaluationsecond}, one can show that the map $\pi : \car^{*}_\Om \smallsetminus \{0\}   \rightarrow  \Om $ defined as 
$\pi(\gamma,h) := \gamma(0)$ is continuous. Therefore, as $\car^{*}_\Om \cap \Ha^1_\Om \subset \car^{*}_\Om \smallsetminus \{0\}$, we can write $E_x = \pi^{-1}(x) \cap \mathcal{H}^{v,g}_\Om \cap \Ha^1_\Om$. Thus $E_x$ is measurable, given that $\Ha^1_\Om$ is closed and $\mathcal{H}^{v,g}_\Om$ is measurable by Lemma \ref{lem:cutoffoperator}.
\begin{lemma}\label{lem:disintegration}
Let  $v \colon X_\Om \rightarrow \R^d$, $g \colon X_\Om \rightarrow \R$ be measurable. Let $\rho \in \mathcal{D}_{v,g}$   
and $\sigma \in \mathcal{M}_1^+(\car_\Om)$ be such that \eqref{eq:representation} holds.
Then there exists a Borel family of measures $\{\sigma^{x}\}_{x \in \Om} \subset \mathcal{M}^+(\car_\Om)$  such that for every $f \in L_{\sigma}^1(\car_\Om)$ we have
\begin{equation}\label{eq:dis}
\int_{\mathcal{H}^{v,g}_\Om \cap \mathcal{H}^1_\Om \cap \car_\Om^*}f(\gamma,h)\, d\sigma(h,\gamma) = \int_{\Omega} \int_{E_x}f(\gamma,h)\, d\sigma^x(\gamma,h)\, d\rho_0(x)\,,
\end{equation}
where $E_x$ is defined as in \eqref{eq:eerre}.
Moreover $\sigma^{x}$ is concentrated on $E_{x}$ for $\rho_0$-a.e.~$x \in \Om$.
\end{lemma}

\begin{proof}
 Set $\Ha := \mathcal{H}^{v,g}_\Om \cap \mathcal{H}^1_\Om \cap \car_\Om^*$ and notice that it is measurable thanks to Lemma \ref{lem:cutoffoperator}.
 Consider  the map $\pi : \car_\Om   \rightarrow  \Om$ defined by
$\pi(\gamma,h) :=\gamma(0) \,\rchi_{\Ha}( \gamma,h) + z \, \rchi_{\car_\Om \smallsetminus \Ha}( \gamma,h)$,
where $z \in \Om$ is arbitrary, but fixed.
Notice that, as $h(0) > 0$ for every $(\gamma,h) \in \Ha$, the map $\pi$  is well-defined and measurable in $\car_\Om$ using similar arguments as in Lemma \ref{lem:evaluationsecond}. Define then $\tilde \sigma := \sigma \zak \Ha$.
We aim to show that $\pi_\# \tilde \sigma \ll   \rho_0$. To this end, consider a Borel set $B \subset \Om$ such that $\rho_0(B) = 0$. Define  $\tilde B = \{(\gamma,h) \in \Ha : \gamma(0) \in B\}$ and notice that $\tilde B$ is measurable as $\tilde B = \pi^{-1}(B) \cap \Ha$. 
Then
\begin{equation*}
0 =  \rho_0(B) = \int_{\car_\Om} h(0) \rchi_B(\gamma(0)) \, d\sigma \geq \int_{\Ha} h(0) \rchi_B(\gamma(0)) \, d\sigma = \int_{\tilde B} h(0) \, d\sigma \,,
\end{equation*}
implying that $\sigma(\tilde B) = 0$, since $h(0) > 0$ for all $(\gamma,h) \in \Ha$. By direct calculation we have $(\pi_\#\tilde \sigma)(B) = \sigma(\tilde B)$, and thus $(\pi_\#\tilde \sigma)(B) =0$, concluding that $\pi_\# \tilde \sigma \ll  \rho_0$.
Hence, as $\car_\Om$ is a complete separable metric space by Proposition \ref{prop:complete},
we can apply Theorem \ref{thm:app:disintegration} to $\tilde \sigma \in \mathcal{M}^+(\car_\Om)$, and obtain a Borel family of measures $\{\sigma^{x}\}_{x \in \Om} \subset \mathcal{M}^+(\car_\Om)$ satisfying the thesis.
\end{proof}

\subsection{Proof of Theorem \ref{thm:uniqueness}} \label{sec:proof_uniqueness}
Assume that $t \mapsto \rho_t$ belongs to $\mathcal{D}_{v,g}$. Moreover suppose that $\rho_0$ is concentrated on $A\subset \Om$ and that \eqref{eq:minimalityofJ} holds. 
By Theorem~\ref{thm:lifting}, there exists $\sigma \in \mathcal{M}_1^+(\car_\Om)$ concentrated on $\mathcal{H}^{v,g}_\Om \cap \Ha_\Om^1$ that represents $\rho_t$, that is, \eqref{eq:representation} holds. 
Using Lemma \ref{lem:minimalityimpliesconcentration}, we infer that $\sigma$ is concentrated on $\Ha := \mathcal{H}^{v,g}_\Om \cap \mathcal{H}^1_\Om \cap \car_\Om^*$. 
Thanks to Lemma~\ref{lem:disintegration} we can disintegrate $\sigma$ into a Borel family $\{\sigma^{x}\}_{x \in \Om} \subset \mathcal{M}^+(\car_\Om)$ such that \eqref{eq:dis} holds, with $\sigma^x$ concentrated on $E_x$ for $\rho_0$-a.e.~$x \in \Om$. 
We claim that  assumption $(\textsc{Hyp})$  implies that
$E_x$ contains at most one point for all $x \in A$. 
Indeed, suppose that $(\gamma_1^x,h_1^x), (\gamma_2^x,h_2^x) \in E_x$. As $(\gamma_i^x,h_i^x) \in \car_\Om^* \cap \Ha_\Om^1$,  there exist $\tau_i \in \R$ such that $\{h_i > 0\} = [0,1] \cap (-\infty,\tau_i)$ and $\|h^x_i\|_{1} = 1$. Assume $\tau_1 \leq \tau_2$. As $(\gamma_i^x,h_i^x) \in \Ha^{v,g}_\Om$, we  have that $(\gamma_i^x,h_i^x)$ solves \eqref{ODE1}-\eqref{ODE2} in $[0,\tau_1)$.  
Now notice that by linearity of \eqref{ODE2} and assumption $(\textsc{Hyp})$, we have that $\gamma_1^x(t) = \gamma_2^x(t)$ and $h_1^x(t) = h_2^x(t) h_1^x(0)/h_2^x(0)$ for all $t \in [0,\tau_1)$.  
As $\|h^x_i\|_{1} = 1$, we then infer $(\gamma_1^x,h_1^x) = (\gamma_2^x,h_2^x)$ in $[0,\tau_1)$ and by the continuity of $h_i$ we also obtain that $h^x_1(\tau_1) = h^x_2(\tau_2) = 0$. By definition of $\tau_i$ we conclude that $\tau_1 = \tau_2 = \tau$, so that $(\gamma_1^x,h_1^x) = (\gamma_2^x,h_2^x)$ in $[0,\tau)$. Since $h^x_1(t) = h^x_2(t) = 0$ for all $t \geq \tau$, we conclude that $E_x$ contains at most one point. 
Thus, for $\rho_0$-a.e.~$x\in E$, $ E:=\{ x \in \Om \, \colon \, E_x \neq \emptyset\}$,  we have  $\sigma^x =c_x\, \delta_{(\gamma^x,h^x)}$, with $c_x:= \|\sigma^x\|_{\mathcal{M}(\car_\Om)}$, $(\gamma^x,h^x) \in E_x$. 
  We claim that $c_x =1/h^x(0)$. Indeed, by definition of $E_x$, we have $\gamma^x(0)=x$. Using \eqref{eq:representation}, \eqref{eq:dis}, and $\sigma(\car_\Om \smallsetminus \Ha)=0$, we then obtain  
\begin{align*}
\int_\Om \f(x)  \, d\rho_0(x)  =  \int_E c_x h^x(0)\f(x)\, d\rho_0(x) \,,
\end{align*}     
for all $\f \in C(\Om)$, showing that $c_x h^x(0) = 1$ for $\rho_0$-a.e. $x\in E$.
Again by \eqref{eq:representation} and \eqref{eq:dis} we get
\[
 \int_\Om \f(x)\, d \rho_t (x) = \int_{E} \frac{1}{h^x(0)} h^x(t)\f(\gamma^x(t))\, d\rho_0(x)\,,
 \]
for every $\f \in C(\Om)$, where we also used that $\sigma^x = \frac{1}{h^x(0)}\delta_{(\gamma^x,h^x)}$ for $\rho_0$-a.e. $x\in E$. Thus $\rho_t$ depends only on the initial data $\rho_0$, ending the proof.

\section{Extremal points of the Wasserstein-Fisher-Rao energy}\label{sec:extremalpoints}

Let $\Om \subset \R^d$ with $d \geq 1$ be the closure of a bounded domain of $\R^d$.  
Let $\alpha,\beta>0$, $\delta \in (0,\infty]$ and define $\mathscr{B}$ to be the unit ball of the functional $J_{\alpha,\beta,\delta}$ defined at \eqref{bb reg}, that is,
\[
\mathscr{B}:= \left\{ (\rho,m,\mu ) \in \mathcal{M}_\Om \, \colon \, J_{\alpha,\beta,\delta}(\rho,m,\mu) \leq 1 \right\} \,.
\]
The aim of this section is to characterize the extremal points $\ext \mathscr{B}$. Notice that $J_{\alpha,\beta,\infty}$ corresponds to the coercive version of the Benamou-Brenier energy, whose extremal points were characterized in \cite{bcfr}. Hence here we focus on the case $\delta<\infty$. After the characterization of $\ext \mathscr{B}$ is obtained, we will show how this information can be applied to the analysis of dynamic inverse problems which are regularized via the optimal transport energy $J_{\alpha,\beta,\delta}$ \cite{bf}. In particular we will obtain a sparse representation formula for regularized solutions to the dynamic problem.

Before stating the characterization theorem we remind the reader the notations $\cone_\Om,\car_\Om, \Ha_\Om$ introduced at \eqref{eq:cone}, \eqref{eq:narr}, \eqref{eq:char_smooth}. In the following $\car_\Om$ is equipped with the distance $\D$ at \eqref{distance_sup}, making it a complete metric space (Proposition \ref{prop:complete}). We now define the set of characteristics of \eqref{cont weak} with energy $J_{\alpha,\beta,\delta}=1$, which will play a role in the characterization of $\ext \mathscr{B}$.

\begin{definition}[Characteristics] \label{def:extrfr}
	Define the set $\mathcal{C}$ of all the triples $(\rho,m,\mu) \in \mathcal{M}_\Om$ of the form $\rho= h(t)\, dt \otimes \delta_{\gamma(t)}$, $m=\dot{\gamma}(t) \rho$, $\mu=\dot{h}(t)\,  dt \otimes \delta_{\gamma(t)}$ that satisfy the following properties:
\begin{itemize}
\item[i)] $t \mapsto h(t) \delta_{\gamma(t)}$ belongs to $\Ha_\Om$, 
\item[ii)] the set $\{h>0\}:=\{ t \in [0,1] \, \colon \, h(t)>0\}$ is connected,
\item[iii)] the energy satisfies $J_{\alpha,\beta,\delta} (\rho,m,\mu)=1$.
\end{itemize}
\end{definition}

The above definition is well-posed since $(\rho,m,\mu)$ belongs to $\M_\Om$ and solves the continuity equation \eqref{cont weak} in $X_\Om$ (by the converse of Proposition \ref{prop:sublevelJ} with $V=\Om$). Hence  (iii) is compatible with the definition of $J_{\alpha,\beta,\delta}$. 

\begin{remark}\label{rem:inclusion}
If $(\rho,m,\mu) \in \M_\Om$ with $\rho \in \mathcal{H}_\Om$, then an application of Proposition \ref{prop:sublevelJ} (with $V=\Om$) yields the representation
\begin{equation} \label{def:extremal:4}
 J_{\alpha,\beta,\delta} (\rho,m,\mu)=  J_{\alpha,\beta,\delta}(\gamma,h) = \int_{\{h >0\}} \frac{\beta}{2} \,  h(t)  |\dot\gamma(t)|^2 + \frac{\beta \delta^2}{2}  \,
 \frac{\dot{h}(t)^2}{h(t)} + \alpha  h(t) \, dt \,. 
\end{equation}
In particular $J_{\alpha,\beta,\delta}$ is $\D$-measurable, as a consequence of Proposition \ref{lem:compact_sublevels}. For a  measurable set $E \subset [0,1]$ we define the localized energy
\[
 J_{\alpha,\beta,\delta,E} (\rho,m,\mu) :=  \int_{E \cap\{h >0\}}   \frac{\beta}{2} \, h(t)  |\dot\gamma(t)|^2 + \frac{\beta \delta^2}{2} \,
 \frac{\dot{h}(t)^2}{h(t)} +  \alpha h(t) \, dt \,.
\]
\end{remark}

We are now ready to state the characterization theorem.

\begin{theorem} \label{thm:fr}
For parameters $\alpha,\beta,\delta>0$ we have
$
\ext \mathscr{B} = \mathcal{C} \cup \{0\}$, 	
where $0$ denotes the null triple in $\M_\Om$. 
\end{theorem}

The proof of Theorem \ref{thm:fr} will be carried out in the next section, while in Section \ref{sec:ext_sparse} we will detail the application of Theorem \ref{thm:fr} to dynamic inverse problems.

\subsection{Proof of Theorem \ref{thm:fr}}

In order to simplify notations, we will denote $J:=J_{\alpha,\beta,\delta}$ and $J_E:=J_{\alpha,\beta,\delta,E}$ for any $E\subset [0,1]$  measurable.  

\medskip

\textit{Step 1.}  $\mathcal{C} \cup \{0\} \subset \ext\mathscr{B}$:
Assume first that $(\rho,m,\mu)=0$, and that there exists a decomposition
\begin{equation} \label{eq:zero_decomposition}
(0,0,0) = \lambda (\rho^1,m^1, \mu^1) + (1- \lambda) (\rho^2,m^2, \mu^2)
\end{equation}
with $(\rho^j,m^j,\mu^j) \in \Bb$ and $\lambda \in (0,1)$. In particular by Lemma \ref{lem:prop B} point (i) we have $\rho^j \geq 0$ and $m^j,\mu^j \ll \rho^j$. Therefore \eqref{eq:zero_decomposition} immediately implies that $(\rho^j,m^j,\mu^j)=0$, showing that $0 \in \ext \Bb$. 

Assume now that $(\rho,m,\mu) \in \mathcal{C}$, according to Definition \ref{def:extrfr}. 
In particular the set $\{h>0\}$ is non-empty, since $J(\rho,m,\mu)=1$. 
Assume that $(\rho^1,m^1,\mu^1), (\rho^2,m^2,\mu^2) \in \mathscr{B}$ are such that 
\begin{equation} 
\label{thm:fr:1}
	(\rho,m,\mu)=\lambda (\rho^1,m^1,\mu^1) + (1-\lambda) (\rho^2,m^2,\mu^2) \,,
\end{equation}
for some $\lambda \in (0,1)$. We need to show that $(\rho,m,\mu)=(\rho^j,m^j,\mu^j)$. By \eqref{thm:fr:1}, convexity of $J$ (see Lemma \ref{lem:prop J}), and the fact that $J(\rho^j,m^j,\mu^j)\leq 1$, $J(\rho,m,\mu)=1$, we have that $J(\rho^j,m^j,\mu^j) =1$. Thus, by Lemmas \ref{lem:prop cont}, \ref{lem:prop B} we infer $\rho^j=dt \otimes \rho_t^j$ with $t \mapsto \rho^j_t$ in $\pcurves$. Set $h_j(t):=\rho_t^j (\Om)$ and notice that $h_j$ is continuous by narrow continuity of $\rho^j$. From the decomposition \eqref{thm:fr:1} and the uniqueness of the disintegration, we thus obtain 
$\rho^j_t = h_j(t) \delta_{\gamma(t)} \in \car_\Om$,
and in particular
\begin{equation} \label{thm:fr:9}
h(t)=\lambda h_1(t) + (1-\lambda)h_2(t) \,\, \text{ for every } \,\,
t \in [0,1] \,.
\end{equation}
We will now show that there exists $c>0$ such that
\begin{equation} \label{thm:fr:17}
h_2(t) = c \, h_1(t) \,\, \text{ for all } \,\, t \in \{h>0\}\,.
\end{equation}

We start by defining the sets
$E:= \{ h_1>0\} \cap \{h_2>0\}$, $   
Z_1:= \{ h_1>0\} \cap \{ h_2=0\}$ and $Z_2:= \{ h_1=0\} \cap \{ h_2>0\}$.
These sets are pairwise disjoint, and by \eqref{thm:fr:9} we have 
$\{h>0\} = E \cup Z_1 \cup Z_2$, 
where we recall that $\{h>0\} \neq \emptyset$ is connected by assumption.
We claim that $E \neq \emptyset$. Indeed, assume by contradiction that $E = \emptyset$, so that in particular $Z_1 \cup Z_2 = \{h > 0\}$. Notice that $Z_1, Z_2$ are relatively closed in $\{h>0\}$ since they can be written as $Z_1=\{h>0\} \cap \{h_2=0\}$, $Z_2 = \{h>0\} \cap \{h_1=0\}$, due to \eqref{thm:fr:9}. As $\{h>0\}$ is connected, we deduce that either $Z_1 =\emptyset$ or $Z_2=\emptyset$. If $Z_1=\emptyset$, then we would have $h_1 \equiv 0$, which in turn would imply $\rho^1=0$. Hence by Lemma \ref{lem:prop B} point (i) we would obtain $J(\rho^1,m^1,\mu^1)=0$,  contradicting $J(\rho^1,m^1,\mu^1)=1$. Similarly $Z_2 =\emptyset$ leads to the contradiction $J(\rho^2,m^2,\mu^2)=0$. We therefore conclude $E \neq \emptyset$. %

\medskip
\noindent \textit{Claim:} $h_1/h_2$ is constant in each connected component of $E$.

\smallskip
\noindent \textit{Proof of Claim:}  %
Since $J(\rho^j,m^j,\mu^j) < \infty$, by Proposition \ref{prop:sublevelJ}, we have that $\rho^j \in \mathcal{H}_\Om$ and there exist $v^j\colon X_\Om \to \R^d$, $g^j \colon X_\Om \to \R$ measurable such that $m^j=v^j\rho^j$, $\mu^j = g^j\rho^j$ and
\begin{gather} \label{thm:fr:6}
\dot{h}_j (t) = g^j (t,\gamma(t)) h_j(t) \,\, \text{ for a.e. } \,\, t \in (0,1) \,, \\
\label{thm:fr:3}
\dot\gamma(t) = v^j(t,\gamma(t)) \, \text{ a.e.~in } \, \{h_j>0\}\,. 
\end{gather}
Moreover $J(\rho^j,m^j,\mu^j)=J(h_j,\gamma)$ can be computed via \eqref{def:extremal:4}.  %
By direct calculation, and using \eqref{thm:fr:1} and \eqref{def:extremal:4} we have %
\[
\begin{aligned}
& J_E(\rho,m,\mu)=  J_E(h,\gamma)= J_E(\lambda h_1  +  (1-\lambda)h_2,\gamma)  \\
&  =  \int_E ( \lambda h_1 + (1-\lambda )h_2) \left( \frac{\beta}{2} \, |\dot\gamma|^2 + \alpha  \right) \, dt + \frac{\beta \delta^2}{2}\int_E \frac{ (  \lambda \dot{h}_1 + (1-\lambda) \dot{h}_2)^2  }{\lambda h_1 + (1-\lambda) h_2} \, dt \\ 
& = \lambda J_E(h_1,\gamma) + (1-\lambda) J_E(h_2,\gamma) + \frac{\beta\delta^2}{2}\int_E \frac{ (  \lambda \dot{h}_1 + (1-\lambda) \dot{h}_2)^2  }{\lambda h_1 + (1-\lambda) h_2} - \lambda \frac{\dot{h}_1^2}{h_1} - (1-\lambda)\frac{\dot{h}_2^2}{h_2} \, dt \,,  \\
\end{aligned}
\]
so that
\begin{equation} \label{thm:fr:12}
\begin{aligned} 
J_E(\rho,m,\mu)  = \lambda J_E(\rho^1,m^1,\mu^1)  + & (1-\lambda) J_E(\rho^2,m^2,\mu^2) \\
& - \frac{\beta \delta^2}{2} \lambda (1-\lambda)   \int_E \frac{ (\dot{h}_1 h_2 - h_1\dot{h}_2)^2 }{(\lambda h_1 + (1-\lambda)h_2) h_1 h_2} \, dt \,.
\end{aligned}
\end{equation}
By proceeding as above, one can check that 
\begin{equation} \label{thm:fr:13}
J_{Z_1}(\rho,m,\mu)=\lambda J_{Z_1} (\rho^1,m^1,\mu^1) \,\,, \quad  J_{Z_2}(\rho,m,\mu)=(1-\lambda) J_{Z_2} (\rho^2,m^2,\mu^2) \,,
\end{equation}
where we used  \eqref{thm:fr:1}, \eqref{thm:fr:9}, definition of $Z_j$ and \cite[Theorem 4.4]{evansgariepy}. 
Moreover by definition
\begin{gather} \label{thm:fr:14}
J(\rho^1,m^1,\mu^1) = J_{Z_1} (\rho^1,m^1,\mu^1) + J_{E} (\rho^1,m^1,\mu^1) \,,\\
J(\rho^2,m^2,\mu^2) = J_{Z_2} (\rho^2,m^2,\mu^2) + J_{E} (\rho^2,m^2,\mu^2) \,. \label{thm:fr:15}
\end{gather}
By combining \eqref{thm:fr:12}-\eqref{thm:fr:15}, we obtain
\[
\begin{aligned}
J(\rho,m,\mu) = \lambda  J (\rho^1,m^1,\mu^1) + & (1-\lambda)  J (\rho^2,m^2,\mu^2)  \\
& - \frac{\beta\delta^2}{2} \lambda (1-\lambda)   \int_E \frac{ (\dot{h}_1 h_2 - h_1\dot{h}_2)^2 }{(\lambda h_1 + (1-\lambda)h_2) h_1 h_2} \, dt \,.
\end{aligned}
\]
Now we can make use of the fact that $J(\rho,m,\mu)=J(\rho^j,m^j,\mu^j)=1$ to infer $\dot{h}_1 h_2 = h_1\dot{h}_2$ a.e.~in $E$. In particular  $(h_1/h_2)'=0$ a.e.~in $E$, and hence the claim follows.

\medskip

We are now ready to show \eqref{thm:fr:17}. For an arbitrary $C>0$ and $t \in \{h>0\}$ define the map
\[ f(t):=
	\min \left(\displaystyle \frac{h_1(t)}{h_2(t)} , C \right) \rchi_{\{h_2>0\}} + C\, \rchi_{\{h_2=0\}} \,.
\]
Notice that $f$ is continuous and, since $E \neq \emptyset$, $f$ is not identically zero. Moreover, as $(h_1/h_2)'=0$ a.e.~in $E$,  the image $f(\{h>0\})$ is at most countable. Assume by contradiction that $Z_2 \neq \emptyset$, and notice that $f$ vanishes on $Z_2$. Therefore $f$ assumes at least two different values on $\{h>0\}$, which is a contradiction, as $f(\{h>0\})$ is connected and, consequently, uncountable. Hence $Z_2 = \emptyset$ and $\{h>0\}=E \cup Z_1$. By interchanging the roles of $h_1$ and $h_2$, we can repeat the same argument and conclude that $Z_1 = \emptyset$, so that $E=\{h>0\}$. As $\{h>0\}$ is connected, we thus deduce \eqref{thm:fr:17} directly from the fact that $(h_1/h_2)'=0$ a.e.~in $E$.

\medskip

We are now ready to conclude. Indeed, note that, as $\rho^j_t=h_j(t)\delta_{\gamma(t)}$, condition \eqref{thm:fr:17} implies that 
$\rho^2=c  \rho^1$ and $\{h_j>0\}=\{h>0\}$. In particular \eqref{thm:fr:3} yields $v^j(t,\gamma(t))=\dot \gamma(t)$  a.e.~in $\{h>0\}$, showing that $m^2=c \, m^1$. Finally from \eqref{thm:fr:6} we infer 
$g^1(t,\gamma(t))=g^2(t,\gamma(t))$ a.e.~in $(0,1)$, from which we conclude $\mu^2 = c \, \mu^1$. In total we have 
$(\rho^2,m^2,\mu^2)=c\, (\rho^1,m^1,\mu^1)$, and by $J(\rho^j,m^j,\mu^j)=1$ and one-homogeneity of $J$ we conclude that $c=1$. Therefore \eqref{thm:fr:1} yields extremality of $(\rho,m,\mu)$. 

\medskip

\textit{Step 2.}  $\ext\mathscr{B} \subset \mathcal{C} \cup \{0\}$:
Let $(\rho,m,\mu) \in \ext\mathscr{B}$. We can assume that $(\rho,m,\mu) \neq 0$, so that $J(\rho,m,\mu)>0$. By extremality of $(\rho,m,\mu)$, convexity and 1-homogeneity of $J$, we conclude that $J(\rho,m,\mu)=1$. In particular by Lemma \ref{lem:prop B} we obtain $\rho \geq 0$ and $m=v\rho$, $\mu=g \rho$ for some measurable maps $v \colon X_\Om \to \R^d$, $g \colon X_\Om \to \R$ satisfying
\begin{equation} \label{eq:equal1}
J(\rho,m,\mu) =  \int_0^1 \int_\Om \left( \frac{\beta}{2} |v(t,x)|^2 +  \frac{\beta \delta^2}{2} |g(t,x)|^2 + \alpha \right) \, d\rho_t(x) \, dt = 1 \,. 
\end{equation}
By definition of $J$, we then have that $\de_t \rho_t + \div (v\rho_t )=g\rho_t$ in $X_\Om$. 
Thanks to Lemma~\ref{lem:prop cont} we also have $\rho=dt \otimes \rho_t$ with $t \mapsto \rho_t$ in $\pcurves$. Set $h:=\rho_t(\Om)$, and recall that $h$ is continuous. We first prove the following claim.

\smallskip
\noindent \textit{Claim:} $\supp \rho_t$ is a singleton for every $t \in \{h>0\}$.

\smallskip
\noindent \textit{Proof of Claim:} 
Assume by contradiction that there exists $\hat t \in \{h>0\}$ such that $\supp \rho_{\hat t}$ is not a singleton. Then there exist disjoint Borel sets $E_1, E_2 \subset \Om$ such that $E_1 \cup E_2 =\Om$ and $\rho_{\hat t}(E_i) > 0$ for $i=1,2$. 
Invoking Theorem \ref{thm:lifting}, there exists a measure $\sigma \in \M^+_1(\mathscr{S}_\Om)$ concentrated on $\Ha^{v,g}_\Om$ which represents $\rho_t$, that is, \eqref{eq:representation} holds. 
Define the sets
\begin{align*}
A_i := \{(\gamma,h) \in \mathscr{S}_\Om : \gamma(\hat t) \in E_i,\, h(\hat t) > 0\} \,, \,\,\,\,\,\,
Z:=\{ (\gamma,h) \in \mathscr{S}_\Om :h(\hat t) = 0 \}\,,
\end{align*}
and notice that $A_1,A_2,Z$ are pairwise disjoint and $\car_{\Om}=A_1\cup A_2 \cup Z$. Also $Z$ is $\D$-measurable, being $\D$-closed, as it is readily seen by \eqref{eq:char_flat}. We claim that also $A_i$ is $\D$-measurable. To this end define the maps $e_t \colon \car_{\Om} \to \cone_\Om$ with $e_t(\rho):=\rho_t$ and $\pi \colon \cone_{\Om} \to \R^d$ where 
$
\pi( \gamma,h):=\gamma \, \rchi_{\cone_\Om \smallsetminus \{0\}} (\gamma,h) + p \, \rchi_{\{0\}} (\gamma,h)$,
with $p \in \R^d \smallsetminus \Om$ arbitrary but fixed. Notice that by construction $e_t$ is continuous from $(\car_\Om,\D)$ into $(\cone_\Om,\D_F)$. Moreover $\pi$ is measurable since the map $(\gamma,h) \mapsto \gamma$ is $\D_F$-continuous in $\cone_\Om \smallsetminus \{0\}$. Since $A_i=( \pi \circ e_{\hat{t}} )^{-1} (E_i)$, we have that $A_i$ is measurable. By applying \eqref{eq:representation}, we get
\begin{align}\label{eq:positivityweights1}
0 < \rho_{\hat t}( E_i) = \int_{\car_\Om}  h(\hat t) \rchi_{E_i} (\gamma(\hat t))\, d\sigma(\gamma,h) = \int_{A_i} h(\hat t)\, d\sigma(\gamma,h)\,,
\end{align}
which implies $\sigma(A_i) > 0$. Hence setting $\Sigma_1:=A_1$, $\Sigma_2:=A_2 \cup Z$ we obtain a measurable partition of $\car_\Om$ with $\sigma (\Sigma_i)>0$. Notice now that the map $\Psi(t,x):= \beta  |v(t,x)|^2/2 + \beta \delta^2 |g(t,x)|^2/2+\alpha$ belongs to $L^1_{\rho_t}(\Om)$ for a.e.~$t \in (0,1)$, thanks to \eqref{eq:equal1}. Moreover $J$ is non-negative and \mbox{$\D$-measurable} by Remark~\ref{rem:inclusion}. Since $\sigma$ is concentrated on $\Ha_\Om^{v,g}$, we can apply Remark~\ref{rem:representationL1} to $\Psi$ and obtain
\begin{equation} \label{eq:equal1bis}
\begin{aligned}
\int_{\car_\Om} J(\gamma,h) \, d\sigma(\gamma,h)  = \int_{\car_\Om} \int_0^1 h(t)\Psi(t,\gamma(t)) \, dt \, d\sigma(\gamma,h) 
 =  \int_0^1 \int_\Om \Psi(t,x) \, d\rho_t(x) \, dt = 1\,,
\end{aligned}
\end{equation}
where in the last equality we again used \eqref{eq:equal1}. 
Define the coefficients 
$\lambda_i := \int_{\Sigma_i} J (\gamma,h)\, d\sigma(\gamma,h)$. 
From  \eqref{eq:equal1bis} we infer $0\leq \lambda_1,\lambda_2\leq 1$ and $\lambda_1+\lambda_2 =  1$. We claim that $\lambda_i >0$. Indeed, the map $f_i(t):=\int_{\Sigma_i} h(t) \, d\sigma(\gamma,h)$ for $t \in [0,1]$ is continuous by dominated convergence and the fact that $\int_{\car_\Om} \nor{h}_{\infty} \, d\sigma(\gamma,h)  < \infty$, as $\sigma \in \M^+_1(\car_\Om)$.  Notice that by construction $f_i(\hat{t})>0$. Therefore by definition of $J$ and continuity of $f_i$ we have
\[
\lambda_i = \int_{\Sigma_i}J (\gamma,h)\, d\sigma(\gamma,h) \geq \int_{\Sigma_i}\int_0^1 h(t)  \, dt \, d\sigma(\gamma,h) = \int_0^1 f_i(t)\, dt >0 \,,
\]   
as claimed. The measure $\sigma \zak \Sigma_i$ satisfies the hypothesis of the converse in Theorem \ref{thm:lifting}, given that \eqref{eq:equal1bis} holds and $\sigma$ is concentrated on $\Ha_\Om^{v,g}$. Hence, the curve  $t \mapsto \rho_t^i$ defined by  
\begin{align}\label{eq:representationrestriction}
\int_\Om \f(x) \, d\rho^i_{t}(x) :=  \int_{\Sigma_i}  h(t) \f(\gamma(t)) \, d\sigma(\gamma,h)  \,, \,\,\, \text{ for all } \,\,\, \f \in C(\Om)  
\end{align}
belongs to $\pcurves$ and solves the continuity equation with $v$ and $g$. We can now define $(\rho^i,m^i,\mu^i) \in \M$ by setting 
$\rho^i:= dt \otimes \rho_t^i$, $m^i:= v \rho^i$, $\mu^i:=g  \rho^i$. 
Note that by \eqref{eq:representation} and \eqref{eq:representationrestriction} we have that $\rho_t^i \leq \rho_t$ for every $t\in [0,1]$. Hence 
\[
\int_0^1 \int_\Om \left( \frac{\beta}{2} |v(t,x)|^2 + \frac{\beta \delta^2}{2} |g(t,x)|^2 +\alpha \right) \,  d\rho^i_t (x)  \, dt  \leq 1 \,,
\]
by \eqref{eq:equal1}. Given that the above holds, by repeating the same arguments used to prove \eqref{eq:equal1bis}, but applied to $\rho_t^i$ and $\sigma \zak \Sigma_i$, we have that
$J(\rho^i,m^i,\mu^i) = \lambda_i$. 
Consider the decomposition
\begin{equation} \label{eq:dec_ext}
(\rho,m,\mu) = \lambda_1 \, \frac{1}{\lambda_1} (\rho^1,m^1,\mu^1)+ 
\lambda_2 \, \frac{1}{\lambda_2} (\rho^2,m^2,\mu^2)\,,
\end{equation}
and notice that $\lambda_i^{-1}(\rho^i,m^i,\mu^i) \in \mathscr{B}$ thanks to the condition $J(\rho^i,m^i,\mu^i) = \lambda_i$ and to the \mbox{one-homogeneity} of $J$. We assert that
\begin{equation} \label{eq:dec_ext2}
 \frac{1}{\lambda_1} (\rho^1,m^1,\mu^1) \neq  
\lambda_2 \, \frac{1}{\lambda_2} (\rho^2,m^2,\mu^2) \,.
\end{equation}
Indeed we have that $\lambda_1^{-1}\rho^1 \neq \lambda_2^{-1}\rho^2$: If they were equal then by narrow continuity we would have $\lambda_1^{-1}\rho_{\hat t}^1 = \lambda_2^{-1}\rho_{\hat t}^2$. However by \eqref{eq:positivityweights1} it is immediate to check that $\rho^1_{\hat t} (E_1) = \rho_{\hat t}(E_1)>0$ and $\rho^2_{\hat t} (E_1)=0$, yielding a contradiction. Thus \eqref{eq:dec_ext2} holds and \eqref{eq:dec_ext} gives a non-trivial convex decomposition of $(\rho,m,\mu)$, contradicting extremality. This proves the claim.

\smallskip

In particular, we have shown that $\rho_t = h(t) \delta_{\gamma(t)}$ for some $\gamma \colon [0,1] \to \Om$, $h\geq 0$. Thus $t \mapsto \rho_t$ belongs to $\car_\Om$, being narrowly continuous. Hence $\gamma \in C(\{h>0\};\R^d)$ thanks to Lemma~\ref{lem:narrow_cone}.
Moreover, as a consequence of \eqref{eq:equal1} and Proposition \ref{prop:sublevelJ}, we have that $t \mapsto \rho_t$ belongs to 
$\mathcal{H}_\Om$, $m=\dot{\gamma} \rho$, $\mu=\dot{h}(t) \, dt \otimes \delta_{\gamma(t)}$ and
\begin{equation}\label{eq:ext_bb}
J(\rho,m,\mu) =  \int_{\{h >0\}}  \frac{\beta}{2} h(t)  |\dot\gamma(t)|^2 + \frac{\beta \delta^2}{2} \,
 \frac{\dot{h}(t)^2}{h(t)} +  \alpha h(t) \,  dt =1	\,.
\end{equation}

In order to prove that $(\rho,m,\mu) \in \mathcal{C}$ we are left to show that the set $\{h>0\}$ is connected. To this end, assume by contradiction that $\{h>0\} = E_1 \cup E_2$ with $E_1$, $E_2$ relatively open, non-empty and disjoint. For $t \in [0,1]$ set
$\rho^i_t := h(t)\rchi_{E_i}(t) \delta_{\gamma(t)}$.
Note that as $\{h>0\}$ is relatively open we have that $\partial^{\{h>0\}} E_i = \partial^{[0,1]}E_i \cap \{h>0\}$ where we denote by $\partial^A$ the relative boundary with respect to the set $A$. Hence as $\partial^{\{h>0\}} E_i = \emptyset$ we deduce that $h(t) = 0$ for every $t\in \partial^{[0,1]}E_i$. In particular the map $t \mapsto h(t)\rchi_{E_i}(t)$ is continuous in $[0,1]$. Moreover $\gamma \in C(\{h\rchi_{E_i}>0\};\R^d)$, hence Lemma \ref{lem:narrow_cone} ensures that the curve $t\mapsto \rho_t^i$ belongs to $\car_{\Om}$. We claim that $t\mapsto \rho_t^i$ belongs to $\Ha_\Om$. In order to show this, we make use of the information $(t \mapsto \rho_t) \in \Ha_\Om$. Notice that the set $E_i$ is relatively open in $[0,1]$, given that $\{h>0\}$ is open. Thus $E_i  = \bigcup_{n=1}^\infty I_n$, where $\{I_n\}_n$ are pairwise disjoint intervals in $[0,1]$. By dominated convergence 
\[
\int_0^1 h(t)\rchi_{E_i}(t) \dot{\f}(t) \, dt = \sum_{n=1}^\infty \int_{I_n} h(t)  \dot{\f}(t) \, dt = - \sum_{n=1}^\infty\int_{I_n} \dot{h}(t) \f(t) \, dt = \int_0^1 \dot{h}(t)\rchi_{E_i}(t) \f(t)  \, dt 
\]
for every $\f \in C^1_c(0,1)$, where we used that $h=0$ on $\de^{[0,1]} I_n$, given that $\de^{[0,1]} I_n \subset \de^{[0,1]} E_i$. Since $h \in \AC^2[0,1]$, we infer that $h\rchi_{E_i} \in \AC^2[0,1]$, with derivative $\dot{h}\rchi_{E_i}$. Noticing that $\sqrt{h\rchi_{E_i}} = \sqrt{h}\rchi_{E_i}$ by similar arguments we also deduce that $\sqrt{h} \in \AC^2[0,1]$ and $\sqrt{h} \rchi_{E_i}  \gamma \in \AC^2([0,1];\R^d)$, thus concluding $(t \mapsto \rho_t^i) \in \Ha_\Om$. Set
\[
\rho^i:= \rchi_{E_i}(t)h(t) \,  dt \otimes \delta_{\gamma(t)} \,, \,\,\, m^i:= \dot{\gamma}(t) \rho^i \,, \,\,\, \mu^i  =  \rchi_{E_i}(t)\dot{h}(t) dt \otimes \delta_{\gamma(t)}\,.
\]
Thanks to Proposition \ref{prop:sublevelJ} we have that $(\rho^i,m^i,\mu^i)$ belongs to $\M_\Om$ and
\[
J(\rho^i,m^i,\mu^i) =  \int_{E_i}  \frac{\beta}{2} h(t)  |\dot\gamma(t)|^2 + \frac{\beta \delta^2}{2} \,
 \frac{\dot{h}(t)^2}{h(t)} +  \alpha h(t) \,  dt < \infty \,.
\]
Set $\lambda_i := J(\rho^i,m^i,\mu^i)$ and notice that $0 < \lambda_i < 1 $, $\lambda_1 + \lambda_2=1$ thanks to \eqref{eq:ext_bb} and definition of $E_i$. By construction we have $\rchi_{E_1}+\rchi_{E_2}=1$ in $\{h>0\}$. By recalling that $\dot h=0$ a.e.~in $\{h>0\}$, we have that a decomposition of the form \eqref{eq:dec_ext} holds. 
As $\lambda^{-1}_1 (\rho^1,m^1,\mu^1) \neq \lambda^{-1}_2 (\rho^2,m^2,\mu^2)$ and $\lambda^{-1}_i (\rho^i,m^i,\mu^i) \in \mathscr{B}$, this contradicts the extremality of $(\rho,m,\mu)$. Thus we conclude that the set $\{h>0\}$ must be connected, ending the proof.

\subsection{Sparsity for dynamic inverse problems with optimal transport regularization} \label{sec:ext_sparse}

In this section we analyze the problem of reconstructing a family of time-dependent Radon measures given a finite number of observations. More precisely, let $H$ be a finite dimensional Hilbert space and $K: \curves \rightarrow H$ be a linear operator which is continuous in the following sense: 
given a sequence $\{(t \mapsto \rho^n_t)\}_n$ in $\curves$, we require that
\begin{equation}\label{eq:topologycweak}
\rho_t^n  \to \rho_t \quad \text{ narrowly in } \,\, \M(\Om)  \,\,\text{ for all }\,\, t\in [0,1] \,\, \text{ implies } \,\, K\rho^n \rightarrow K\rho \ \text{ in } H\,.
\end{equation}
For a given datum $y \in H$, we aim at finding a solution $\rho \in \curves$ to the ill-posed inverse problem
\begin{equation} \label{appl:inverse}
K \rho = y \,. 
\end{equation}
We regularize \eqref{appl:inverse} via the Hellinger-Kantorovich-type energy $J_{\alpha,\beta,\delta}$ defined at \eqref{bb reg}, following the approach in \cite{bf}. To this end, introduce the space
\[
\widetilde{\mathcal{M}}_\Om:= \curves \times \mathcal{M}(X_\Om;\R^d)\times \mathcal{M}(X_\Om;\R)\,,
\]  
and define the Tikhonov functional $G:\widetilde{\mathcal{M}}_\Om \to \R \cup \{\infty\}$ by
\begin{equation}\label{eq:inversefunctional}
G(\rho,m,\mu) :=F(K\rho) + J_{\alpha,\beta,\delta}(\rho,m,\mu) \,, 
\end{equation}
where $F: H \to \R$ is a fidelity functional assumed to be convex, lower semi-continuous and bounded from below. %
We then replace \eqref{appl:inverse} by
\begin{equation}\label{eq:inverseproblem}
\min_{(\rho,m,\mu) \in \widetilde{\mathcal{M}}_\Om } \ G(\rho,m,\mu) \, .
\end{equation}
Note that $G$ is proper, since $J_{\alpha,\beta,\delta}(0,0,0)=0$. 
Moreover under the assumptions on $K$ and $F$, problem \eqref{eq:inverseproblem} admits a solution: This is indeed an immediate consequence of the direct method and of Lemma \ref{lem:prop J}.

It is well-known that the finite-dimensionality of the data space $H$ promotes sparsity in the reconstruction of solutions to \eqref{appl:inverse}, in the sense that there exists a minimizer to \eqref{eq:inverseproblem} which is finite linear combination of extremal points of the ball of the regularizer. This observation was recently made rigorous in the works \cite{chambolle,bc} (see also \cite{unser2,unsersplines}). Since in Theorem \ref{thm:fr} we characterized the extremal points of the ball of $J_{\alpha,\beta,\delta}$, we can specialize the representation results in \cite{chambolle,bc} to our setting, and obtain the following statement for sparse minimizers to \eqref{eq:inverseproblem}.

\begin{theorem}\label{thm:sparserepresentation}
There exists a solution $(\hat \rho, \hat m,\hat \mu) \in \widetilde{\M}_\Om$ to \eqref{eq:inverseproblem} which is of the form
\begin{equation}
(\hat \rho, \hat m,\hat \mu) = \sum_{i=1}^p c_i \, (\rho^i,m^i,\mu^i)\,,
\end{equation}
where $p\leq \dim(H)$, $c_i >0$, $\sum_{i=1}^p c_i = J_{\alpha,\beta,\delta}(\hat \rho,\hat m,\hat \mu)$ and  $(\rho^i,m^i,\mu^i) \in  \mathcal{C}$, with $\mathcal{C}$ is as in Definition \ref{def:extrfr}. %
\end{theorem}

In order to prove the above theorem, it is sufficient to apply Theorem \ref{thm:fr} and check validity for the assumptions of Corollary $2$ in \cite{chambolle}. The proof is a straightforward adaptation of the one of Theorem $10$ in \cite{bcfr} (which deals with the case $\delta=\infty$) and is hence omitted.

We now present an application of Theorem \ref{thm:sparserepresentation} to dynamic inverse problems, in a simplified case of the framework introduced in \cite{bf}. 
To be more specific, let $t_1<\ldots<t_N$ be a finite discretization of the time interval $[0,1]$. The aim is to reconstruct an element of $\curves$ by only making observations at the time instants $t_i$. Hence let $H_{i}$ be a family of finite-dimensional Hilbert spaces and set $H=\bigtimes_{i=1}^N H_{i}$, normed by $\|y\|_{H}^2 := \sum_{i=1}^N \|y_i\|^2_{H_{i}}$. Let
$K_{i} : \M(\Om) \to H_{i}$ be linear and weak* continuous operators. 
For a given observation $y \in  H$, consider the problem of finding $\rho \in \curves$ such that
\[
K_{i} \rho_{t_i} = y_{i} \,\, \text{ for each } \,\,
i = 1 ,\ldots,N\,.
\] 
Following \cite{bf}, we regularize the above problem by
\begin{equation}\label{eq:mri}
\min_{(\rho,m,\mu)\in \widetilde{\mathcal{M}}_\Om} \,  \frac{1}{2}\sum_{i=1}^N \|K_{i}\rho_{t_i} - y_{i}\|^2_{H_{i}} + J_{\alpha,\beta,\delta}(\rho,m,\mu)  \,.
\end{equation}
To recast \eqref{eq:mri} into the form \eqref{eq:inverseproblem}, define the linear operator $K : \curves \rightarrow H$ as $K\rho := (K_{1}\rho_{t_1}, \ldots,K_{N}\rho_{t_N})$ and note that $K$ is continuous in the sense of \eqref{eq:topologycweak}. Moreover define the fidelity term $F \colon H \to \R$ by $F(x) := \frac{1}{2}\|x - y\|^2_{H}$, which is convex, lower semi-continuous and bounded from below. In this way \eqref{eq:mri} is a particular case of \eqref{eq:inverseproblem} and Theorem \ref{thm:sparserepresentation} applies, thus showing the existence and characterizing the structure of sparse solutions to the discrete reconstruction problem regularized via the Hellinger-Kantorovich energy.

\section*{Acknowledgements}

KB and SF are supported by the Christian Doppler Research Association (CDG) and Austrian Science Fund (FWF) through
project PIR-27 ``Mathematical methods for motion-aware medical imaging'' and project P 29192 ``Regularization graphs for variational imaging''.
MC is supported by the Royal Society (Newton International Fellowship NIF\textbackslash R1\textbackslash 192048). The Institute of Mathematics and Scientific Computing, to which KB and SF are affiliated, is a member of NAWI Graz (\texttt{http://www.nawigraz.at/en/}). The authors  KB and SF are members of/associated with BioTechMed Graz (\texttt{https://biotechmedgraz.at/en/}).

\bibliography{bibliography}

\begin{thebibliography}{10}

\bibitem{agueh}
M.~Agueh, N.~Ghoussoub, and X.~Kang.
\newblock Geometric inequalities via a general comparison principle for
  interacting gases.
\newblock {\em Geometric and Functional Analysis}, 14:215--244, 2004.

\bibitem{afp}
L.~Ambrosio, N.~Fusco, and D.~Pallara.
\newblock {\em {F}unctions of bounded variation and free discontinuity
  problems}.
\newblock Oxford Science Publications, 2000.

\bibitem{ags}
L.~Ambrosio, N.~Gigli, and G.~Savar{\'e}.
\newblock {\em Gradient flows: {I}n metric spaces and in the space of
  probability measures}.
\newblock Birkh{\"a}user Basel, 2005.

\bibitem{ambrosioinventiones}
L.~Ambrosio.
\newblock Transport equation and {C}auchy problem for {BV} vector fields.
\newblock {\em Inventiones mathematicae}, 158(2):227--260, 2004.

\bibitem{acfarma}
L.~Ambrosio, M.~Colombo, and A.~Figalli.
\newblock Existence and uniqueness of maximal regular flows for non-smooth
  vector fields.
\newblock {\em Archive for Rational Mechanics and Analysis}, 218:1043--1081,
  2015.

\bibitem{ambrosiocrippalecturenotes}
L.~Ambrosio and G.~Crippa.
\newblock {\em Existence, uniqueness, stability and differentiability
  properties of the flow associated to weakly differentiable vector fields},
  pages 3--57.
\newblock Springer Berlin Heidelberg, 2008.

\bibitem{agsduke}
L.~Ambrosio, N.~Gigli, and G.~Savar{\'e}.
\newblock Metric measure spaces with {R}iemannian {R}icci curvature bounded
  from below.
\newblock {\em Duke Mathematical Journal}, 163(7), 2011.

\bibitem{agsinventiones}
L.~Ambrosio, N.~Gigli, and G.~Savar{\'e}.
\newblock Calculus and heat flow in metric measure spaces and applications to
  spaces with {R}icci bounds from below.
\newblock {\em Inventiones mathematicae}, 195(2):289--391, 2014.

\bibitem{bb}
J.-D. Benamou and Y.~Brenier.
\newblock A computational fluid mechanics solution to the {M}onge-{K}antorovich
  mass transfer problem.
\newblock {\em Numerische Mathematik}, 84(3):375--393, 2000.

\bibitem{bernard1}
P.~Bernard.
\newblock Young measures, superposition and transport.
\newblock {\em Indiana University Mathematics Journal}, 57(1):247--275, 2008.

\bibitem{bernard2}
P.~Bernard and L.~Ambrosio.
\newblock Uniqueness of signed measures solving the continuity equation for
  {O}sgood vector fields.
\newblock {\em Atti della Accademia Nazionale dei Lincei, Classe di Scienze
  Fisiche, Matematiche e Naturali, Rendiconti Lincei Matematica E
  Applicazioni}, 19(3), 2008.

\bibitem{bianchinibonicattoinventiones}
S.~Bianchini and P.~Bonicatto.
\newblock A uniqueness result for the decomposition of vector fields in
  $\mathbb{R}^d$.
\newblock {\em Inventiones mathematicae}, 220(1):255--393, 2020.

\bibitem{bianchiniSIAM}
S.~Bianchini, P.~Bonicatto, and N.~A. Gusev.
\newblock Renormalization for autonomous nearly incompressible {BV} vector
  fields in two dimensions.
\newblock {\em SIAM Journal on Mathematical Analysis}, 48(1):1--33, 2016.

\bibitem{bogachev}
V.~I. Bogachev.
\newblock {\em Measure theory}.
\newblock Springer-Verlag Berlin Heidelberg, 2007.

\bibitem{bonicattogusev2018}
P.~Bonicatto and N.~A. Gusev.
\newblock Superposition principle for the continuity equation in a bounded
  domain.
\newblock {\em Journal of Physics: Conference Series}, 990:012002, 2018.

\bibitem{chambolle}
C.~Boyer, A.~Chambolle, Y.~D. Castro, V.~Duval, F.~de~Gournay, and P.~Weiss.
\newblock On representer theorems and convex regularization.
\newblock {\em SIAM Journal on Optimization}, 29(2):1260--1281, 2019.

\bibitem{bcfr2}
K.~Bredies, M.~Carioni, S.~Fanzon, and F.~Romero.
\newblock A generalized conditional gradient method for dynamic inverse
  problems with optimal transport regularization.
\newblock {\em Foundations of Computational Mathematics}, 2022.
\newblock https://doi.org/10.1007/s10208-022-09561-z.

\bibitem{inprep2}
K.~Bredies, M.~Carioni, S.~Fanzon, and D.~Walter.
\newblock {L}inear convergence of {A}ccelerated {G}eneralized {C}onditional
  {G}radient {M}ethods.
\newblock {\em arXiv e-prints},  arXiv:2110.06756, 2021.

\bibitem{bc}
K.~Bredies and M.~Carioni.
\newblock Sparsity of solutions for variational inverse problems with
  finite-dimensional data.
\newblock {\em Calculus of Variations and Partial Differential Equations},
  59(1):14, 2020.

\bibitem{bcfr}
K.~Bredies, M.~Carioni, S.~Fanzon, and F.~Romero.
\newblock On the extremal points of the ball of the {Benamou--Brenier} energy.
\newblock {\em Bulletin of the London Mathematical Society}, 53(5):1436--1452,
  2021.

\bibitem{bf}
K.~Bredies and S.~Fanzon.
\newblock An optimal transport approach for solving dynamic inverse problems in
  spaces of measures.
\newblock {\em ESAIM: Mathematical Modelling and Numerical Analysis},
  54(6):2351--2382, 2020.

\bibitem{chizat}
L.~Chizat, G.~Peyr{\'e}, B.~Schmitzer, and F.-X. Vialard.
\newblock An interpolating distance between optimal transport and
  {F}isher--{R}ao metrics.
\newblock {\em Foundations of Computational Mathematics}, 18(1):1--44, 2018.

\bibitem{chizat3}
L.~Chizat, G.~Peyr{\'e}, B.~Schmitzer, and F.-X. Vialard.
\newblock Scaling algorithms for unbalanced transport problems.
\newblock {\em Mathematics of Computation}, 87(314):2563--2609, 2016.

\bibitem{cordero}
D.~Cordero-Erausquin, B.~Nazaret, and C.~Villani.
\newblock A mass-transportation approach to sharp {S}obolev and
  {G}agliardo--{N}irenberg inequalities.
\newblock {\em Advances in Mathematics}, 182(2):307--332, 2004.

\bibitem{dipernalions}
R.~J. DiPerna and P.~L. Lions.
\newblock Ordinary differential equations, transport theory and {S}obolev
  spaces.
\newblock {\em Inventiones mathematicae}, 98(3):511--547, 1989.

\bibitem{evansgariepy}
L.~C. Evans and R.~F. Gariepy.
\newblock {\em Measure theory and fine properties of functions}.
\newblock CRC Press, 2015.

\bibitem{ottokinderleher}
R.~Jordan, D.~Kinderlehrer, and F.~Otto.
\newblock Free energy and the {F}okker--{P}lanck equation.
\newblock {\em Physica D: Nonlinear Phenomena}, 107(2):265--271, 1997.

\bibitem{ottokinderleher2}
R.~Jordan, D.~Kinderlehrer, and F.~Otto.
\newblock The variational formulation of the {F}okker--{P}lanck equation.
\newblock {\em SIAM Journal on Mathematical Analysis}, 29(1):1--17, 1998.

\bibitem{kantorovich}
L.~Kantorovitch.
\newblock On the translocation of masses.
\newblock {\em Comptes Rendus (Doklady) de l'Acad{\'e}mie des Sciences de
  l'URSS}, 37:199--201, 1942.

\bibitem{kechris}
A.~S. Kechris.
\newblock {\em Classical {D}escriptive {S}et {T}heory}.
\newblock Springer-Verlag New York, 1995.

\bibitem{kmv}
S.~Kondratyev, L.~Monsaingeon, and D.~Vorotnikov.
\newblock A new optimal transport distance on the space of finite {R}adon
  measures.
\newblock {\em Advances in Differential Equations}, 21(11/12):1117--1164, 2016.

\bibitem{liero2}
M.~Liero, A.~Mielke, and G.~Savar{\'e}.
\newblock {O}ptimal transport in competition with reaction: {T}he
  {H}ellinger--{K}antorovich distance and geodesic curves.
\newblock {\em SIAM Journal on Mathematical Analysis}, 48(4):2869--2911, 2016.

\bibitem{liero}
M.~Liero, A.~Mielke, and G.~Savar{\'e}.
\newblock Optimal {E}ntropy-{T}ransport problems and a new
  {H}ellinger-{K}antorovich distance between positive measures.
\newblock {\em Inventiones mathematicae}, 211:969--1117, 2018.

\bibitem{lisini}
S.~Lisini.
\newblock Characterization of absolutely continuous curves in {W}asserstein
  spaces.
\newblock {\em Calculus of Variations and Partial Differential Equations},
  28(1):85--120, 2007.

\bibitem{maggi2}
F.~Maggi and C.~Villani.
\newblock Balls have the worst best {S}obolev inequalities. {P}art {II}:
  variants and extensions.
\newblock {\em Calculus of Variations and Partial Differential Equations},
  31(1):47--74, 2008.

\bibitem{maggi}
F.~Maggi and C.~Villani.
\newblock Balls have the worst best {S}obolev inequalities.
\newblock {\em The Journal of Geometric Analysis}, 15(1):83--121, 2005.

\bibitem{maniglia}
S.~Maniglia.
\newblock Probabilistic representation and uniqueness results for
  measure-valued solutions of transport equations.
\newblock {\em Journal de Math{\'e}matiques Pures et Appliqu{\'e}es},
  87(6):601--626, 2007.

\bibitem{nazaret}
B.~Nazaret.
\newblock Best constant in {S}obolev trace inequalities on the half-space.
\newblock {\em Nonlinear Analysis: Theory, Methods \& Applications},
  65(10):1977--1985, 2006.

\bibitem{otto_villani}
F.~Otto and C.~Villani.
\newblock Generalization of an inequality by {T}alagrand and links with the
  logarithmic {S}obolev inequality.
\newblock {\em Journal of Functional Analysis}, 173(2):361--400, 2000.

\bibitem{otto2}
F.~Otto.
\newblock Dynamics of labyrinthine pattern formation in magnetic fluids: {A}
  mean-field theory.
\newblock {\em Archive for Rational Mechanics and Analysis}, 141(1):63--103,
  1998.

\bibitem{otto1}
F.~Otto.
\newblock The geometry of dissipative evolution equations: the porus medium
  equation.
\newblock {\em Communications in Partial Differential Equations},
  26(1-2):101--174, 2001.

\bibitem{sant}
F.~Santambrogio.
\newblock {\em Optimal transport for applied mathematicians}.
\newblock Birkh{\"a}user {B}asel, 2015.

\bibitem{schmitzerwirth2019}
B.~Schmitzer and B.~Wirth.
\newblock Dynamic models of {W}asserstein-1-type unbalanced transport.
\newblock {\em ESAIM: Control, Optimisation and Calculus of Variations}, 25:23,
  2019.

\bibitem{3superposition}
E.~Stepanov and D.~Trevisan.
\newblock {T}hree superposition principles: {C}urrents, continuity equations
  and curves of measures.
\newblock {\em Journal of Functional Analysis}, 272(3):1044--1103, 2017.

\bibitem{unser2}
M.~Unser.
\newblock A unifying representer theorem for inverse problems and machine
  learning.
\newblock {\em Foundations of Computational Mathematics}, 21:941--960, 2021.

\bibitem{unsersplines}
M.~Unser, J.~Fageot, and J.~P. Ward.
\newblock Splines are universal solutions of linear inverse problems with
  generalized {TV} regularization.
\newblock {\em SIAM Review}, 59(4):769--793, 2017.

\bibitem{villani}
C.~Villani.
\newblock {\em Optimal transport: {O}ld and new}.
\newblock Springer Berlin Heidelberg, 2008.

\end{thebibliography}

\bibliographystyle{my_plain}

\appendix
\section{}

\subsection{Properties of narrow convergence}

We give some results about narrow convergence of measures. These results are classical and are stated for probability measures in the literature: here we adapt them to positive measures. For a complete separable metric space $Y$, we say that a family of measures $\mathcal{A} \subset \M (Y)$ is {\it tight} if for every $\e>0$ there exists a compact set $K_\e \subset Y$ such that $|\mu|(Y \smallsetminus K_\e)<\e$ for all $\mu \in \mathcal{A}$. 
The next proposition provides a tightness criterion for positive measures. The proof follows as in \cite[Remark 5.1.5]{ags}, and is hence omitted.

\begin{proposition} \label{prop:tight}
Let $Y$ be a complete separable metric space and $\mathcal{A} \subset \mathcal{M}^+(Y)$. Suppose that there exists a measurable function $\mathscr{F}: Y \rightarrow [0,\infty]$ such that $\{y \in Y \, \colon \, \mathscr{F}(y) \leq c\}$ is compact for each $c \geq 0$ and 
$
\sup_{\mu \in \mathcal{A}} \int_Y \mathscr{F}(y) \, d\mu(y) < \infty$. Then $\mathcal{A}$ is \emph{tight}. 	
\end{proposition}

Finally, we provide a result which clarifies the behaviour of narrowly convergent sequences of positive measures when tested against lower semi-continuous, or continuous unbounded integrands. The proof easily follows by combining \cite[Lemma 5.1.7]{ags} with a scaling argument. 

\begin{proposition}\label{prop:narrow_unbounded}
Let $Y$ be a complete separable metric space. Assume that $\{\mu_n\}_n$, $\mu$ belong to $\M^+(Y)$ and $\mu_n \to \mu$ narrowly as $n \to \infty$. If $g \colon Y \to [0,\infty]$ is lower semi-continuous then 
\begin{equation} \label{prop:narrow_unbounded:1}
\int_Y g(y) \, d\mu(y) \leq \liminf_{n \to \infty} \int_Y g(y)\, d \mu_n(y) \,,
\end{equation}
If $f \colon Y \to \R$ is continuous with $|f|$ uniformly integrable with respect to $\{\mu_n\}_n$, that is,
\begin{equation} \label{def:UI}
\lim_{k \to \infty} \, \sup_{n \in \N} \int_{ \{y \in Y \colon  |f(y)| \geq k\}} |f(y)| \, d\mu_n(y) = 0 \,,
\end{equation}
then it holds
\begin{equation} \label{prop:narrow_unbounded:2}
\lim_{n \to \infty} \int_Y f(y) \, d\mu_{n}(y) = \int_Y f(y)\, d\mu(y) \,.
\end{equation}
\end{proposition}

\subsection{Disintegration of measures} \label{app:disintegration}
In this section we state and prove the disintegration theorem employed in Section \ref{sec:uniqueness}. This result is a straightforward consequence of \cite[Theorem 5.3.1]{ags}.
\begin{theorem}\label{thm:app:disintegration}
Let $Z,X$ be Radon separable metric spaces and let $\mu \in \mathcal{M}^+(Z)$, $\nu \in \mathcal{M}^+(X)$ be given. Let $\pi : Z \rightarrow X$ a measurable map such that $\pi_\# \mu \ll \nu$. Then there exists a Borel family of measures $\{\mu^x\}_{x \in X} \subset  \mathcal{M}^+(Z)$ such that
\begin{itemize}
\item [i)]$\mu^x(Z\smallsetminus \pi^{-1}(x)) = 0$ for $\nu$-a.e.~$x \in X$,
\item [ii)] for every function $f \in L^1_\mu(Z)$ there holds 
 \begin{equation} \label{app100}
\int_Z f(z)\, d\mu(z) = \int_X \int_Z f(z)\, d\mu^x(z)\, d\nu(x)\,,
\end{equation}
\item [iii)] if $\mu$ is concentrated on $E \subset Z$, then $\mu^x$ is concentrated on $\pi^{-1}(x) \cap E$ for $\nu$-a.e. $x \in X$.
\end{itemize}
\end{theorem}

\begin{proof}
Without loss of generality we can suppose that $\mu \neq 0$. By a rescaling argument we  can assume that $\|\mu\|_{\mathcal{M}(Z)} = 1$ as well.
Thanks to \cite[Theorem 5.3.1]{ags} there exists a Borel family of measures $\{\tilde\mu^x\}_{x \in X} \subset  \mathcal{M}^+(Z)$ such that
$\tilde \mu^x(Z \smallsetminus \pi^{-1}(x)) = 0$ for $(\pi_\#\mu)$-a.e.~$x \in X$, and that \eqref{app100} holds with $\mu^x$ and $\nu$ replaced by $\tilde{\mu}^x$ and $\pi_\#\mu$, respectively, 
for every Borel function $f : Z \rightarrow [0,\infty]$.
For all $x \in X$ set
$\mu^x := \frac{\partial(\pi_\# \mu)}{\partial \nu}(x)\,  \tilde \mu^x$. 
We immediately obtain that $\mu^x \in \mathcal{M}^+(Z)$ is a family of Borel measures satisfying $(i)$. Moreover, it is easy to check that $\mu^x$ satisfies \eqref{app100} for every Borel function $f : Z \rightarrow [0,\infty]$. 
If $f  \in L^1_\mu(Z)$, by \eqref{app100} we get $f \in  L^1_{\mu^x}(Z)$ for $\nu$-a.e.~$x\in X$, yielding $(ii)$. Finally, $(iii)$ is implied by $(ii)$.
\end{proof}

\subsection{Properties of $B_\delta$ and $J_{\alpha,\beta,\delta}$}

In this section we gather some of the properties of the functionals $B_\delta$ and $J_{\alpha,\beta,\delta}$ introduced in Section \ref{sec:benamou}. The interested reader can find the proofs of such results in Proposition 2.6 and Lemmas 4.5, 4.6 in \cite{bf}.

\begin{lemma}[Properties of $B_\delta$] \label{lem:prop B}
	The functional $B_\delta$ defined in \eqref{bb} is non-negative, convex, one-homogeneous and sequentially lower semi-continuous with respect to the weak* topology on $\M_\Om$. Moreover it satisfies the following properties:
	\begin{enumerate}
	\item [i)] if $B_\delta(\rho,m,\mu)<\infty$, then $\rho \geq 0$ and $m,\mu \ll \rho$, that is, there exist measurable maps $v \colon X_\Om \to \R^d$, $g \colon X_\Om \to \R$ such that $m=v\rho$, $\mu=g \rho$,
	\item [ii)] if $\rho \geq 0$ and $m = v \rho,\mu=g \rho$ for some  measurable $v \colon X_\Om \to \R^d$, $g \colon X_\Om \to \R$, then
\begin{equation} \label{formula B}
B_\delta(\rho,m,\mu) =\int_{X_\Om} \Psi_\delta (1,v,g) \, d\rho =  \frac12 \int_{X_\Om} \left(|v|^2 +\delta^2 g^2 \right) \, d\rho \,. 
\end{equation}
	\end{enumerate}
\end{lemma}

\begin{lemma}[Properties of $J_{\alpha,\beta,\delta}$] \label{lem:prop J}
Let $\alpha, \beta,\delta >0$. The functional $J_{\alpha,\beta,\delta}$ is non-negative, convex, one-homogeneous and sequentially lower semi-continuous with respect to the weak* topology on $\M_\Om$.  
		For $(\rho,m,\mu) \in \M_\Om$ such that $J_{\alpha,\beta,\delta}(\rho,m,\mu)< \infty$ we have that 
	\begin{equation} \label{lem:prop J est}
         \max \{ \alpha \norm{\rho}_{\M(X_\Om)}, \,  C \norm{m}_{\M(X_\Om;\R^d)} ,\, C \norm{\mu}_{\M(X_\Om)} \} \leq
          J_{\alpha,\beta,\delta}(\rho,m,\mu)
	\end{equation}	
	where $C:=\min\{ 2\alpha,\beta \min \{1, \delta^2\}\}$.
	If in addition the sequence  $\{(\rho^n,m^n,\mu^n)\}_n$ in $\M_\Om$ is such that $\sup_n J_{\alpha,\beta,\delta} (\rho^n,m^n,\mu^n)<\infty$, 
then $\rho^n = dt \otimes \rho_t^n$ for some $(t \mapsto \rho_t^n) \in \pcurves$ and there exists 
	$(\rho,m,\mu)$ in $\Distr_\Om$ with $\rho= dt \otimes \rho_t$, $(t \mapsto\rho_t) \in \pcurves$, such that, up to subsequences,   
	$(\rho^n,m^n,\mu^n) \weakstar (\rho,m,\mu)$ weakly* in $\M_\Om$ and 
	$\rho_t^n \to \rho_t$ narrowly in $\M(\Om)$ for every $t \in [0,1]$.
\end{lemma}

\subsection{Proof of Proposition \ref{prop:complete}} \label{app:complete}
 
Remember that $\car_V = C([0,1];\cone_V)$ by Proposition \ref{prop:narrow}. Therefore, in order to prove that $(\car_V,\D)$ is complete and separable, it is sufficient to show that $(\cone_V,\D_F)$ is complete and separable (see Theorem 4.19 in \cite{kechris}). Let us first prove that $(\cone_V,\D_F)$ is complete. Hence, let $\rho^n=h_n \delta_{\gamma_n} \in \cone_V$ be a Cauchy sequence. 
By \eqref{eq:char_flat} we have 
$|h_n - h_m| \leq \D_F(\rho^n,\rho^m)$ for all $m,n \in \N$. Therefore $h_n \to h$ for some $h \geq 0$. If $h=0$, by \eqref{eq:char_flat} we have $\D_F(\rho^n,0 )=h_n \to 0$, showing that $\rho^n$ converges to $0 \in \cone_V$. Assume now that $h>0$. Notice that $|\gamma_n - \gamma_m|\leq 2$ for sufficiently large $m,n$, otherwise we could extract a subsequence (not relabelled) such that $\D_F(\rho^n,\rho^m) = h_n + h_m \to 2h >0$ as $m,n \to \infty$, which contradicts $\rho^n$ being Cauchy. By \eqref{eq:char_flat} and the facts that $h_n \to h >0$ and that $\rho^n$ is Cauchy, we get that $\gamma_n$ is Cauchy, so that $\gamma_n \to\gamma \in V$. An application of \eqref{eq:char_flat} shows that $\rho^n \to \rho:=h \delta_\gamma$ with respect to $\D_F$, concluding completeness. The fact that $(\car_V,\D)$ is separable is immediate: indeed the countable set $\cone_V' :=\left\{ h \delta_{\gamma}  \, \colon \,  h \in [0,\infty)\cap\Q,\, \gamma \in V \cap \Q^d \right\} \subset \cone_V$ is $\D_F$-dense in $\cone_V$, since $V$ is the closure of a domain.

\subsection{Comparison principle}

In this section we recall a comparison principle for signed measure solutions of the continuity equation. 

\begin{proposition}[Comparison principle] \label{prop:comparison}
	Let  $\rho_t \colon [0,1]  \to \M(\R^d)$ be narrowly continuous and $v \colon (0,1) \times \R^d \to \R^d$, $g \colon (0,1) \times \R^d \to \R$ be measurable. Suppose that $\de_t \rho_t + \div (v \rho_t) = g \rho_t$ holds in $(0,1) \times \R^d$ in the sense of \eqref{cont weak}. Assume that $\rho_0 \leq 0$, as well as \eqref{eq:comp1}, \eqref{eq:comp2} and 
	\begin{equation}
  \int_0^1 \int_{\R^d} \left( |v(t,x)|+|g(t,x)| \right) \, d|\rho_t|(x) \, dt < \infty\,. \label{eq:comp222} 	
	\end{equation}
Then $\rho_t \leq 0$ for all $t \in [0,1]$. 
\end{proposition}

A proof of the above proposition can be found in   \cite[Lemma 3.5]{maniglia}. We just point out that in \cite[Lemma 3.5]{maniglia} it is assumed that the narrowly continuous curve $t \mapsto \rho_t \in \M(\R^d)$ satisfies $\int_0^1 |\rho_t|(B) \, dt < \infty $ for all $B \subset \R^d$ compact. However this condition is always fulfilled, since $\rho_t$ automatically satisfies $\sup_{t \in [0,1]}\norm{\rho_t}_{\M(\R^d)}< \infty$, as shown in \cite[Proposition A.3]{bf}. Moreover the statement of \cite[Lemma 3.5]{maniglia} also requires that $g$ is bounded: after carefully inspecting the proof, we noticed that such assumption is not needed.

\subsection{Property of convolutions}

Here we recall a result on convolution of measures, which can be found in \cite[Lemma 3.9]{maniglia}.

\begin{proposition} \label{prop:smoothmeasure}
Let $p \geq 1$, $\rho \in \M^+(\R^d)$, $E \in \M(\R^d,\R^m)$ and $\xi$ be a convolution kernel on $\R^d$. Suppose that $E$ is absolutely continuous with respect to $\rho$. Then,
	\[
	\int_{\R^d} \left|  \frac{E \ast \xi}{\rho \ast \xi}  \right|^p \, (\rho \ast \xi) \, dx  \leq 
	\int_{\R^d} \left|  \frac{dE }{d\rho }  \right|^p \,  \, d\rho\,,
	\]
	where $ dE /d\rho$ is the Radon-Nikodym derivative of $E$ with respect to $\rho$.
\end{proposition}

\end{document}